\documentclass[11pt,reqno]{amsart}

\usepackage{amsmath}
\usepackage{amsthm}
\usepackage{mathrsfs}
\usepackage{graphicx}
\usepackage{mathtools}
\usepackage{amsfonts}
\usepackage{amssymb}
\usepackage{hyperref}
\usepackage{bm}
\usepackage[margin=1 in]{geometry}
\usepackage{enumerate}
\usepackage[normalem]{ulem}
\usepackage{tikz}
\usepackage{tikz-cd}
\usepackage{xcolor}
\usetikzlibrary{matrix,arrows,positioning,automata}
  \usepackage{cancel}
\usepackage{comment}

  \usepackage{scalerel,stackengine}
\stackMath
\newcommand\reallywidehat[1]{%
\savestack{\tmpbox}{\stretchto{%
  \scaleto{%
    \scalerel*[\widthof{\ensuremath{#1}}]{\kern-.6pt\bigwedge\kern-.6pt}%
    {\rule[-\textheight/2]{1ex}{\textheight}}
  }{\textheight}%
}{0.5ex}}%
\stackon[1pt]{#1}{\tmpbox}%
}

\newcommand{\N}{\mathbb N}
\newcommand{\Z}{\mathbb Z}

\newcommand{\R}{\mathbb R}
\def\E{\mathbb E}

\newcommand{\linf}{L^{\infty}}
\newcommand{\sF}{\mathcal{F}}
\newcommand{\bP}{\mathbb{P}}
\newcommand{\sL}{\mathcal{L}}

\newcommand{\sA}{\mathcal{A}}
\newcommand{\sP}{\mathcal{P}}

\def\XXint#1#2#3{{\setbox0=\hbox{$#1{#2#3}{\int}$}
\vcenter{\hbox{$#2#3$}}\kern-.5\wd0}}

\newcommand{\T}{\mathbb{T}}

\numberwithin{equation}{section}
\newtheorem{thm}{Theorem}[section]
\newtheorem{lem}[thm]{Lemma}
\newtheorem{cor}[thm]{Corollary}
\newtheorem{prop}[thm]{Proposition}
\newtheorem{assumption}[thm]{Assumption}
\theoremstyle{definition}

\newtheorem{rmk}[thm]{Remark}

\newtheorem{convention}[thm]{Convention}

\def\smallnegint{\mathop{\int\mkern-13mu
        \raise.5ex\hbox{${\scriptscriptstyle\diagup}$}}\nolimits}
\def\ds{\displaystyle}

\def\div{\operatorname{div}}
\def\tr{\operatorname{tr}}

\def\bx{{\bm x}}
\def\by{{\bm y}}
\newcommand{\bX}{\bm{X}}

\def\ssetminus{\,\raise.4ex\hbox{$\scriptstyle\setminus$}\,}

\newcommand{\be}{\begin{equation}}
\newcommand{\ee}{\end{equation}}

\newcommand{\bc}{\begin{case}}
\newcommand{\ec}{\end{cases}}
\newcommand{\bs}{\begin{split}}
\newcommand{\es}{\end{split}}

\newcommand{\norm}[1]{\left\Vert#1\right\Vert}

\renewcommand{\bar}{\overline}
\renewcommand{\tilde}{\widetilde}
\renewcommand{\hat}{\widehat}

\def \be {\begin{equation}}
\def \ee {\end{equation}}

\def \E {\mathbb{E}}

\def \R {\mathbb{R}}

\renewcommand{\tilde}{\widetilde}

\newcommand{\cC}{\mathcal{C}}

\newcommand{\leb}{\mathrm{Leb}}
\newcommand{\bbF}{\mathbb{F}}

\newcommand{\trip}[1]{{\left\vert\kern-0.25ex\left\vert\kern-0.25ex\left\vert #1 
    \right\vert\kern-0.25ex\right\vert\kern-0.25ex\right\vert}}

\newcommand{\cP}{\mathcal{P}}
\newcommand{\lip}{C_{\text{Lip}}}
\newcommand{\semi}{C_{\text{S}}}
\newcommand{\tv}{\text{TV}}
\newcommand{\ov}{\overline}

\newcommand{\bd}{\bm{d}}
\newcommand{\cL}{\mathcal{L}}
\newcommand{\bW}{\mathbb{W}}
\newcommand{\com}{\Delta_{\text{com}}}

\newcommand{\eps}{\epsilon}

\title[Quantitative convergence with common noise and degenerate idiosyncratic noise]{Quantitative convergence for mean field control with common noise and degenerate idiosyncratic noise}

\begin{document}

\author[A. Cecchin]{Alekos Cecchin
\address{(A. Cecchin) Universit\`a di Padova, Dipartimento di Matematica ``Tullio Levi-Civita'', 35121 Padova, Italy}\email{alekos.cecchin@unipd.it}}

 \author[S. Daudin]{Samuel Daudin 
\address{(S. Daudin) Universit\'e Paris Cité, Laboratoire Jacques-Louis-Lions, Paris, France
}\email{samuel.daudin@u-paris.fr
}}

\author[J. Jackson]{Joe Jackson
\address{(J. Jackson) Department of Mathematics, The University of Chicago, Chicago, Illinois 60637 USA}\email{jsjackson@uchicago.edu}}

\author[M. Martini]{Mattia Martini
\address{(M. Martini) Universit\'e C\^ote d'Azur, Laboratoire J. A. Dieudonné, 06108 Nice, France
}\email{mattia.martini@univ-cotedazur.fr
}}

\begin{abstract}
We consider the convergence problem in the setting of mean field control with common noise and degenerate idiosyncratic noise. Our main results establish a rate of convergence of the finite-dimensional value functions $V^N$ towards the mean field value function $U$. In the case that the idiosyncratic noise is constant (but possibly degenerate), we obtain the rate $N^{-1/(d+7)}$, which is close to the conjectured optimal rate $N^{-1/d}$, and improves on the existing literature even in the non-degenerate setting. In the case that the idiosyncratic noise can be both non-constant and degenerate, the argument is more complicated, and we instead find the rate $N^{-1/(3d + 19)}$. Our proof strategy builds on the one initiated in \cite{ddj2023} in the case of non-degenerate idiosyncratic noise and zero common noise, which consists of approximating $U$ by more regular functions which are almost subsolutions of the infinite-dimensional Hamilton-Jacobi equation solved by $U$. Because of the different noise structure, several new steps are necessary in order to produce an appropriate mollification scheme. In addition to our main convergence results, we investigate the case of zero idiosyncratic noise, and show that sharper results can be obtained there by purely control-theoretic arguments. We also provide examples to demonstrate that the value function is sensitive to the choice of admissible controls in the zero noise setting. 
\end{abstract}

\setcounter{tocdepth}{1}

\thanks{
A. Cecchin is partially supported by INdAM-GNAMPA Project 2023 ``Stochastic
Mean Field Models: Analysis and Applications'', the PRIN 2022 Project 2022BEMMLZ ``Stochastic control and games and the role of information'', the PRIN 2022 PNRR Project P20224TM7Z  ``Probabilistic methods for energy transition'', and the project MeCoGa ``Mean field control and games'' of the University of Padova through the program STARS@UNIPD - NextGenerationEU. 
S. Daudin and M. Martini 
acknowledge the financial support of the European Research Council (ERC) under the European Union’s Horizon Europe research and innovation program (AdG
ELISA project, Grant Agreement No. 101054746). Views and opinions expressed are however
those of the author(s) only and do not necessarily reflect those of the European Union or the
European Research Council Executive Agency. Neither the European Union nor the granting
authority can be held responsible for them.
 J. Jackson is supported by the NSF under Grant No. DMS2302703. Any opinions, findings and conclusions or recommendations expressed in this material are those of the authors and do not necessarily reflect the views of the NSF}

\maketitle

\tableofcontents


\section{Introduction} \label{sec: intro} 

\subsection{The $N$-particle and mean field control problems} Mean field control theory is concerned with the large population limits of certain optimal control problems. In the setting addressed in the present paper, we have, for each $N \in \N$, a stochastic control problem set on the state space $(\T^d)^N$, where $\T^d = \R^d/ \Z^d$ is the $d$-dimensional flat torus. This problem is efficiently described by a value function $V^N : [0,T] \times (\T^d)^N \to \R$, which is defined informally as follows:
\begin{align} \label{vnintrodef}
    V^N(t_0,\bx_0) = \inf_{\bm \alpha} \E\bigg[\int_{t_0}^T \Big(\frac{1}{N} \sum_{i = 1}^N L(X_t^i, \alpha_t^i) + F(m_{\bX_t}^N) \Big)dt + G(m_{\bX_T}^N) \bigg], 
\end{align}
where we write $\bx_0 = (x_0^1,...,x_0^N)$, and the $(\T^d)^N$-valued state process $\bm X = (X^1,...,X^N)$ is determined from the control $\bm \alpha = (\alpha^1,...,\alpha^N)$ via the dynamics 
\begin{align*}
    dX_t^i = \alpha_t^i dt + \sigma(X_t^i) dW_t^i + \sigma^0 dW_t^0, \quad t_0 \leq t \leq T, \quad X^i_{t_0} = x_0^i.
\end{align*}
The problem is set on a filtered probability space hosting the independent $d$-dimensional Brownian motions $(W^i)_{i \in \N \cup \{0\}}$, and the controls $\alpha^i$ are adapted $\R^d$-valued processes satisfying appropriate integrability conditions. For $\bx = (x^1,...,x^N) \in (\T^d)^N$, we are using the notation $m_{\bx}^N = \frac{1}{N} \sum_{i = 1}^N \delta_{x^i} \in \cP(\T^d)$ for the empirical measure associated to $\bx$. The data for our problem consists of 
\begin{align} \label{introdata}
    L : \T^d \times \R^d \to \R, \quad F,G : \cP(\T^d) \to \R,  \quad \sigma : \T^d \to \R^{d \times d}, \quad \sigma^0 \geq 0.
\end{align}
The Brownian motions $(W^i)_{i = 1,...,N}$ are \textit{individual} or \textit{idiosyncratic} noises, while the Brownian motion $W^0$ is a \textit{common} noise which impacts all players. We can interpret this control problem by imagining that there is a team of $N$ agents. For $i = 1,...,N$, agent $i$ has state process $X^i$, which they can control by choosing $\alpha^i$, and the players cooperate to minimize the cost function appearing in \eqref{vnintrodef}.

We recall (see \cite{flemingsoner} for instance) that under mild assumptions on the data, $V^N$ is the unique viscosity solution of the Hamilton-Jacobi-Bellman equation 
\begin{align} \tag{$\text{HJB}_N$} \label{hjbn}
    \begin{cases} 
   \ds  - \partial_t V^N - \sum_{i = 1}^N \tr\big( A(x^i) D^2_{x^ix^i} V^N\big) - A_0 \sum_{i,j = 1}^N \tr(D^2_{x^ix^j} V^N) + \frac{1}{N} \sum_{i = 1}^N H(x^i, ND_{x^i} V^N) = F(m_{\bx}^N), \vspace{.1cm} \\
   \hspace{9cm} (t,\bx) \in [0,T) \times (\T^d)^N, \vspace{.1cm} \\
   \ds V^N(T,\bx) = G(m_{\bx}^N), \quad \bx \in (\T^d)^N,
    \end{cases}
\end{align}
where we have defined 
\begin{align} \label{hadef}
    H(x,p) = \sup_{a \in \R^d} \Big(- L(x,a) - a \cdot p \Big), \quad A(x) = \frac{1}{2} \sigma(x) \sigma^T(x), \quad A_0 = \frac{1}{2} |\sigma^0|^2.
\end{align}

As $N \to \infty$, we expect that these control problems should converge in some sense to the corresponding mean field control problem, which is described by a value function $U : [0,T] \times \cP(\T^d) \to \R$ defined informally by
\begin{align}
    U(t_0,m_0) = \inf_{\alpha} \E\bigg[\int_{t_0}^T \Big( L(X_t, \alpha_t) + F \big( \sL(X_t | W^0) \big) \Big) dt + G\big( \sL(X_T | W^0) \big) \bigg],
\end{align}
where the state process $X$ is determined form the control $\alpha$ via the dynamics 
\begin{align*}
    dX_t = \alpha_t dt + \sigma(X_t)dW_t + \sigma^0 dW_t^0, \quad X_{t_0} \sim m_0.
\end{align*}
Above, the $\R^d$-valued Brownian motions $W$ and $W^0$ are independent, 
the initial condition $X_{t_0} \sim m_0$ indicates that $X_{t_0}$ has law $m_0$ and is independent of $W^0$, and $\sL(X_t | W^0)$ denotes the law of $X_t$ conditional on the $\sigma$-algebra generated by $W^0$. 

By the dynamic programming principle (see e.g. \cite{djete2019mckean}), we expect that under appropriate conditions on the data, the function $U$ solves the infinite-dimensional Hamilton-Jacobi equation

\begin{align} \tag{$\text{HJB}_{\infty}$} \label{hjbinf}
    \begin{cases}
        \ds - \partial_t U - \int_{\T^d} \tr \big(A(x) D_{xm} U(t,m,x) \big)m(dx) - A_0 \com U  \vspace{.1cm} \\ \ds \qquad \qquad + \int_{\T^d} H\big(x, D_m U(t,m,x)\big) m(dx) = F(m), \quad (t,m) \in [0,T) \times \cP(\T^d), \vspace{.1cm} \\
        \ds U(T,m) = G(m), \quad m \in \cP(\T^d),
    \end{cases}
\end{align}
where $\com$ is the ``common noise operator"
\begin{align*}
    \com U(t,m) = \int_{\T^d} \tr\big(D_{xm} U(t,m,x) \big) m(dx) + \int_{\T^d}\int_{\T^d} \tr\big(D^2_{mm}U(t,m,x,y) \big)m(dx)m(dy),
\end{align*}
and $D_mU, D_{xm}U, D^2_{mm}U$ are appropriate derivatives of $U$ with respect to the measure variable (see Section \ref{sec: notation} for the precise definitions).

Without convexity conditions on $F$ and $G$, the value function $U$ may fail to be $\cC^1$, even if all the data is smooth. This is related to the non-uniqueness of optimal controls for the mean field control problem, and an example can be found in \cite{BrianiCardaliaguet} (see also \cite{daudinseeger} Section 2.4). Thus \eqref{hjbinf} must be interpreted in an appropriate viscosity sense. The comparison principle (and hence uniqueness of viscosity solutions) for the equation \eqref{hjbinf} has received significant attention in the literature recently, see e.g. \cite{GangboTudorascu, conforti, CossoGozziKharroubiPham, soner2022viscosity, Soner2023, daudinseeger, bayraktar2023, shao2023, talbitouzizhang, bertucci2023, touzizhangzhou, djs2023}. In the present paper, we shall not in fact use the uniqueness of viscosity solutions, but we will use the fact that $U$ is a viscosity subsolution of \eqref{hjbinf}.

\subsection{Existing results on the convergence problem}

There are a number of ways to formalize the convergence of the $N$-particle control problems towards their mean field counterpart. For example, one can consider the convergence of the optimal controls or of the optimally controlled state processes, in the spirit of propagation of chaos. In the present paper, we work under conditions that do not guarantee uniqueness of optimizers for the limiting problem, which makes convergence of the optimal controls or states quite subtle. We thus focus on the convergence of the value functions, i.e. we aim to show that $V^N(t,\bx) \approx U(t,m_{\bx}^N)$ when $N$ is large.

The qualitative convergence of $V^N$ to $U$ was first demonstrated via probabilistic compactness methods in \cite{budhiraja2012}, for a class of models with purely quadratic Hamiltonians arising from the study of large deviations of weakly interacting particle systems. This result was greatly generalized by Lacker in \cite{Lacker2017}, and then Lacker's result was further generalized in \cite{DjetePossamaiTan} and \cite{djete2022extended} to allow common noise and interaction through the controls, respectively. In \cite{GangboMayorgaSwiech, mayorgaswiech}, qualitative convergence of $V^N$ to $U$ is obtained for models with purely common noise via PDE methods, by leveraging a comparison principle for the limiting Hamilton-Jacobi equation \eqref{hjbinf}. We mention also \cite{CavagnariLisiniOrrieriSavare,FornasierLisiniOrrieriSavare} for a treatment of deterministic mean field control problems via $\Gamma$-convergence techniques, \cite{talbitouzizhang} for the convergence of the value functions of certain mean field optimal stopping problems, and \cite{daudin2023} for an extension of the probabilistic compactness methods from \cite{Lacker2017} to problems with state constraints.

The references in the previous paragraph all rely in some way on compactness arguments, and so do not provide a rate of convergence of $V^N$ to $U$. On the other hand, if $U$ is smooth enough, then an explicit computation shows that $U^N(t,\bx) := U(t,m_{\bx}^N) : [0,T] \times (\T^d)^N \to \R$ solves \eqref{hjbn} but with an additional error term on the right-hand side of the form
\begin{align} \label{errorterm}
    E^N(t,\bx) = - \frac{1}{N^2} \sum_{i = 1}^N \tr\big( A(x^i) D^2_{mm} U(t,m_{\bx}^N,x^i,x^i) \big).
\end{align}
In particular, by the comparison principle $|U - V^N| = O(1/N)$ if $D^2_{mm} U$ is bounded. This is in fact a simpler version of the argument introduced in \cite{CardaliaguetDelarueLasryLions} to show that a similar regularity estimate for the master equation implies convergence of the so-called Nash system, in the context of mean field games.

Unfortunately, smoothness of $U$ is only expected under certain restrictive convexity conditions, and establishing a rate of convergence in the non-convex regime has proved much more subtle. In \cite{CecchinFinite}, a sharp rate was provided for mean field control problems set on a finite state space, in which case the limiting value function $U$ is defined on a finite-dimensional simplex. Similar techniques were used in \cite{bayraktarcecchinchakrabory} to obtain rates of convergence with a continuous state space, but with costs and dynamics which depend on $m$ through finitely many of its moments. The first quantitative result for a general class of non-convex diffusion models came in \cite{cdjs2023}, where the estimate 
\begin{align} \label{cdjsest}
   |V^N(t,\bx) - U(t,m_{\bx}^N)| \leq CN^{-\beta_d}
\end{align} was established\footnote{In fact this is the transposition to the periodic setting of the main result of \cite{cdjs2023}.}, with $\beta_d$ a non-explicit constant depending only on the dimension $d$. This estimate was substantially sharpened in \cite{ddj2023}, where it was explained that the optimal rate of convergence depends on the metric with respect to which $U$ is Lipschitz. In particular, it was shown in \cite[Proposition 3.2]{ddj2023} that if $\sigma(x) = \sqrt{2} I_{d \times d}$ and $\sigma^0 = 0$, and if 
and if $H$, $F$, and $G$ are ``sufficiently smooth", then $U$ is Lipschitz and semiconcave with respect to the negative order Sobolev space $H^{-s}$ (see the notation section below), with $s > d/2 + 2$. In this introduction, we will informally use the term ``$H^{-s}$-regular case" to refer to the case that the data is ``smooth enough" to prove that $U$ is Lipschitz and semiconcave with respect to the Hilbert structure of $H^{-s}$ with $s > d/2 + 2$. Theorem 2.7 of \cite{ddj2023} states that the estimate
\begin{align} \label{ddjhminus}
   |V^N(t,\bx) - U(t,m_{\bx}^N)| \leq CN^{-1/2}
\end{align}
holds in the $H^{-s}$-regular case, and \cite[Example 2]{ddj2023} shows that this is optimal (at least if the degenerate Hamiltonian $H = 0$ is permitted). 

On the other hand, it was shown in \cite[Theorem 2.6]{ddj2023} that if the data is less smooth, so that $U$ is only Lipschitz (and semiconcave in an appropriate sense) with respect to the $1$-Wasserstein metric $\bd_1$, then for each $\eta > 0$ there is a constant $C(\eta)$ such that 
\begin{align} \label{ddjd1}
     |V^N(t,\bx) - U(t,m_{\bx}^N)| \leq C(\eta) N^{- \frac{2}{3d+ 6 + \eta}}.
\end{align}
We informally use the term ``$\bd_1$-regular case" to refer to the case that the data is smooth enough only to guarantee that $U$ is Lipschitz and semiconcave with respect to $\bd_1$. It was conjectured in \cite{ddj2023} that in fact the optimal rate in the $\bd_1$-regular case should be $N^{-1/d}$, and \cite[Example 1]{ddj2023} provides support for this conjecture. We mention also that while \cite{ddj2023} clarifies the optimal \textit{global} rate of convergence, the recent work \cite{cjms2023} shows that in fact the rate $1/N$ holds locally uniformly over an open and dense set $\mathcal{O}$ constructed in \cite{cardaliaguet-souganidis:2}, which means that the global convergence rate is in general different than the rate inside $\mathcal{O}$.

\subsection{Our contribution}

The results of \cite{ddj2023} indicate that the optimal rate of convergence is somewhere between $N^{-1/d}$ and $N^{-1/2}$, depending on the metric with respect to which $U$ is Lipschitz. However, the results of \cite{ddj2023} require a constant, non-degenerate idiosyncratic noise and zero common noise. Thus, in the present paper we aim to investigate the following question:
\begin{align*}
    \text{How is the rate of convergence impacted by the structure of the noise?}
\end{align*}
In particular, our goal is to obtain sharp estimates on $|V^N - U|$ in the case that $\sigma(x)$ is degenerate and $\sigma^0$ is non-zero. 

Let us first make an important point regarding the role of the idiosyncratic noise. As explained above, the rate $N^{-1/2}$ cannot be expected unless $U$ is Lipschitz with respect to a ``sufficiently weak" metric, as in the $H^{-s}$-regular case. On the other hand, if $A_0 = 0$, $F = 0$, and $ G(m) = \int g(x) m(dx)$ for some $g : \T^d \to \R$, then we have $U(t,m) = \int u(t,x)m(dx)$, where $u$ solves 
\begin{align} \label{degenhb}
    - \partial_t u - \tr\big( A(x) D^2 u) + H(x,Du) = 0, \quad (t,x) \in [0,T) \times \T^d, \quad u(T,x) = g(x), \quad x \in \T^d.
\end{align}
It is well known that if $\sigma$ is degenerate (for example when $\sigma = 0$), the solution $u$ of \eqref{degenhb} is at best Lipschitz, even when $g$, $H$ and $A$ are all smooth. This means that $U = \int u \, dm$ is at best $\bd_1$-Lipschitz, and so we cannot expect an anologue of \cite[Theorem 2.7]{ddj2023} when $\sigma$ is degenerate, because in this case $U$ does not ``inherit" $H^{-s}$-Lipschitz regularity from the data. On the other hand, because we can still get $\bd_1$-Lipschitz bounds for degenerate $\sigma$, we can hope for an analogue of Theorem 2.6 of \cite{ddj2023}, i.e. a dimension-dependent rate which is close to $N^{-1/d}$. This is exactly what we obtain in Theorem \ref{thm.main} below, where we show that
\begin{align} \label{mainestintro}
    |V^N(t,\bx) - U(t,m_{\bx}^N)| \leq C N^{- \gamma_d},
\end{align}
where 
\begin{align} \label{gammadef}
    \gamma_d \coloneqq \begin{cases}
        1/ (3d + 19) & \text{if $d$ is even}, \\
        1/ (3d + 16) & \text{if $d$ is odd}.
    \end{cases}
\end{align}
under general conditions on the data, which in particular allow for both degenerate idiosyncratic noise and non-zero common noise. In fact, if either (i) $\sigma$ is constant or (ii) $U$ is globally Lipschitz with respect to the metric inherited from $W^{-2,\infty}$, then \eqref{mainestintro} can be improved to
\begin{align} \label{mainestintro2}
    |V^N(t,\bx) - U(t,m_{\bx}^N)| \leq C N^{- \gamma'_d}, 
\end{align}
where 
\begin{align}
    \label{gammaprimedef}
    \gamma'_d = \begin{cases}
        1/(d + 7) & \text{if $d$ is even}, \\
        1/(d + 6) & \text{if $d$ is odd},
    \end{cases}
\end{align}
which is very close to the conjectured optimal rate of $N^{-1/d}$, and obviously sharper than the estimate \eqref{ddjd1}. We note that a $W^{-2,\infty}$-Lipschitz bound on $U$ can be obtained when the data is smooth enough and $\sigma$ is non-degenerate via the arguments in Section 3 of \cite{djs2023}. Thus while our main motivation is the case of degenerate $\sigma$ and non-zero $\sigma^0$, even in the case $\sigma = \sqrt{2}I_{d \times d}$, $\sigma^0 = 0$, our main result improves on \cite[Theorem 2.6]{ddj2023} and provides an almost optimal result in this setting, with the only possible improvement being to replace $N^{-1/(d+7)}$ (or $N^{-1/(d+6)}$ if $d$ is odd) with $N^{-1/d}$. We note that the improved rate in the case of constant $\sigma$ comes from an improvement in a key technical step of the subsolution bound for the regularized $U^{\delta, \eps}$ (see Paragraph \ref{sec:strategyofproof} of this introduction). In particular, Lemma \ref{lem.linfbound} below is a refinement of \cite[Proposition 4.4]{ddj2023} with a sharper dependence on $\delta$, which is achieved by taking advantage of an explicit representation of the duality map $H^{s} \ni \phi \mapsto \phi^* \in H^{-s}$ for integer $s$.

In addition to the general results discussed above, we study in Section \ref{sec: zeronoise} the special case of zero idiosyncratic noise, i.e. $\sigma = 0$. The mean-field limit in this setting was already addressed in the literature in the presence of a pure common noise in \cite{GangboMayorgaSwiech, mayorgaswiech} and without any noise in  \cite{Fornasier2014,CavagnariLisiniOrrieriSavare,FornasierLisiniOrrieriSavare}. From a PDE point of view,  this case is intriguing since the Hamilton-Jacobi equations describing the $N$-particle problem and the limit one have precisely the \textit{same} structure. To wit, if $U$ is a smooth (enough) solution to the HJB equation \eqref{hjbinf} then $U^N(t,\bx) = U(t, m_{\bx}^N) : [0,T] \times (\T^d)^N \rightarrow \R$ solves exactly the finite dimensional equation \eqref{hjbn} and so if $U$ is smooth, then for all $N \in \mathbb{N}$ and $(t,\bx) \in [0,T] \times (\T^d)^N$ we have $V^N(t,\bx) = U(t, m_{\bx}^N)$. This point was addressed in the very recent preprint \cite{SwiechWessels}, where convexity conditions were used to establish sufficient regularity on $U$ to guarantee that $U^N = V^N$. 

We observe in Proposition \ref{prop.zeronoise} that if $\sigma = 0$, then $U(t,m_{\bx}^N) \leq V^N(t,\bx)$, and this does not require any additional regularity on $U$. Combining this with Proposition \ref{prop.easy}, we get 
\begin{align} \label{introboundnonoise}
    0 \leq V^N(t,\bx) - U(t,m_{\bx}^N) \leq C r_{N,d}, 
\end{align}
where
\begin{align} \label{rnddef}
 r_{N,d} = \begin{cases} 
N^{-1/d} & d \geq 3, \\
N^{-1/2} \log N & d = 2, \\
N^{-1/2} & d = 1.
 \end{cases}
\end{align}
We then turn our attention to the special case of no noise ($\sigma = \sigma^0 = 0$) and purely quadratic Hamiltonian $(F = 0, H = \frac{1}{2} |p|^2$). In this setting a Hopf-Lax formula for both $V^N$ and $U$ is available, which we use to construct an example where 
\begin{align*}
    \sup_{t,\bx} \Big( V^N(t,\bx) - U(t,m_{\bx}^N) \Big) > cN^{-1/d}
\end{align*}
for some $c > 0$ (see Proposition \ref{prop:opt:bound:N}). This confirms that when $\sigma = 0$ and $U$ is not smooth, there can indeed be a gap between $U$ and $V^N$, but only in one direction. In addition, this example implies that \eqref{introboundnonoise} is optimal for $d \geq 3$. Finally, we highlight two other interesting phenomena in the special case of zero noise and quadratic Hamiltonian. First, we show that in this case the value function $U$ is sensitive to the choice of admissible controls, and in particular the value over Lipschitz feedback controls can be strictly larger than the value over less regular controls (see Corollary \ref{cor.feedback}). This is in contrast to the case of non-degenerate idiosyncratic noise, where it is possible to approximate an arbitrary open-loop control in an appropriate sense by regular feedbacks, as explained in \cite[Theorem 2.14 and Proposition 3.12]{Djete:COCV}. Second, we show in Proposition \ref{prop:equalityifconvexity} that if $G$ satisfies a weak form of convexity (which is implied by displacement convexity) we have $V^N = U$, i.e. this weak convexity is enough to guarantee that $V^N = U$ without smoothness of $U$. 

To the best of our knowledge, these results in the zero idiosyncratic noise regime are new and improve in some ways on the existing literature \cite{GangboMayorgaSwiech, mayorgaswiech,CavagnariLisiniOrrieriSavare,FornasierLisiniOrrieriSavare}. In particular, the observation that the value of the mean-field control problem is always lesser than the value of the $N$-particle problem seems to have been overlooked in previous contributions on the subject. This being said, we rely there heavily on the control formulation, and, in some places, on an explicit Hopf-Lax formula, available for purely quadratic Hamiltonians. We also believe that, by contrast, these findings shed some light on the role of the idiosyncratic noise with respect to the comparison of the value functions. 

To summarize, the main contribution of the present paper is to establish the bound \eqref{mainestintro} in a setting where $\sigma(x)$ can be degenerate and $\sigma^0$ can be non-zero. This result extends Theorem 2.6 of \cite{ddj2023} (but with a worse exponent) to include common noise and degenerate, non-constant idiosyncratic noise. In the case that $\sigma$ is constant or $U$ is Lipschitz with respect to the weaker metric inherited from $W^{-2,\infty}$, we get the better estimate \eqref{mainestintro2}, which is almost optimal. Finally, when $\sigma = 0$, we show that the problem greatly simplifies because $U \leq V^N$, and this allows to obtain the bound \eqref{introboundnonoise}, which we confirm through an example is optimal when $d \geq 3$.

\subsection{Strategy of the proof} 

\label{sec:strategyofproof}

The basic strategy of the proof is similar to that of Theorem 2.6 in \cite{ddj2023}. In particular, as in \cite{ddj2023} we obtain the bound $V^N \leq U + C r_{N,d}$, with $r_{N,d}$ as in \eqref{rnddef}, by a control-theoretic argument which has already appeared several times in the literature. For the more difficult inequality $U \leq V^N + CN^{- \gamma_d}$, we use a more analytical approach, and attempt to approximate $U$ by a function which is almost a subsolution to \eqref{hjbinf}, and which is smooth enough that the ``error term" appearing \eqref{errorterm} can be controlled. In \cite{ddj2023}, this approximation is achieved in three steps $U \rightarrow U^{\delta} \rightarrow U^{\delta, \eps} \rightarrow U^{\delta, \eps, \lambda}$,
defined successively by 
\begin{align} \label{regularizationscheme}
    &U^{\delta}(t,m) = U(t,m * \rho_{\delta}), \quad 
    U^{\delta, \eps}(t,m) = \sup_{m' \in \cP(\T^d)} \Big\{U^{\delta}(t,m') - \frac{1}{2\eps} \|m -m'\|_{-s}^2 \Big\}, 
    \nonumber \\
    &U^{\delta, \eps, \lambda}(t,m) = U^{\delta, \eps}\big(t, (1-\lambda) m + \lambda \leb \big), 
\end{align} 
where $\rho_{\delta}$ is a standard smooth approximation to the identity, $s$ is a large enough constant, $\leb$ is the Lebesgue measure on $\T^d$, and $\delta > 0$, $\eps > 0$, $\lambda \in (0,1)$ are small parameters. Roughly speaking, $U \to U^{\delta}$ ``upgrades" from $\bd_1$-Lipschitz to $H^{-s}$-Lipschitz, $U^{\delta} \to U^{\delta, \eps}$ creates $C^{1,1}$-regularity with respect to the Hilbert space $H^{-s}$ (and hence bounds on $D^2_{mm} U^{\delta, \eps}$, at least formally, see \eqref{dmmbound} below), and the final transformation is included because it is difficult to obtain an efficient estimate on the subsolution property of $U^{\delta, \eps}$ at measures which do not have full support. The main technical difficulty in the proof of Theorem 2.6 in \cite{ddj2023} is to obtain quantitative error estimates which show that 
\begin{enumerate}
    \item $D^2_{mm} U^{\delta, \eps, \lambda}$ blows up at a certain rate as $\delta, \eps, \lambda \to 0$,
    \item $U^{\delta, \eps, \lambda}$ converges to $U$ as $\delta, \eps, \lambda \to 0$,
    \item $U^{\delta, \eps, \lambda}$ is a subsolution of \eqref{hjbinf} up to an error term which vanishes as $\delta, \eps, \lambda \to 0$.
\end{enumerate}
These bounds show that $(t,\bx) \mapsto U^{\delta, \eps, \lambda}(t,m_{\bx}^N)$ solves the equation for \eqref{hjbn} up to an error which can be estimated in terms of $\delta$, $\eps$, $\lambda$, and $N$, which leads to an estimate of the form
\begin{align*}
    U \leq V^N + C (\delta, \eps, \lambda,N).
\end{align*}
The final estimate is obtained by choosing $\delta, \eps, \lambda$ as functions of $N$ (see the proof of Proposition 5.12 in \cite{ddj2023} for details).

There are various challenges we have to face when attempting to adapt this regularization scheme to the case of common noise and degenerate idiosyncratic noise. Here we explain the main issues caused by the the degenerate idiosyncratic noise and the common noise, respectively, as well as the techniques we use to overcome these challenges.

\subsubsection{The degenerate and non-constant idiosyncratic noise} Let us set $\sigma^0 = 0$ for the moment, and discuss only the effect of idiosyncratic noise. In this case, the main difference with the case treated in \cite{ddj2023} comes in the step $U \to U^{\delta}$. In particular, the proof that $U^{\delta}$ is almost a subsolution of \eqref{hjbinf} (see Lemma \ref{lem.udeltahatsubsol}) is quite involved when $\sigma$ depends on $x$, and in order to conclude we have to establish a sort of commutator estimate, see Lemma \ref{lem.commutator}. Moreover, in order to complete the proof of Lemma \ref{lem.udeltahatsubsol}, we have to treat the ``full Hamiltonian" 
\begin{align} \label{fullham}
  \R^{d \times d} \times \R^d \times \T^d \ni  (q,p,x) \mapsto \tr\big(A(x) q \big) + H(x,p)
\end{align}
as globally Lipschitz. Notice that while $H$ is generally just locally Lipschitz, we can show that $U$ is $\bd_1$-Lipschitz, hence we have (at least formally) an a-priori estimate on $\|D_{m} U\|_{\infty}$, which allows us to treat the map $(x,p) \mapsto H(x,p)$ as Lipschitz. When $A$ is constant, this is enough, but when $A$ is non-constant, the full Hamiltonian in \eqref{fullham} will be Lipschitz only when $q = D_{xm} U$ is bounded, i.e. when $D_m U$ is Lipschitz in $x$, which formally holds when $U$ is $W^{-2,\infty}$-Lipschitz. In other words, when $A$ is non-constant we need an a-priori estimate on the $W^{-2,\infty}$-Lipschitz constant of $U$ in order to conclude that $U^{\delta}$ is almost a subsolution of \eqref{hjbinf}, even formally. But as discussed above, when $\sigma$ is degenerate we cannot expect $U$ to be $W^{-2,\infty}$-Lipschitz, even if the data is smooth. This creates a significant complication when the data is both non-constant and degenerate.

In order to circumvent this first difficulty, we use a vanishing viscosity procedure, i.e. we begin by replacing $U$ by $U^{\eta}$, where $U^{\eta}$ is defined exactly like $U$ but with an additional idiosyncratic noise of strength $\eta$ in the dynamics of $X$. For technical reasons, we also mollify $F$ and $G$ in an appropriate way when defining $U^{\eta}$. We define $V^{N,\eta}$ similarly, see Subsection \ref{subsec:uepsdef} for details. In Subsection \ref{subsec:uepsestimates}, we show that under appropriate assumptions, we have
\begin{align} \label{uetaintro}
   |V^N - V^{N,\eta}| \leq C\sqrt{\eta}, \quad  |U - U^{\eta}| \leq C \sqrt{\eta}, \quad \text{Lip}(U^{\eta}; W^{-2,\infty}) \leq C/\eta, 
\end{align}
where $\text{Lip}(U^{\eta}; W^{-2,\infty})$ denotes the Lipschitz constant of $U^{\eta}$ in $m$ with respect to the $W^{-2,\infty}$-metric. The bounds in \eqref{uetaintro} allow us to reduce to the case that $U$ is $W^{-2,\infty}$-Lipschitz, at least provided that we can track the dependence of the estimates on the $W^{-2,\infty}$-Lipschitz constant. Indeed, in Proposition \ref{prop.mainextrareg}, we obtain the bound 
\begin{align*}
    U \leq V^N + C \Big(1 + \lip(U; W^{-2,\infty}) \Big)^{1 -1/(2s^* +1)} N^{-1 / (2s^* + 1)},
\end{align*}
with $s^*$ given by \eqref{sdef} below,
under the additional assumption that $U$ is $W^{-2,\infty}$-Lipschitz, which, when applied to each $\eta > 0$, leads to 
\begin{align*}
    U &\leq C\sqrt{\eta} + U^{\eta} \leq   V^{N,\eta} + C\sqrt{\eta}+ C \eta^{-1 + 1/(2s^* + 1)} N^{-1 / (2s^* + 1)} 
    \\
    &\leq V^N + C \sqrt{\eta} + C\eta^{-1 + 1/(2s^* + 1)} N^{-1 / (2s^* + 1)}, 
\end{align*}
and it is optimizing in $\eta$ and using the definition of $s^*$ which leads to the estimate \eqref{mainestintro}. When $A$ is constant or $U$ is already known to be $W^{-2,\infty}$-Lipschitz, this vanishing viscosity step turns out to be unnecessary, which explains why we get a sharper estimate in these cases. 

We emphasize that while the estimate on $|U - U^{\eta}|$ is standard, the estimate on $\text{Lip}(U^{\eta}; W^{-2,\infty})$ is not, because the natural control-theoretic technique for obtaining Lipschitz bounds on $U^{\eta}$, as employed in Section 3 of \cite{ddj2023}, leads to bounds on $\text{Lip}(U^{\eta}; W^{-2,\infty})$ which have an exponential dependence on $\eta$. In order to get a more efficient bound, we instead use a stochastic-analytic argument which resembles the nonlinear adjoint method to get a bound of the form $|D^2_{x^ix^i} V^{N,\eta}| \leq C/(N\eta)$, where $V^{N,\eta}$ is defined like $V^N$ but with an additional idiosyncratic noise of strength $\eta$. By an argument introduced in \cite{djs2023}, this is enough to conclude the desired bound on $\text{Lip}(U; W^{-2,\infty})$.

\subsubsection{The common noise} The vanishing viscosity procedure discussed above allows us to deal with the degeneracy of the idiosyncratic noise, so we now assume that the noise is non-degenerate and $U$ is $W^{-2,\infty}$-Lipschitz, so we can safely ignore the issues pointed out above. Nevertheless, we still have to incorporate the common noise into the regularization scheme. The issue here is not with the step $U \to U^{\delta}$ (which works fine with common noise) but instead with the step $U^{\delta} \to U^{\delta, \eps}$. Indeed, when $\sigma^0 > 0$, we are unable to show that the function $U^{\delta, \eps}$ defined in \eqref{regularizationscheme} is almost a subsolution to \eqref{hjbinf}. Roughly speaking, the problem seems to be that the common noise operator $\com$ is simply not compatible with the sup-convolution operation appearing in the definition of $U^{\delta, \eps}$. We found a way around this issue by employing a change of variables which has recently been used in \cite{bayraktar2023, djs2023} to treat the comparison principle for equations like \eqref{hjbinf}. In particular, we first define $\hat{U} : [0,T] \times \R^d \times \cP(\T^d) \to \R$ and $\hat{V}^N : [0,T] \times \R^d \times (\T^d)^N \to \R$ by
\begin{align*}
    \hat{U}(t,z,m) = U\big(t, (\text{Id}+z)_{\#} m\big), \quad \hat{V}^N(t,z,\bx) = \hat{V}^N(t,z,x^1,...,x^N) = V^N(t,x^1 + z,...,x^N + z).
\end{align*}
It turns out that $\hat{U}$ solves (formally at least) the PDE 
\begin{align} \label{hjbz}
     \begin{cases} \ds 
          - \partial_t \hat{U}(t,z,m) - A_0 \Delta_z \hat{U}(t,z,m) - \int_{\T^d}\tr\big( \hat{A}(x,z) D_{xm} \hat{U}(t,z,m,x) \big) m(dx)  \\ \ds 
        \qquad \qquad + \int_{\T^d} \hat{H}\big(x,z,D_m \hat{U}(t,z,m,x)\big) = \hat{F}(z,m), \quad (t,z,m) \in [0,T) \times \R^d \times \cP(\T^d), \vspace{.1cm} \\ \ds
        \hat{U}(t,z,m) = \hat{G}(z,m), 
     \end{cases} 
\end{align}
in a viscosity sense, where we have set
\begin{align} \label{hatdefs}
        \hat{A}(x,z) =A(x+z), \quad \hat{H}(x,z,p) = H(x + z, p), \quad \hat{F}(z,m) = F(m^z), \quad  \hat{G}(z,m) = G(m^z), 
\end{align} 
and we have used the notation $m^z = (\text{Id} + z)_{\#} m$ for simplicity. For $x \in \T^d = \R^d / \Z^d$ and $z \in \R^d$, we are defining the sum $x + z$ in the obvious way, i.e. if $x = [y]$, i.e. $x$ is the equivalence class of $y \in \R^d$, then $x + z = [y + z]$. Obviously, $\hat{U}$ is periodic in $z$, so could instead be viewed as a function on $[0,T] \times \T^d \times \cP(\T^d)$, but it is convenient for technical reasons to work with $z \in \R^d$. Similarly, $\hat{V}^N$ satisfies 
\begin{align} \tag{$\widehat{\text{HJB}}_N$} \label{hjbhatn}
    \begin{cases} 
   \ds  - \partial_t \hat{V}^N - \sum_{i = 1}^N \tr\big( \hat{A}(x^i,z) D^2_{x^ix^i} \hat{V}^N\big) - A_0 \Delta_z \hat{V}^N + \frac{1}{N} \sum_{i = 1}^N \hat{H}(x^i,z,ND_{x^i} \hat{V}^N) = \hat{F}(z,m_{\bx}^N), \vspace{.1cm} \\
   \hspace{9cm} (t,z,\bx) \in [0,T) \times \R^d \times (\T^d)^N, \vspace{.1cm} \\
   \ds \hat{V}^N(T,\bx) = \hat{G}(z,m_{\bx}^N), \quad (z,\bx) \in \R^d \times (\T^d)^N.
    \end{cases}
\end{align}
In other words, this change of variables exchanges the common noise operator $\com$ for a finite-dimensional Laplacian in the variable $z$, and so we focus instead on the equivalent problem of estimating $\hat{U} - \hat{V}^N$. Remarkably, if we define $\hat{U}^{\delta}$, $\hat{U}^{\delta, \eps}$, $\hat{U}^{\delta,\eps,\lambda}$ by 
\begin{align*}
       &\hat{U}^{\delta}(t,m) = \hat{U}(t,z,m * \rho_{\delta}), \\ 
    &\hat{U}^{\delta, \eps}(t,z,m) = \sup_{z' \in \T^d,  m' \in \cP(\T^d)} \Big\{\hat{U}^{\delta}(t,z',m') - \frac{1}{2\eps} \Big( \|m -m'\|_{-s}^2 + |z-z'|^2 \Big) \Big\}, 
    \nonumber \\
    &\hat{U}^{\delta, \eps, \lambda}(t,m) = \hat{U}^{\delta, \eps}\big(t, (1-\lambda) m + \lambda \leb \big), 
\end{align*}
then we can in fact get all the relevant error estimates, and so once again complete the proof by choosing $\delta, \eps, \lambda$ as functions of $N$. To make this argument rigorous, we have to argue that $\hat{U}^{\delta, \eps, \lambda}$ is almost a subsolution to \eqref{hjbz} and satisfies certain regularity estimates, and then infer that the function $(t,z,\bx) \mapsto \hat{U}^{\delta, \eps, \lambda}(t,z,m_{\bx}^N)$ is almost a subsolution to \eqref{hjbhatn}. In \cite{ddj2023}, this point is less subtle because there are no second-order terms in the equation, so that the subsolution estimates can be interpreted in a pointwise fashion (see \cite[Lemma 5.13]{ddj2023}). Here, we instead use an argument based on Jensen's lemma to verify that $(t,z,\bx) \mapsto \hat{U}^{\delta, \eps, \lambda}(t,z,m_{\bx}^N)$ satisfies the relevant estimate in a viscosity, rather than pointwise, sense; see the proof of Proposition \ref{prop.comparison} below.

\subsubsection{Summary of the regularization argument} We have discussed how to overcome the difficulties posed by the degenerate idiosyncratic noise and the non-zero common noise separately. In order to treat them at the same time, we have to put together (i) the vanishing viscosity procedure necessary to create $W^{-2,\infty}$-Lipschitz bounds because of the degenerate and non-constant idiosyncratic noise, (ii) the change of variables necessary to exchange the common noise operator with a finite-dimensional Laplacian, and (iii) the three transformations already used in \cite{ddj2023}. This leads to the sequence of transformations
\begin{align*}
    U \to U^{\eta} \to \hat{U}^{\eta} \to \hat{U}^{\eta, \delta} \to \hat{U}^{\eta, \delta, \eps} \to \hat{U}^{\eta, \delta, \eps, \lambda}.
\end{align*}
The function $U^{\eta}$ is defined just like $U$ but with an additional vanishing viscosity term (see \eqref{uepsdefcontrol} below), and then the remaining transformation are defined successively by 
\begin{align*}
    &\hat{U}^{\eta}(t,z,m) = U^{\eta}\big(t, (\text{Id}+z)_{\#} m \big), \quad \hat{U}^{\eta, \delta}(t,z,m) = U(t,z, m * \rho_{\delta}), \\
    &\hat{U}^{\eta, \delta, \eps}(t,z,m) = \sup_{z' \in \T^d, m' \in \cP(\T^d)} \Big\{\hat{U}^{\eta, \delta}(t,z',m') - \frac{1}{2\eps} \Big( \|m -m'\|_{-s}^2 + |z-z'|^2 \Big) \Big\}, \\
    &\hat{U}^{\eta, \delta, \eps, \lambda }(t,z,m) = \hat{U}^{\eta, \delta, \eps, \lambda}\big(t,z, (1 -\lambda) m + \lambda \leb  \big).
\end{align*}
As in \cite{ddj2023}, the key technical step is then to obtain quantitative error bounds showing that 
\begin{enumerate}
    \item $D^2_{mm} \hat{U}^{\eta, \delta, \eps, \lambda}$ blows up at a certain rate as $\delta, \eps, \lambda \to 0$,
    \item $\hat{U}^{\eta, \delta, \eps, \lambda}$ converges to $U$ as $\delta, \eps, \lambda \to 0$,
    \item $\hat{U}^{\eta,\delta, \eps, \lambda}$ is a subsolution of \eqref{hjbinf} up to an error term which is vanishes as $\delta, \eps, \lambda \to 0$.
\end{enumerate}
Finally, using these bounds we must argue that $(t,z,\bx) \mapsto \hat{U}^{\eta, \delta, \eps, \lambda}(t,z,m_{\bx}^N)$ is a subsolution of \eqref{hjbhatn} up to some explicit error bounds, which leads to
\begin{align*}
    U \leq V^N + C(\eta, \delta, \eps, \lambda,N)
\end{align*}
for some explicit function $C$. Finally, we choose $\eta, \delta, \eps, \lambda$ as functions of $N$ to conclude.

\subsection{Possible extensions and limitations of the method} \label{subsec:ext}

Our main assumptions (see Assumption \ref{assump.maindegen}) are fairly minimal in terms of regularity. We require only that $F$ and $G$ are $\bd_1$-Lipschitz and $\tv$-semiconcave (explained below), $H$ is $C^2$ and satisfies natural convexity and local Lipschitz properties, and that $\sigma$ is $C^2$. These represent a simple set of conditions under which we can verify that $U$ is $\bd_1$-Lipschitz and $\tv$-semiconcave, which is the basic starting point of our regularization scheme.

Regarding the Hamiltonian, we could relax the $C^2$ assumption at the expense of an additional mollification step, and also we could replace $H(x,p) - F(m)$ with $\mathcal{H}(x,p,m)$, provided that $\mathcal{H}$ satisfies appropriate regularity conditions. We choose to work with separable Hamiltonians, because there are several steps (e.g. when we mollify the data or when we differentiate $V^{N,\eta}$ twice in order to prove that $|D^2_{x^ix^i} V^{N,\eta}| \leq C/(N\eta)$) where the separability of the Hamiltonian leads to shorter and cleaner arguments. 

We could also relax somewhat the dynamics of the state process. For example, we could replace the linear drift control \eqref{dynamicsinf} with 
\begin{align*}
    dX_t = b(X_t,\alpha_t) dt + \sigma(X_t) dW_t + \sigma^0 dW_t^0
\end{align*}
without much difficulty, at least provided that $H(x,p) = \sup_{a} \Big(- L(x,a) - b(x,a) \cdot p \Big)$ satisfies appropriate regularity conditions. If we included $m$-dependent dynamics of the form 
\begin{align*}
    dX_t = b(X_t,\alpha_t,\cL(X_t | W^0)) dt + \sigma(X_t, \cL(X_t | W^0)) dW_t + \sigma^0 dW_t^0, 
\end{align*}
most of the proofs go through but there is a technical problem with the $\tv$-semiconcavity bound on $U$. In particular, it is not clear to us whether an analogue of Lemma \ref{lemma: stability_mubar} holds in this case, the issue being that linear Fokker-Planck equations always generate contractions in the total variation norm (even if they have irregular data), but the same is not true for non-linear Fokker-Planck equations we would get if we had $m$-dependent dynamics. But when $\sigma$ is non-degnerate, $\tv$-semiconcavity can be recovered by another argument, and so we expect in this case including $m$-dependent dynamics should not be a major problem. We do not pursue this direction because we prefer to focus on degenerate idiosyncratic noise. 

Finally, we mention that we certainly cannot allow $\sigma^0$ to depend on $x$ (or $m$) with our current methods, because then we are unable to perform the same change of variables to handle the corresponding common noise operator. This is one of the primary limitations of our approach.

\subsection{Organization of the paper} In Section \ref{sec: notation}, we fix notation, define precisely the value functions $V^N$ and $U$, and state our main results. In Section \ref{sec: reg_values}, we approximate $V^N$ and $U$ via a vanishing viscosity procedure. Section \ref{sec:easy} contains a proof for the easier of the two inequalities appearing in our main result, Theorem \ref{thm.main}, while Section \ref{sec:hardinequalityconditional} contains a proof of the harder inequality. Finally, in Section \ref{sec: zeronoise} we analyze the special case of zero idiosyncratic noise.

\section{Notation, preliminaries, and main results} \label{sec: notation}

\noindent \textbf{Spaces of functions and distributions on $\T^d$:} Given $d \in \N$, we define $\T^d \coloneqq \R^d/ \Z^d$ to be the  $d$-dimensional flat torus. We write $x = (x_1,...,x_d)$ for the general element $\T^d$, and abuse notation slightly by writing $|x - y|$ for the distance between elements $x$ and $y$ of $\T^d$. We endow $\T^d$ with the usual Lebesgue measure. For $\phi : \T^d \to \R$, we write $D\phi = (D_{x_1} \phi, ..., D_{x_d} \phi)$ for the gradient of $\phi$ and $D^2 \phi = (D^2_{x_ix_j} \phi)_{i,j = 1,...,d}$ for the Hessian of $\phi$. For higher derivatives we use the following multi-index notation: for $k \in \N$ and $\alpha = (\alpha_1,...,\alpha_k) \in \{1,...,d\}^k$, we write $D^{\alpha} \phi = D_{x_{\alpha_1}}...D_{x_{\alpha_k}} \phi$. We write $|\alpha| = k$ to indicate that $\alpha$ is a multi-index of length $k$. We write $C^k(\T^d)$ for the space of functions $\phi : \T^d \to \R$ whose partial derivatives up to order $k$ are continuous, and we endow $C^k(\T^d)$ with the norm 
\begin{align*}
    \|\phi\|_{C^k} = \|\phi\|_{\infty} + \sum_{1 \leq |\alpha| \leq k} \|D^{\alpha} \phi\|_{\infty},
\end{align*}
with $\|\cdot\|_{\infty}$ being the usual $L^{\infty}$ (essential supremum) norm.
Similarly, we write $W^{1,\infty} = W^{1,\infty}(\T^d)$ or $\text{Lip}(\T^d)$ for the space of functions $\phi : \T^d \to \R$ which are Lipschitz continuous. Likewise, we define $W^{2,\infty} = W^{2,\infty}(\T^d)$ to be the space of function $\phi \in C^1(\T^d)$ with Lipschitz first derivatives, and endow these spaces with the norms 
\begin{align*}
   \|\phi\|_{1,\infty} \coloneqq  \| \phi \|_{\infty} + \|D \phi\|_{\infty}, \quad  \|\phi\|_{2,\infty} \coloneqq  \| \phi \|_{\infty} + \|D \phi\|_{\infty} + \|D^2 \phi\|_{\infty}.
\end{align*}
We define $L^2 = L^2(\T^d)$ to be the usual space of square-integrable functions on $\T^d$, with the inner product 
\begin{align*}
    \langle \phi, \psi \rangle_{L^2} = \int_{\T^d} \phi \psi dx.
\end{align*}
For $s \in \N$, we define $H^s = H^s(\T^d)$ to be the Hilbert space of functions whose derivatives up to order $s$ are in $L^2$, and we endow $H^s$ with the usual inner product 
\begin{align*}
    \langle \phi, \psi \rangle_s = \sum_{0 \leq |\alpha| \leq s} \langle D^{\alpha} \phi, D^{\alpha} \psi \rangle_{L^2}, 
\end{align*}
where if $|\alpha| = 0$, we set $D^{\alpha} \phi = \phi$. We denote by $H^{-s}$ the topological dual of $\phi$. We recall that $H^{-s}$ inherits an inner product from $H^s$, which we denote $\langle \cdot, \cdot \rangle_{-s}$. For $\phi \in H^s$, we denote by $\phi^*$ the dual element of $H^{-s}$, defined by the formula $\phi^*(\psi) = \langle \phi, \psi\rangle_s$. Likewise, given $q \in H^{-s}$, we denote by $q^*$ the dual element of $H^s$, which satisfies $\langle q^*, \phi \rangle_s = q(\phi)$. We note that if $\psi$ is smooth enough, then we can integrate by parts to get
\begin{align*}
    \phi^*(\psi) = \langle \phi, \psi \rangle_s = \sum_{0 \leq |\alpha| \leq s} \langle D^{\alpha} \phi, D^{\alpha} \psi \rangle_{L^2} = \sum_{0 \leq |\alpha| \leq s} (-1)^{|\alpha|} \langle \phi, D^{\alpha} D^{\alpha} \psi \rangle_{L^2}. 
\end{align*}
In particular, this means that
\begin{align} \label{dualidentity}
    \phi^* = \sum_{0 \leq |\alpha| \leq s} (-1)^{|\alpha|} D^{\alpha} D^{\alpha} \phi  
\end{align}
in a distributional sense, and in particular if $\phi \in C^{2s} \subset H^s$, then we have the bound
\begin{align}  \label{c2sbound}
    \|\phi^*\|_{\infty} \leq C \|\phi\|_{C^{2s}}, \quad C = C(d,s).
\end{align}

\noindent
\textbf{Functions on $(\T^d)^N$:} We write $\bx = (x^1,..,x^N)$ for the general element of $(\T^d)^N$. For $i = 1,...,N$, $x^i \in \T^d$ can be further expanded as $x^i = (x^i_1,...,x^i_d)$. Given a smooth function $V : (\T^d)^N \to \R$, we write $D_{x^i} V$ for the gradient of $V$ in the direction $x^i$, i.e. $D_{x^i} V : (\T^d)^N \to \R^d$ with $D_{x^i} V = (D_{x^i_1} V,...,D_{x^i_d} V)$. Similarly, for $i,j \in \{1,...,N\}$, we write $D^2_{x^ix^j} V$ for $d \times d$ matrix of partial derivatives given by $ D^2_{x^ix^j} V = (D^2_{x^i_k x^j_l} V)_{k,l = 1,...,d}$. 
\newline \newline 
\textbf{Analysis on $\cP(\T^d)$:} We denote by $\cP(\T^d)$ the space of Borel probability measures on $\T^d$. We denote by $\bd_1$ the 1-Wasserstein metric on $\cP(\T^d)$, defined as 
\begin{align*}
    \bd_1(m,m') = \sup_{\phi \text{ $1$-Lipschitz}} \int_{\T^d} \phi \,d(m - m').
\end{align*}
At times we will also endow $\cP(\T^d)$ with the metric inherited from duality with $W^{2,\infty}$, and make use of the notation 
\begin{align*}
    \|m - m'\|_{-2,\infty} = \sup_{\|\phi\|_{2,\infty} \leq 1} \int_{\T^d} \phi \,d(m -m').
\end{align*}
We say that a function $\Phi(m) : \cP(\T^d) \to \R$ admits a linear derivative if there exists a continuous function $\frac{\delta \Phi}{\delta m}(m,x) : \cP(\T^d) \times \T^d \to \R$ such that
\begin{align} \label{linderivative}
    \Phi(m') - \Phi(m) = \int_0^1 \int_{\T^d} \frac{\delta \Phi}{\delta m}\big(tm + (1-t)m',x \big)  (m-m')(dx)
\end{align}
for each $m,m' \in \cP(\T^d)$, and
\begin{align} \label{normalization}
    \int_{\T^d} \frac{\delta \Phi}{\delta m}(m,x) dx = 0, \quad m \in \cP(\T^d).
\end{align}
We call $\frac{\delta \Phi}{\delta m}$ the linear derivative of $\Phi$. The linear derivative is uniquely defined when it exists, thanks to the normalization convention \eqref{normalization}. If $\Phi$ admits a linear derivative such that $x \mapsto \frac{\delta \Phi}{\delta m}(m,x)$ is $C^1$ for each fixed $x$, we define
\begin{align*}
    D_m \Phi(m,x) = D_x \frac{\delta \Phi}{\delta m}(m,x).
\end{align*}
We say that $\Phi$ is $\cC^1$ if $D_m \Phi : \cP(\T^d) \times \T^d \to \R^d$ exists and is jointly continuous. Higher order derivatives can be defined by iteration, e.g. $D^2_{mm} \Phi(m,x,x') = D_m \big[ D_m\Phi(\cdot, x) \big](m,x')$, $D_{xm} \Phi(m,x) = D_x \big[D_m \Phi(m,\cdot) \big](x)$. We say that $\Phi$ is $\cC^2$ if $D_m \Phi$, $D_{xm} \Phi, D^2_{mm} \Phi$ exist and are continuous. We refer to \cite{CarmonaDelarue_book_I} for more details about calculus on spaces of probability measures.

We note that, as in \cite{ddj2023}, we will at time work with functions $\Phi : H^{-s} \to \R$, and we will make use of Lemma 2.11 and Proposition 2.12 of \cite{ddj2023}, which point out some connections between analysis on $H^{-s}$ and analysis on $\cP(\T^d)$. In particular, Lemma 2.11 therein states that if $s > d/2 + 1$ and a function $\Phi : H^{-s} \to \R$ is $\cC^1$ with respect to the Hilbertian structure on $H^{-s}$, then in fact the restriction of $\Phi$ to $\cP(\T^d)$ is $\cC^1$ in the sense described above, and we have the formulas 
\begin{align} \label{linderivformula}
    &\frac{\delta \Phi}{\delta m}(m,x) = D_{-s} \Phi(m)(x) - \int_{\T^d} D_{-s} \Phi(m)(y) m(dy), \\ \label{wassderivformula}
    &D_m \Phi(m,x) = D_x D_{-s} \Phi(m)(x) 
\end{align}
for $m \in \cP(\T^d)$ and $x \in \T^d$, where $D_{-s} \Phi : H^{-s} \to H^s$ denotes the Frechet derivative of $\Phi$. The same arguments show that if $\Phi$ is $\cC^1$ with respect to $H^{-s}$ for $s > d/2 + 2$, then in fact 
\begin{align*}
    (m,x) \ni \cP(\T^d) \times \T^d \to \frac{\delta  \Phi}{\delta m}(m,\cdot) \in C^2(\T^d)
\end{align*}
is continuous, and in particular $D_{xm} \Phi(m,\cdot)$ exists and 
\begin{align} \label{dxmcont}
    \Phi \in \cC^1(H^{-s}) \text{ for } s > d/2 + 2 \implies \cP(\T^d) \times \T^d \ni (m,x) \mapsto D_{xm} \Phi(m,x) \text{ is continuous.} 
\end{align}Proposition 2.12 of \cite{ddj2023}, meanwhile, states that if $s > d/2 + 1$ and $\Phi : H^{-s} \to \R$ is $\cC^1$ with respect to $H^{-s}$, then there is a constant $C$ depending only on $d$ and $s$ such that
\begin{align} \label{dmmbound}
    |D_m \Phi(m,x) - D_m \Phi(m',x)| \leq C \big[ \Phi \big]_{C^{1,1}(H^{-s})}\bd_1(m,m')
\end{align}
for each $m,m' \in \cP(\T^d)$ and each $x \in \T^d$, where $\big[ \Phi \big]_{C^{1,1}(H^{-s})}$ denotes the seminorm 
\begin{align*}
    \big[ \Phi \big]_{C^{1,1}(H^{-s})} = \sup_{q,q' \in H^{-s}, q \neq q'} \frac{|\nabla_{-s} \Phi(q) - \nabla_{-s} \Phi(q')|}{\|q - q'\|_{-s}},
\end{align*}
and $\nabla_{-s}\Phi = (D_{-s} \Phi)^*$ is the Hilbertian gradient of $\Phi$.

We close this section with a useful technical lemma. Its proof is an easy generalization of the argument leading to equation (5.29) in \cite{ddj2023}, and so we omit it.

\begin{lem} \label{lem.regfromtouching}
    Let $\Phi, \Psi : \cP(\T^d) \to \R$, and assume that that $\Psi \in \cC^1(\cP(T^d))$. Assume moreover that for some $m_0 \in \cP(\T^d)$ satisfying $m_0 \geq c\leb$ for some $c > 0$, we have 
    \begin{align*}
        \Phi(m_0) - \Psi(m_0) = \sup_{m \in \cP(\T^d)} \Big( \Phi(m) - \Psi(m) \Big).
    \end{align*}
    Then,
    \begin{enumerate}
        \item If $\Phi$ is Lipschitz with respect to $\bd_1$, with Lipschitz constant $\lip(\Phi; \bd_1)$, then \begin{align*}
        \|D_m \Psi(m_0) \|_{\infty} \leq \lip(\Phi; \bd_1).
    \end{align*}
        \item If $\Phi$ is Lipschitz with respect to $W^{-2,\infty}$, with Lipschitz constant $\lip(\Phi; W^{-2,\infty})$, then 
        \begin{align*}
            \| \frac{\delta \Psi}{\delta m}(m_0,\cdot)\|_{2,\infty} \leq \lip(\Phi; W^{-2,\infty}).
        \end{align*}
    \end{enumerate}
    As a consequence, if $\Psi \in \cC^1(\cP(T^d))$ is Lipschitz with respect to $W^{-2,\infty}$, then 
    \begin{align*}
        \sup_{m \in \cP(\T^d)} \| \frac{\delta \Psi}{\delta m}(m,\cdot) \|_{2,\infty} \leq \lip(\Psi; W^{-2,\infty}).
    \end{align*}
\end{lem}

\subsection{Problem statement} \label{susec: prob_statement}

We fix in this paper numbers $d \in \N$ and $T > 0$, and we work on a filtered probability space $\big(\Omega, \sF, \bbF = (\sF_t)_{0 \leq t \leq T}, \bP\big)$ hosting independent $d$-dimensional Brownian motions $(W^i)_{i \in \N \cup \{0\}}$. The data for our mean field and $N$-particle control problems consist of the functions $\sigma$, $F$, $G$, $L$ and the constant $\sigma^0$ introduced in \eqref{introdata}. For notational convenience, we define $A : \T^d \to \R^{d \times d}$, $A_0 \geq 0$, $H(x,p) : \T^d \times \R^d \to \R$ by \eqref{hadef}.
\newline 
\newline 
\textbf{The definition of $V^N$:} For the $N$-particle problem, the $(\T^d)^N$-valued state process $\bX = (X^1,...,X^N)$ evolves according to the controlled dynamics
\begin{align} \label{nplayerdynamics}
    dX_t^i = \alpha_t^i dt + \sigma(X_t^i) dW_t^i + \sigma_0 dW_t^0, \,\, t_0 \leq t \leq T, \quad \bX_{t_0} = \bx_0,
\end{align}
where $\bm \alpha = (\alpha^1,...,\alpha^N)$ is the control and $\bx_0 = (x_0^1,...,x_0^N)$ is the initial condition. The value function $V^N(t,\bx) : [0,T] \times (\T^d)^N \to \R$ is defined by 
\begin{align} \label{vndef}
    &V^N(t_0,\bx_0) \coloneqq \inf_{\bm \alpha = (\alpha^1,...,\alpha^N) \in \sA^N_{t_0}} J^N(t_0,\bx_0,\bm \alpha), \text{ where } \nonumber 
    \\
    &J^N(t_0,\bx_0,\bm \alpha)  \coloneqq \E\bigg[\int_{t_0}^T \Big( \frac{1}{N} \sum_{i = 1}^N L\big(X_t^i, \alpha_t^i\big) + F\big(m_{\bX_t}^N\big) \Big) dt + G\big(m_{\bX_T}^N \big)\bigg],
\end{align}
with the infimum being taken over the set of $\sA^N_{t_0}$ of square-integrable, $(\R^d)^N$-valued controls $\bm \alpha = (\bm \alpha_t)_{t_0 \leq t \leq T}$ which are progressively measurable with respect to $\bbF$, and where in definition of $J^N$, $\bX$ is determined from $\bm \alpha$ by the dynamics \eqref{nplayerdynamics}. 
\newline \newline 
\noindent 
\textbf{The definition of $U$:} In defining the value function for the limiting problem, we take care to use a formulation which is compatible with the dynamic programming result in \cite{djete2019mckean}. In particular, for $t_0 \in [0,T)$ we define canonical space
\begin{align*}
    \Omega^{\infty,t_0} \coloneqq \T^d \times \cC^{t_0} \times \cC^{t_0}, \text{ with } \cC^{t_0} \coloneqq C([t_0,T] ; \R^d). 
\end{align*}
For $m_0 \in \cP(\T^d)$, we define a probability measure $\bP^{\infty, t_0,m_0}$ on $\Omega^{\infty,t_0}$ via the formula $\bP^{\infty, t_0, m_0} \coloneqq m_0 \otimes \bW^{t_0} \otimes \bW^{t_0}$, where $\bW^{t_0}$ is the Wiener measure on $\cC^{t_0}$. We write $(x, \omega, \omega^0)$ for the general element of $\Omega^{\infty,t_0}$, and denote by $(\xi, W, W^0)$ the canonical random variable and processes on $\Omega^{\infty,t_0}$, i.e. 
\begin{align*}
\xi(x,\omega, \omega^0) = x, \quad W_t(x,\omega, \omega^0) = \omega(t), \quad  W^0_t(x,\omega, \omega^0) = \omega^0(t), \quad t_0 \leq t \leq T.
\end{align*}
Let $\bbF^{\infty,t_0,m_0} = (\sF_t^{\infty,t_0,m_0})_{t_0 \leq 
 t \leq T}$ denote the right-continuous and $\bP^{\infty, t_0, m_0}$-complete augmentation of the filtration generated by $\xi$, $W$, and $W^0$, and let $\bbF^{0,t_0,m_0} = (\sF_t^{0,t_0,m_0})_{t_0 \leq t \leq T}$ denote the right-continuous and $\bP^{\infty, t_0, m_0}$-complete augmentation of the filtration generated by $W^0$. We denote by $\sA_{t_0,m_0}$ the set of all $\R^d$-valued processes $\alpha = (\alpha_t)_{t_0 \leq t \leq T}$ defined on $\Omega^{\infty, t_0,m_0}$ progressively measurable with respect to $\bbF^{\infty, t_0,m_0}$ and square-integrable with respect to $\bP^{\infty, t_0,m_0}$. Then we define 
\begin{align} \label{udefcontrol}
    &U(t_0,m_0) \coloneqq \inf_{\alpha \in \sA_{t_0,m_0}} J^{\infty}(t_0,m_0,\alpha), \text{ with } \nonumber 
    \\
    &J^{\infty}(t_0,m_0,\alpha) \coloneqq \E^{\bP^{\infty, t_0,m_0}}\bigg[ \int_{t_0}^T \Big( L\big(X_t, \alpha_t\big) + F\big( \sL^{0,m_0}(X_t) \big) \Big) dt + G\big(\sL^{0,m_0}(X_{T})\big)\bigg],
\end{align}
where in the definition of $J^{\infty}$, $X$ is determined from $\alpha$ by the dynamics
\begin{align} \label{dynamicsinf}
    dX_t = \alpha_t dt + \sigma(X_t) dW_t + \sigma_0 dW_t^0, \quad t_0 \leq t \leq T, \quad X_{t_0} = \xi,
\end{align}
and where for a random variable $X$, $\sL^{0,m_0}$ denotes the law of $X$ conditional on $\sF^{0,t_0,m_0}_T$ and with respect to $\bP^{\infty, t_0, m_0}$.

\subsection{Assumptions and main results} \label{subsec:mainresults}


\begin{assumption} \label{assump.maindegen}
  There are constants $C_G$ and $C_F$ such that $G$ satisfies
  \begin{align}
         |G(m) - G(m')| &\leq C_G \bd_1(m,m'),  \\ \label{gsc}
         G(\lambda m + (1-\lambda)m')  &\geq \lambda G(m) + (1-\lambda)G(m') - \frac{C_G}{2} \lambda(1-\lambda) \|m - m'\|_{\tv}^2
  \end{align}
  for all $m,m' \in \cP(\T^d)$ and all $\lambda \in (0,1)$, and $F$ satisfies the same estimates but with $C_F$ replacing $C_G$. There is a constant $C_{\sigma}$ such that
   \begin{align}
      \norm{\sigma}_{C^2(\T^d)} + |\sigma_0| \leq C_{\sigma}.
   \end{align}
   The Hamiltonian $H$ is $C^2$, and there is a constant $C_H > 0$ such that
   \begin{align} \label{hlocallip}
      \frac{1}{C_H} I_{d \times d} \leq D^2_{pp} H(x,p) \leq C_H I_{d \times d}, \quad |D_pH(x,p)| \leq C_H (1 + |p|)
      ,\quad 
    |D_x H(x,p)| \leq C_H (1 + |p|)
   \end{align}
    for all $(x,p) \in \T^d \times \R^d$.
\end{assumption}

\begin{rmk}

    We note that in the case $\sigma = \sqrt{2} I_{d \times d}$ and $\sigma^0 = 0$, Assumption \ref{assump.maindegen} matches exactly Assumption 2.1 in \cite{ddj2023}, except the ``$\bd_1$-semiconcavity" of $F$ and $G$ in \cite{ddj2023} is replaced with the weaker ``$\tv$-semiconcavity" property \eqref{gsc}. In particular, Assumption \ref{assump.maindegen} represents a natural set of conditions under which we can prove that the limiting value function $U$ is $\bd_1$-Lipschitz and $\tv$-semiconcave. The need to work with $\tv$-semiconcavity rather than $\bd_1$-semiconcavity is discussed below in Remark \ref{rmk.semiconcavity}. We refer to Subsection \ref{subsec:ext} above for a discussion of possible relaxations of these assumptions.

\end{rmk}

\begin{convention} \label{conv.data}
    Throughout the paper we will say that a constant $C$ \textbf{depends on the data} if it depends only on $d$, $T$ and the constants $C_G$, $C_F$, $C_{\sigma}$, $C_H$
    appearing in Assumption \ref{assump.maindegen}.
\end{convention}

Our main result is as follows.

\begin{thm} \label{thm.main}
Suppose that Assumption \ref{assump.maindegen} is in force. Then there is a constant $C$ such that
\begin{align*}
  -C r_{N,d} \leq  U(t,m_{\bx}^N) - V^N(t,\bx) \leq C N^{-\gamma(d)}, 
\end{align*}
for all $N \in \N$, $(t,\bx) \in [0,T] \times (\T^d)^N$,
where $\gamma_d$ is as in \eqref{gammadef} and $r_{N,d}$ is as in \eqref{rnddef}. If we assume in addition that either $\sigma$ is constant, or there is a constant $C$ such that 
\begin{align*}
|U(t,m) - U(t',m')| \leq C\Big(|t-t'| + \|m - m'\|_{-2,\infty} \Big),
\end{align*} then the upper bound can be improved to
\begin{align*}
    U(t,m_{\bx}^N) - V^N(t,\bx) \leq C N^{-\gamma'(d)},
\end{align*}
where $\gamma'$ is given by \eqref{gammaprimedef}.
\end{thm}

We also state the following result for the case $\sigma = 0$, which is obtained by combining Propositions \ref{prop.easy} and \ref{prop.zeronoise}. 

\begin{thm} \label{thm.zeroidio}
    Suppose that Assumption \ref{assump.maindegen} is in force, and in addition $\sigma = 0$. Then there is a constant $C$ depending only on the data such that
    \begin{align*}
        0 \leq V^N(t,\bx) - U(t,m_{\bx}^N) \leq C r_{N,d}
    \end{align*}
    for all $N \in \N$, $(t,\bx) \in [0,T] \times (\T^d)^N$,
where $r_{N,d}$ is as in \eqref{rnddef}.
\end{thm}

\section{The vanishing viscosity procedure} \label{sec: reg_values}

Let us fix $\eta \in [0,1]$. In this section, we introduce a regularized version of the problems presented in Subsection \ref{susec: prob_statement}. 

\subsection{Definition of $V^{N,\eta}$}
We assume in what follows that in addition to the Brownian motions $(W^i)_{i \in \N \cup \{0\}}$, the filtered probability space $\big(\Omega, \sF, \bbF, \bP\big)$ also hosts Brownian motions $(B^i)_{i \in \N}$ which are independent of each other and of $(W^i)_{i \in \N \cup \{0\}}$. We also fix a standard smooth mollifier $(\rho_{\eta})_{\eta > 0}$ on $\T^d$, and define 
\begin{align*}
    F^{\eta}(m) = F(m * \rho_{\eta}), \quad G^{\eta}(m) = G(m * \rho_{\eta}).
\end{align*}
The value function $V^{N,\eta}$ of the regularized problem will be defined exactly as $V^N$, except with the dynamics \eqref{nplayerdynamics} replaced by
\begin{align} \label{nplayerdynamics_eps}
    dX_t^{i} = \alpha_t^i dt + \sqrt{2\eta}dB_t^i + \sigma(X_t^{i}) dW_t^i + \sigma_0 dW_t^0, \,\, t_0 \leq t \leq T, \quad \bX_{t_0} = \bx_0,
\end{align}
and with $F,G$ replaced by $F^{\eta}, G^{\eta}$. 
More precisely, we define $V^{N,\eta}(t,\bx) : [0,T] \times (\T^d)^N \to \R$ via
\begin{align} \label{vnepsdef}
    &V^{N,\eta}(t_0,\bx_0) \coloneqq \inf_{\bm \alpha = (\alpha^1,...,\alpha^N) \in \sA^N_{t_0}} J^{N,\eta}(t_0,\bx_0,\bm \alpha), \text{ with } \nonumber 
    \\
    &J^{N,\eta}(t_0,\bx_0,\bm \alpha)  \coloneqq \E\bigg[\int_{t_0}^T \Big( \frac{1}{N} \sum_{i = 1}^N L\big(X_t^i, \alpha_t^i\big) + F^{\eta} \big( m_{\bX_t}^N ) \Big) dt + G^{\eta} \big(m_{\bX_T}^N \big)\bigg],
\end{align}
where in definition of $J^{N,\eta}$, $\bX$ is determined from $\bm \alpha$ by the dynamics \eqref{nplayerdynamics_eps}. Under Assumption \ref{assump.maindegen}, $V^{N,\eta}$ is the unique classical solution to the equation
\begin{align} \tag{$\text{HJB}_{N,\eta}$} \label{hjbn_eps}
    \begin{cases} 
   \ds  - \partial_t V^{N,\eta} - \eta \sum_{i = 1}^N \Delta_{x^i} V^{N,\eta} - \sum_{i = 1}^N \tr\big( A(x^i) D^2_{x^ix^i} V^{N,\eta}\big) - A_0 \sum_{i,j = 1}^N \tr(D^2_{x^ix^j} V^{N,\eta})
  \vspace{.1cm} \\ \ds \qquad + \frac{1}{N} \sum_{i = 1}^N H(x^i, ND_{x^i} V^{N,\eta} ) = F^{\eta}(m_{\bx}^N),  \qquad (t,\bx) \in [0,T) \times (\T^d)^N, \vspace{.1cm} \\
   \ds V^{N,\eta}(T,\bx) = G^{\eta}(m_{\bx}^N), \quad \bx \in (\T^d)^N.
    \end{cases}
\end{align}

\subsection{Definition of $U^{\eta}$} \label{subsec:uepsdef}

We again work on a canonical space to define $U^{\eta}$. In particular, for $t_0 \in [0,T)$, we set $\ov{\Omega}^{\infty,t_0} = \Omega^{\infty,t_0} \times \cC^{t_0}$, i.e. $\ov{\Omega}^{\infty,t_0} = \T^d \times (\cC^{t_0})^{\otimes 3}$. For $m_0 \in \cP(\T^d)$, we define $\ov{\bP}^{\infty, t_0, m_0}$ to be the the measure on $\ov{\Omega}^{\infty,t_0}$ given by $\ov{\bP}^{\infty, t_0, m_0} = m_0 \otimes (\bW^{t_0})^{\otimes 3}$, where $\bW^{t_0}$ again denotes the Wiener measure on $\cC^{t_0}$. We write $(x, \omega, \omega^0, \beta)$ for the general element of $\ov{\Omega}^{\infty, t_0}$, and write $(\xi, W, W^0, B)$ for the canonical random variable and processes on $\ov{\Omega}^{\infty, t_0}$. We let $\ov{\bbF}^{\infty, t_0, m_0} = (\ov{\sF}_t^{\infty,t_0, m_0})_{t_0 \leq t \leq T}$ denote the right-continuous and $\ov{\bP}^{\infty, t_0, m_0}$-complete augmentation of the filtration generated by $\xi$, $W$, $W^0$, and $B$ and let $\ov{\bbF}^{0,t_0,m_0} = (\ov{\sF}_t^{0,t_0,m_0})_{t_0 \leq t \leq T}$ denote the right-continuous and $\ov{\bP}^{\infty, t_0, m_0}$-complete augmentation of the filtration generated by $W^0$. We define $\ov{\sA}_{t_0,m_0}$ analogously to $\sA_{t_0,m_0}$. Then, as in the definition of $U$, we define 
\begin{align} \label{uepsdefcontrol}
    &U^{\eta}(t_0,m_0) \coloneqq \inf_{\alpha \in \ov{\sA}_{t_0,m_0}} J^{\infty, \eps}(t_0,m_0,\alpha), \text{ where } \nonumber 
    \\
    &J^{\infty,\eta}(t_0,m_0,\alpha) \coloneqq \E^{\ov{\bP}^{\infty, t_0, m_0}}\bigg[ \int_{t_0}^T \Big( L\big(X_t, \alpha_t \big) + F(\sL^{0,m_0}(X_{t})) \Big) dt + G\big(\sL^{0,m_0}(X_{T})\big)\bigg],
\end{align}
where in the definition of $J^{\infty,\eta}$, $X$ is determined from $\alpha$ by the dynamics
\begin{equation}
\label{dyn:Xeps:alpha}
    dX_t = \alpha_t dt + \sigma(X_t) dW_t + \sigma_0 dW_t^0 + \sqrt{2\eta} dB_t, \quad t_0 \leq t \leq T, \quad X_{t_0} = \xi,
\end{equation}
and where for a random variable $X$, $\sL^{0,m_0}$ denotes the law of $X$ conditional on $\ov{\sF}^{0,t_0,m_0}_T$ and with respect to $\ov{\bP}^{\infty, t_0, m_0}$. 

For $\eta \in (0,1]$, we are going to use the fact that $U^{\eta}$ can be equivalently defined in terms of regular feedback controls, which means that we can view the mean field control problem as the optimal control of a corresponding stochastic Fokker-Planck equation. To make this precise, we define the canonical space $\Omega^{0} = \cC = C([0,T] ; \R^d)$, and denote by $\bP^{0}$ the Wiener measure on $\Omega^{0}$. Furthermore we write $\omega^0$ for the general element of $\Omega^{0}$, $W^0$ for the canonical process on $\Omega^{0}$ and define $\bbF^{0,t_0} = (\sF^{0}_t)_{t_0 \leq t \leq T}$ to be the $\bP^{0}$-complete and right-continuous augmentation of the filtration generated by $(W^0_t - W^0_{t_0})_{t_0 \leq t \leq T}$. For $t_0 \in [0,T]$, we define $\sA_{t_0}^f$ to be the set of random fields
\begin{align*}
    \alpha_t(x) : [t_0,T] \times \T^d \times \Omega^{0,t_0} \to \R^d
\end{align*}
such that $\alpha$ is progressively measurable with respect to $\bbF^{0,t_0}$, and is bounded and Lipschitz in $x$, uniformly in $(t,\omega^0)$. By \cite{COGHI2019259} or \cite{CardaliaguetDelarueLasryLions} Theorem 4.3.1, for each $(t_0,m_0) \in [t_0,T] \times \cP(\T^d)$ and $\alpha 
 \in \sA_{t_0}^f$, the stochastic Fokker-Planck equation
\begin{align} \label{feedbackdynamics}
    &dm_t = \Big(  \sum_{i,j = 1}^d D^2_{x_i x_j}(A_{ij}(x) m_t) + (A_0 + \eta) \Delta m_t - \text{div}(m_t \alpha_t) \Big) dt \nonumber  \\
    & \qquad \qquad \qquad  - \sigma^0 \text{div}(m_t  dW_t^0), \quad t_0 \leq t \leq T, \quad m_{t_0} = m_0.
\end{align}
has a unique strong solution (defined on the space $\Omega^{0}$ and progressively measurable with respect to $\bbF^{0,t_0}$). By a strong solution, we mean a continuous,  $\bbF^{0,t_0}$-adapted, $\cP(\T^d)$-valued process $m = (m_t)_{t_0\leq t \leq T}$ such that for each $\phi \in C^{\infty}([t_0,T] \times \T^d)$, the following equation holds for each $t \in (t_0,T)$, with probability one:
\begin{align*}
    &\int \phi(t,x) m_t(dx) - \int \phi(t_0,x) m_0(dx) =
    \\
    & \qquad \qquad  \int_{t_0}^t \int_{\R^d} \Big( \partial_t \phi +  \tr\bigr( A(x) D^2 \phi(s,x) \big) + (A_0 + \eta) \Delta \phi(s,x) + \alpha_s(x) \cdot D\phi(s,x) \Big)  m_s(dx) ds
    \\
    &\qquad \qquad + \sigma^0 \int_{t_0}^t \int_{\R^d} D \phi(s,x) m_s(dx) \cdot dW_s^0.
\end{align*}
We define the feedback value function $U^{\eta, f} : [t_0,T] \times \cP(\T^d) \to \R$ via
\begin{align} \label{uepsfbdef}
    &U^{\eta,f}(t_0,m_0) \coloneqq \inf_{\alpha \in \sA_{t_0}^f} J^{\infty, \eta,f}(t_0,m_0,\alpha), \text{ with } \nonumber 
    \\
    &J^{\infty,\eta,f}(t_0,m_0,\alpha) \coloneqq \E^{\bP^{0}}\bigg[ \int_{t_0}^T \Big(\int_{\T^d} L\big(x, \alpha_t(x)\big) m_t(dx) + F^{\eta}(m_t) \Big) dt + G^{\eta}\big(m_T \big) \bigg],
\end{align}
where in the definition of $J^{\infty,\eta,f}$, $m$ is determined from $\alpha$ by the dynamics \eqref{feedbackdynamics}. Formally, we can expect $U^{\eta} = U^{\eta, f}$ because if $\alpha \in \ov{\sA}_{t_0,m_0}$ happens to take the form $\alpha_t = \alpha_t(X_t)$ for some $\ov{\bbF}^{0,t_0,m_0}$-adapted random field $\alpha$ with sufficient regularity properties, then the process $m_t = \sL^{0,m_0}(X_t)$ satisfies \eqref{feedbackdynamics} (on the space $\ov{\Omega}^{\infty,t_0}$ rather than $\Omega^0$). We will confirm in Lemma \ref{prop.closedloopequiv} that this is true for $\eta > 0$ using some technical results from \cite{Djete:COCV}, but we emphasize that Corollary \ref{cor.feedback} below shows that an analogous result cannot be expected for $\eta = 0$.

\subsection{Estimates on $V^{N,\eta}$}

The aim of this subsection is to prove some regularity properties of $V^{N,\eta}$, which are summarized in the following Proposition.
\begin{prop}\label{prop: bound_der_Veps}
	Let Assumption \ref{assump.maindegen} hold, and suppose in addition that $F$ and $G$ are $\cC^2$. Then there exists a constant $C$ depending only on the data, and a constant $C'$ depending on the data as well as $\eta$, $\|D^2_{mm} F\|_{\infty}$ and $\|D^2_{mm} G\|_{\infty}$, such that for each $\eta \in (0,1]$ and each $N \in \N$, we have
	\begin{align}
	    \max_{i = 1,...,N} \|D_{x^i} V^{N,\eta}\|_{\infty} &\leq C/N, 
     \label{eq:decay:der:1:N}
     \\
       \max_{i,j = 1,...,N} \|D^2_{x^ix^j} V^{N,\eta}\|_{\infty} &\leq \frac{C}{N\eta} + \frac{C'}{N^2}, \label{eq:decay:der:2:eps:N} \\
       \|\partial_t V^{N,\eta}\|_{\infty} &\leq C/\eta + C'/N.
       \label{eq:decay:Lipt:VNeps}
	\end{align}
 As a consequence, 
 for every $\bx,\by\in(\T^d)^N$ and $0 \leq s < t \leq T$ we have 
 \begin{align}
		\lvert V^{N,\eta}(t,\bx) - V^{N,\eta}(t,\by) \rvert &\leq\frac{C}{N}\sum_{i = 1}^N\lvert x^i - y^i\rvert,
  \label{Lip:VNeps}\\
     |V^{N,\eta}(t,\bx) - V^{N,\eta}(s,\bx)| &\leq C \sqrt{t-s}. 
     \label{Holder:time:VNeps}
 \end{align}
\end{prop}

\begin{proof}
We will write the proof in the case $d = 1$ for notational simplicity (since we will need to differentiate twice the equation \eqref{hjbn_eps}) but the argument is the same when $d > 1$. We also assume in the proof that $V^{N,\eta}$ is $C^{\infty}$. This assumption is easily removed at the end by a mollification procedure which we omit. 
\newline \newline \noindent 
\underline{Proof of \eqref{eq:decay:der:1:N}.} We fix an $N \in \N$ and $\eta \in (0,1]$, and for each $i = 1,...,N$, we set $y^i = D_{x^i} V^{N,\eta}$. By explicit computation, we have 
\begin{align}
 \label{eq:der:VNeps}
    - \partial_t y^i - \cL^{N,\eta} y^i - D_x A(x^i) D_{x^i} y^i + F^{N,i} = 0, \quad y^i(T,\bx) = G^{N,i}(\bx), 
\end{align}
where $\cL^{N,\eta}$ is the differential operator which acts (recall that we are assuming $d = 1$ for simplicity) on $\phi : \T^N \to \R$ via
\begin{align*}
    \cL^{N,\eta} \phi = \sum_{k = 1}^N (\eta + A(x^k)) D_{x^kx^k} \phi + A_0 \sum_{k,l = 1}^N D_{x^kx^l} \phi - \sum_{k = 1}^N D_pH(x^k,ND_{x^k}V^{N, \eta}) D_{x^k} \phi,
\end{align*}
and
\begin{align*}
    &F^{N,i}(t,\bx) = \frac{1}{N} D_{x} H(x^i,N D_{x^i} V^{N,\eta}) - \frac{1}{N} D_m F^{\eta} (m_{\bx}^N,x^i), \quad G^{N, i}(t,\bx) = \frac{1}{N} D_m G^{\eta}(m_{\bx}^N,x^i)
\end{align*}
satisfy 
\begin{align} \label{fnibound}
| F^{N, i}(t,\bx) | \leq C/N + C |D_{x^i} V^{N,\eta}|, \quad  \|G^{N,i}\|_{\infty} \leq C/N
\end{align}
with $C$ depending only on the data,
thanks to Assumption \ref{assump.maindegen}. To get the bound on $D_{x^i} V^{N,\eta}$, we fix $(t_0,\bx_0) \in [0,T] \times (\T^d)^N$, and we define a $\T^N$ valued process $\bX = (X^1,...,X^N)$ on $[t_0,T]$ by the dynamics 
\begin{align}
\label{eq:optimal:X}
d X^i_t = - D_p H( X^i_t, N D_{x_i} V^{N,\eta})dt +\sqrt{2\eta}dB^i_t 
+ \sigma(X^i_t) dW^i_t + \sigma_0 dW^0_t, \quad t_0 \leq t \leq T, \quad \bX_{t_0} = \bx_0.
\end{align}
We then define $Y_t^i = y^i(t,\bX_t)$ and $Z_t^{i,j} = D_{x^j} y^i(t,\bX_t) = D^2_{x^ix^j} V^{N,\eta}(t,\bX_t)$.
By It\^o's formula, we find that
\begin{align*}
    dY^i_t &= - \Big(\sigma(X^i_t) D_x \sigma(X^i_t) Z^{i,i}_t + F^{N,i}(t, \bm{X}_t) \Big) dt
     + \sum_{j=1}^N Z^{i,j}_t \Big( \sqrt{2\eta}dB^j_t 
+ \sigma(X^i_t) dW^j_t + \sigma_0 dW^0_t \Big),
\end{align*}
where we used that $A=\tfrac12 \sigma^2$. 
Computing the square we obtain 
\begin{align*}
    d |Y^i_t|^2 = 
    \bigg[ - 2Y^i_t (\sigma(X^i_t) D_x \sigma(X^i_t) Z^{i,i}_t + F^{N,i}(t, \bm{X}_t) ) 
    + \sum_{j=1}^N |Z^{i,j}_t|^2 ( \sigma^2 (X^j_t)+2\eta) 
    + \sigma_0^2 \Big(\sum_{j=1}^N Z^{i,j}_t \Big)^2 \bigg]dt +dM^i_t,
\end{align*}
where $M^i_t$ is a martingale. Computing the integral from $t_0$ to $T$, taking expectation, and erasing the positive terms, we have
\begin{align*}
    |Y^i_{t_0}|^2 &+ 2\eta \E \int_{t_0}^T \sum_{j=1}^N |Z^{i,j}_t|^2 dt
    + \E \int_{t_0}^T \sigma^2(X^i_t) |Z^{i,i}_t|^2 dt \\
    &\leq \E|G^{N, i}(\bm{X}_T)|^2 + 2 \E \int_{t_0}^T  Y^i_t \big(\sigma(X^i_t) D_x \sigma(X^i_t) Z^{i,i}_t +F^{N,i}(t, \bm{X}_t)\big) dt \\
    &\leq \frac{C}{N^2} 
    + \E \int_{t_0}^T \sigma^2(X^i_t) |Z^{i,i}_t|^2 dt + C\E \int_{t_0}^T |Y^i_t|^2 dt, \\
\end{align*}
where we applied Young's inequality and the bounds on $F^{N,i}, G^{N,i}$. Rearranging and using Gronwall's lemma, we get
\be 
\label{eq:energy}
\|Y_{t_0}^i\|_{\infty}^2 + \eta \E \int_{t_0}^T \sum_{j=1}^N |Z^{i,j}_t|^2 dt \leq \frac{C}{ N^2},
\ee 
with $C$ depending only on the data. In particular, recalling the definition of $Y^i$, we see that by taking a supremum over $i$ and over initial conditions $(t_0,\bx_0)$ we get \eqref{eq:decay:der:1:N}.
\newline \newline \noindent 
\underline{Proof of \eqref{eq:decay:der:2:eps:N}.}
We continue to assume that $F$ and $G$ are $\cC^2$. To derive \eqref{eq:decay:der:2:eps:N}, we differentiate again \eqref{eq:der:VNeps} with respect to $x^i$: denoting $z^{i,j}(t,\bm{x}) = D^2_{x^ix^j} V^{N,\eta}(t,\bm{x})$, we have
\begin{align*}
&-\partial_t z^{i,j} -\mathcal{L}^{N,\eta} z^{i,j}  
- \sigma(x^i) D_x\sigma(x^i) D_{x^i} z^{i,j} - \sigma(x^j) D_x \sigma(x^j) D_{x^j} z^{i,j}
+ K^{N,i,j}(t,\bm{x}) z^{i,j}  \ds \\ \ds
&\qquad \qquad \qquad + F_2^{N,i,j}(t, \bm{x}) + N\sum_j D^2_{pp} H (x^j, N V^{N,j})  |D_{x^j} V^{N,i}|^2
 =0, 
 \end{align*}
and $z^{i,j}(T,\bm{x}) = G^{N,i,j}_2(\bm{x})$,
where we set
\begin{align*}
    K^{N,i,j}(t,\bm{x}) &:= - D^2_{xx} A(x^i) 1_{i = j}
    + \big(D_{xp} H (x^i, N V^{N,i}(t,\bm{x})) + D_{xp} H (x^j, N V^{N,j}(t,\bm{x})) \big)  ,\\
    F^{N,i,j}_2(t,\bm{x}) &:= 
    \frac1N D^2_{xx} H (x^i, N V^{N,i}(t,\bm{x})1_{i = j} - \frac{1}{N} D_{xm} F^{\eta}(m_{\bx}^N, x^i)1_{i = j} - \frac{1}{N^2} D^2_{mm} F^{\eta}(m_{\bx}^N,x^i,x^j), \\
    G^{N,i,j}_2(\bm{x})&:= 
    \frac1N D_{xm} G^{\eta}( m^N_{\bm{x}}, x^i)1_{i = j} + \frac{1}{N^2} D^2_{mm} G^{\eta} (m^N_{\bm{x}}, x^i,x^j) .
\end{align*}
Notice that $G^{\eta}$ and $F^{\eta}$ are $C/\eta$-Lipschitz with respect to $W^{-2,\infty}$, and so by Lemma \ref{lem.regfromtouching}, $\|D_{xm} F^{\eta}\|_{\infty} + \|D_{xm} G^{\eta}\|_{\infty} \leq C/\eta$.  Using this and the gradient bound already established, we get
\begin{align*}
   \| K^{N,i} \|_{\infty} \leq C, \quad \| G^{N,i}_2 \|_{\infty} \leq C/(N\eta) + C\|D^2_{mm} G^{\eta}\|{\infty}/N^2, \quad \| 
   F^{N,i}_2 \|_{\infty} \leq C/(N\eta) + C\|D^2_{mm} F^{\eta}\|_{\infty}/N^2, 
\end{align*} with $C$ depending only on the data. We define $Q^{i,j,k}_t = D_{x^k} z^{i,j}(t,\bm{X}_t) = D^3_{x^ix^jx^k} V^{N,\eta}(t,\bm{X}_t)$,
and apply It\^o's formula to $z^{i,j}(t,\bm{X}_t) = Z^{i,j}_t$ to get 
\begin{align*}
    dZ^{i,j}_t &= \Big\{- \sigma(X^i_t) D_x\sigma(X^i_t) Q^{i,j,i}_t - \sigma(X^i_t) D_x\sigma(X^j_t) Q^{i,j,j}_t 
+ K^{N,i}(t,\bm{X}_t) Z^{i,j}_t  \\
&\qquad + F_2^{N,i,j}(t, \bm{X}_t) + N\sum_k D^2_{pp} H (X^k_t, N Y^k_t)  Z^{i,k}_t Z^{j,k}_t
 \Big\} dt \\
& \qquad +\sum_{k} Q^{i,j,k}_t \Big( \sqrt{2\eta}dB^k_t 
+ \sigma(X^k_t) dW^k_t + \sigma_0 dW^0_t \Big).
\end{align*}
Computing the square we find
\begin{align*}
    d|Z^{i,j}_t|^2 &= 2Z^{i,i}_t \big\{- \sigma(X^i_t) D_x\sigma(X^i_t) Q^{i,j,i}_t - \sigma(X^j_t) D_x\sigma(X^j_t) Q^{i,j,j}_t  
+ K^{N,i,j}(t,\bm{X}_t) Z^{i,j}_t + F_2^{N,i,j}(t, \bm{X}_t) \\
& + N\sum_j D^2_{pp} H (X^k_t, N Y^k_t)  Z_t^{i,k} Z_t^{j,k}
\Big\} dt \\
& + \sum_{k=1}^N |Q^{i,j,k}_t|^2 ( \sigma^2 (X^k_t)+2\eta) 
    + \sigma_0^2 \Big(\sum_{k=1}^N Q^{i,j,k}_t \Big)^2 \Big]dt +dM^i_t.
\end{align*}
Setting $\|z^{i,j}(t)\|_{\infty} = \|z^{i,j}(t,\cdot)\|_{\infty}$, and recalling that $Z_t^{i,j} = z^{i,j}(t,\bX_t)$, we obtain (we recall that $C$ denotes any constant independent of $\eta, N$ and $t_0, \bm{x}$, which is allowed to change from line to line, and that we fix a deterministic initial condition $Z^{i,i}_{t_0} = z^i(t_0,\bm{x}_0)$ ) 
\begin{align*}
    |Z^{i,j}_{t_0}|^2 &+ 2\eta \E \int_{t_0}^T \sum_{k=1}^N |Q^{i,j,k}_t|^2 dt
    + \E \int_{t_0}^T \sum_{k = 1}^N \sigma^2(X^k_t) |Q^{i,j,k}_t|^2 dt \\
    & \leq  \|G_2^{N,i,j}\|_{\infty}^2  + 
    C\E\int_{t_0}^T |Z^{i,j}_t|^2dt
    + \E \int_{t_0}^T \sigma^2(X^i_t) |Q^{i,j,i}_t|^2 dt  +  \E \int_{t_0}^T \sigma^2(X^j_t) |Q^{i,j,j}_t|^2 dt\\
    &\qquad + CN  \sup_{t\in[t_0,T]} ||z^{i,j}(t)||_\infty
    \Big(\E \Big[ \int_{t_0}^T  \sum_k  |Z^{i,k}_t|^2 dt \Big]\Big)^{1/2}\Big(\E \Big[ \int_{t_0}^T  \sum_k  |Z^{j,k}_t|^2 dt \Big]\Big)^{1/2} + C \|F^{N,i,j}_2\|^2_{\infty}  \\
    &\leq \frac{C}{N^2\eta } + C \Big(\frac{\|D^2_{mm} F^{\eta}\|^2 + \|D^2_{mm} G^{\eta}\|^2}{N^4}\Big)+ C\E\int_{t_0}^T ||z^{i,j}(t)||_\infty^2dt 
    \\
    &\qquad 
    + \E \int_{t_0}^T \sigma^2(X^i_t) |Q^{i,j,i}_t|^2 dt  +  \E \int_{t_0}^T \sigma^2(X^j_t) |Q^{i,j,j}_t|^2 dt
    + \frac{C}{N\eta} \sup_{t\in[0,T]} ||z^{i,j}(t)||_\infty,
    \end{align*}
where we employed the estimate \eqref{eq:energy} and the bounds on $F_2^{N,i,j}$ and $G_2^{N,i,j}$.
Therefore taking the supremum over the initial conditions $(t_0,\bx_0)$, we find 
\[
||z^{i,j}(t_0)||^2_\infty \leq  C\E\int_{t_0}^T ||z^{i,j}(t)||_\infty^2dt
+ \frac{C}{N^2 \eta } + C \Big(\frac{\|D^2_{mm} F^{\eta}\|^2 + \|D^2_{mm} G^{\eta}\|^2}{N^4}\Big) + \frac{C}{N\eta}  \sup_{t\in[0,T]} ||z^{i,j}(t)||_\infty 
\]
and thus Gronwall's lemma yields 
\[
\sup_{t\in[0,T]} ||z^{i,j}(t)||^2_\infty \leq  \frac{C}{N^2 \eta } + C \Big(\frac{\|D^2_{mm} F^{\eta}\|^2 + \|D^2_{mm} G^{\eta}\|^2}{N^4}\Big) + \frac{C}{N\eta}  \sup_{t\in[0,T]} ||z^{i,j}(t)||_\infty 
\]
for another value of $C$. Hence Young's inequality gives 
\[
\sup_{t\in[0,T]} ||z^{i,j}(t)||^2_\infty \leq  \frac{C}{N^2\eta^2} + C \Big(\frac{\|D^2_{mm} F^{\eta}\|^2 + \|D^2_{mm} G^{\eta}\|^2}{N^4}\Big) + \frac12 \sup_{t\in[0,T]} ||z^{i,j}(t)||^2_\infty,
\]
which establishes 
\begin{align*}
    \|D^2_{x^ix^j} V^{N,\eta}\|_{\infty} \leq \frac{C}{N\eta} + C \Big(\frac{\|D^2_{mm} F^{\eta}\| + \|D^2_{mm} G^{\eta}\|}{N^2}\Big),
\end{align*}
i.e. \eqref{eq:decay:der:2:eps:N} holds with $C' = C ( \|D^2_{mm} F^{\eta}\| + \|D^2_{mm} G^{\eta}\| )$.
\newline \newline \noindent 
\underline{Proofs of \eqref{eq:decay:Lipt:VNeps} and \eqref{Lip:VNeps}.} The bound on $\partial_t V^{N,\eta}$ is a simple consequence of the PDE and the bounds already obtained for $D_{x^i} V^{N,\eta}$ and $D^2_{x^ix^j} V^{N,\eta}$, and \eqref{Lip:VNeps} is a simple consequence of the bounds on $D_{x^i} V^{N,\eta}$ and the symmetry of $V^{N,\eta}$. 
\newline \newline \noindent 
\underline{Proof of \eqref{Holder:time:VNeps}}. It follows as usual from the dynamic programming principle and the Lipschitz estimate \eqref{Lip:VNeps}. Fix an initial time $t\in[0,T)$, and initial position (deterministic) $\bm{x}$, and another time $s\in(t,T]$, and let $\bm{X}$ be the optimal trajectory, given by \eqref{eq:optimal:X} above, starting at $\bm{X}_t=\bm{x}$. We write
\begin{align*}
    V^{N,\eta}(s,\bm{x}) - V^{N,\eta}(t,\bm{x}) 
    = \E \big[ V^{N,\eta}(s, \bm{X}_s)\big] - V^{N,\eta}(t,\bm{x}) 
    + V^{N,\eta}(s,\bm{x}) - \E \big[ V^{N,\eta}(s, \bm{X}_s)\big].     
\end{align*}
The second term is estimated by using the estimate $\E |X^i_s- x^i|^2 \leq C(s-t)$
which follows from the boundedness of $\sigma$ and the drift, as $N D_{x_i}V^{N,\eta}$ is bounded by \eqref{eq:decay:der:1:N}, so that the
Lipschitz continuity \eqref{Lip:VNeps} gives 
\[
\E \big| V^{N,\eta}(s,\bm{x}) -  V^{N,\eta}(s, \bm{X}_s) \big| \leq 
\frac{C}{N} \sum_{i=1}^N \E |X^i_s- x^i| \leq C\sqrt{s-t}.
\]
 The first term is instead bounded by the dynamic programming principle: 
\begin{align*}
    \E \big[ V^{N,\eta}(s, \bm{X}_s)\big] - V^{N,\eta}(t,\bm{x}) 
    = \E \bigg[\int_{t}^s \frac{1}{N} \sum_{i = 1}^N \Big(L\big(X_r^i, \alpha_r^i\big) + F\big(m_{\bX_r}^N \big) \Big) dr \bigg] \leq C(s-t),
\end{align*}
 since $L$ is bounded, as it is continuous and the optimal control is bounded. 

\end{proof}

\subsection{Some technical lemmas}

We now give some technical lemmas about the control problems under consideration. We start with explaining how the bounds in Proposition \ref{prop: bound_der_Veps} imply that we can restrict our attention to bounded controls for the limiting problem.

\begin{lem} \label{lem.boundedcontrols}
    Let Assumption \ref{assump.maindegen} hold. Then there is a constant $R$ depending only on the data such that for any $\eta \in [0,1]$, any $(t_0,m_0) \in [0,T] \times \cP(\T^d)$, and any $\delta > 0$ there exists a $\delta$-optimal control for the problem defining $U^{\eta}(t_0,m_0)$ which is uniformly bounded by $R$.
\end{lem}
\begin{proof}
    Let us fix $R>0$ and define $U^{\eta,R}$ exactly like $U^{\eta}$, but with the infimum in \eqref{uepsdefcontrol} taken only over those $\alpha \in \ov{\sA}_{t_0,m_0}$ which take values in $B_R \subset \R^d$. Similarly we define $V^{N,\eta,R}$ like $V^{N,\eta}$, but with the infimum in \eqref{vnepsdef} restricted to those $\bm \alpha \in \sA_{t_0}^N$ with $\alpha_t^i$ taking values in $B_R$ for each $i = 1,...,N$. Thanks to the estimates obtained in Proposition \ref{prop: bound_der_Veps}, the optimal feedback for $V^{N,\eta}$ remains uniformly bounded, independently of $\eta$. Hence, there exists a $R_0>0$, depending only on the data, such that $V^{N,\eta,R} = V^{N,\eta}$ for all $R\geq R_0$. To conclude, we can can mimic the proof of Lemma 5.2 in \cite{cdjs2023}, and obtain that there exists a $R_0>0$, depending only on the data, such that $U^{\eta,R} = U^{\eta}$ for all $R\geq R_0$.
\end{proof}
\label{rem:bounded_controls}

The following lemma states that $V^{N,\eta} \to U^{\eta}$ qualitatively as $N \to \infty$, and is due to \cite[Thm. 3.6]{DjetePossamaiTan}.

\begin{lem} \label{lem.qualitativeconv}
    Let Assumption \ref{assump.maindegen} hold. For each $\eta \in [0,1]$, $V^{N,\eta} \to U^{\eta}$ 
    and also $V^{N} \to U$ qualitatively, in the sense that for each sequence $(t_N,\bx_N) \in [0,T] \times (\T^d)_N$ and $(t,m)  \in [0,T] \times \cP(\T^d)$ such that $(t_N,m_{\bx_N}^N) \to (t,m)$ as $N \to \infty$, we have 
    \begin{align}
        V^{N,\eta}(t_N,m_{\bx}^N) \to U^{\eta}(t,m), \quad 
        V^{N}(t_N,m_{\bx}^N) \to U(t,m).
    \end{align}
\end{lem}

\begin{lem} \label{lem.epstozero}
Let Assumption \ref{assump.maindegen} hold. There is a constant $C$ depending on the data such that for each $N \in \N$, $\eps \in [0,1]$, we have 
\begin{align} \label{conv:rate:eps}
   \|U^{\eta} - U\|_{\infty} \leq C\sqrt{\eta}, \quad  \|V^{N,\eta} - V^N\|_{\infty} \leq C \sqrt{\eta}. 
\end{align}
\end{lem}

\begin{proof}
We prove the estimate on $|V^{N,\eta} - V^N|$, from which \eqref{conv:rate:eps} follows by the qualitative convergence Lemma \ref{lem.qualitativeconv}. Since it is clear that $V^{N,0} = V^N$, we can instead estimate $|V^{N,0} - V^{N,\eta}|$. Fix $(t_0,\bx_0) \in [0,T] \times (\T^d)^N$, and an open-loop control $\bm \alpha = (\alpha^1,...,\alpha^N) \in \sA_{t_0}^N$. Let $\bX = (X^1,...,X^N)$ be defined by \eqref{nplayerdynamics} and $\bX^{\eta} = (X^{1,\eta},...,X^{N,\eta})$ be defined by \eqref{nplayerdynamics_eps}. It is clear that 
\begin{align*}
    \E\Big[\sup_{t_0 \leq t \leq T} |X_t^{\eta,i} - X_t^i|^2 \Big] \leq C \eta,  \text{  hence  }
    \E\Big[\sup_{t_0\leq t \leq T} \bd_1(m_{\bX_t}^N,m_{\bX_t^{\eta}}) \Big] \leq C \eta.
\end{align*}
Since $F$ and $G$ are $\bd_1$-Lipschitz, this implies that 
\begin{align*}
    |J^{N,\eta}(t_0,\bx_0,\bm \alpha) - J^{N}(t_0,\bx_0, \bm \alpha) | \leq C \sqrt{\eta}
\end{align*}
for a constant $C$ which depends only on the data. This implies \eqref{conv:rate:eps}, and completes the proof.
\end{proof}

\begin{lem} \label{prop.closedloopequiv}
    Let Assumption \ref{assump.maindegen} hold. Then for $\eta \in (0,1]$, $U^{\eta} = U^{\eta,f}$. Moreover, there is a constant $R$ depending only on the data such that for any $\eta \in (0,1]$, any $(t_0,m_0) \in (0,T] \times \cP(\T^d)$, and any $\delta > 0$ there exists a $\delta$-optimal control for the problem defining $U^{\eta,f}(t_0,m_0)$ which is uniformly bounded by $R$.
\end{lem}

\begin{proof} 
     
     For $R > 0$, define $U^{\eta,R}$ as in the proof of Lemma \ref{lem.boundedcontrols} and, similarly, define $U^{\eta,f,R}$ exactly like $U^{\eps,f}$ but with the infimum in \eqref{uepsfbdef} restricted to those $\alpha \in \sA_{t_0}^f$ which take values in $B_R \subset \R^d$. Thanks to Lemma \ref{lem.boundedcontrols}, we have $U^\eta= U^{\eta,R}$, for some fixed $R$ which depends only on the data. Moreover, we clearly have $U^\eta\leq U^{\eta,f}$ and $U^{\eta,R} \leq U^{\eta,f,R}$ by inclusion of the sets of controls, since every feedback control clearly gives rise to an open-loop control. 
     Next, \cite[Theorem 2.14 and Proposition 3.12]{Djete:COCV} can be used to deduce that $U^{\eta,f,R} \leq U^{\eta, R}$, since it explains that any  $\alpha \in \sA_{t_0,m_0}$ with values in $B_R$ can be approximated in an appropriate sense by a sequence of feedback controls. Note that $\sigma_0$ is assumed to be invertible in \cite{Djete:COCV}, but it is not needed in the first part of \cite[Proposition 3.12]{Djete:COCV}, in particular in Proposition A.5 therein, when a general open-loop control is approximated by the kind of feedback controls that we consider; only the non-degeneracy of the idiosyncratic noise is required, and $\sigma_0$ is required to be invertible, in that paper, for the approximation with controls which are functions of $(t,x,\mu)$. Finally, we conclude that $U^\eta = U^{\eta,R} = U^{\eta,f,R} \geq U^{\eta,f} \geq U^\eta$,
 which implies that $U^\eta=U^{\eta,f}$ and also $U^{\eta,f}=U^{\eta,f,R}$.
\end{proof}

\subsection{Estimates on $U^{\eta}$} \label{subsec:uepsestimates}

\begin{prop}\label{prop: d1_lip}
	Let  Assumption \ref{assump.maindegen} hold. Then we have
 \begin{align*}
    \lip(U; \bd_1) \leq C, \quad  \lip(U^{\eta}; \bd_1) \leq C, \quad \lip(U^{\eta}; W^{-2,\infty}) \leq C/\eta,
 \end{align*}
 for each $\eta \in (0,1]$ and a constant $C$ depending only on the data.
\end{prop}

The proof uses the mapping 
$\hat{V}^{N,\eta} \colon[0,T]\times \sP(\T^d)\to\R $ defined by
\begin{equation} \label{hatvndef}
	\hat{V}^{N,\eta}(t,m):=\int_{(\T^d)^N}V^{N,\eta}(t,x^1,\dots,x^N)\prod_{i=1}^N m(dx^i).
\end{equation}

\begin{proof}
   First, suppose in addition that $F$ and $G$ are $\cC^2$. The function $\hat{V}^{N,\eta}$ is smooth and its linear derivative is given by 
    \[
\frac{\delta \hat{V}^{N,\eta}}{\delta m} (t,m, y) = \sum_{i=1}^N\int_{(\T^d)^{N-1}}
V^{N,\eta}(t,(\bm{x}^{-i}, y))  m^{\otimes (N-1)} (d\bm{x}^{-i}).
    \]
    Thus the bounds \eqref{eq:decay:der:1:N} and \eqref{eq:decay:der:2:eps:N} give 
\begin{align*}
|D_m \hat{V}^{N,\eta}(t,m, y)| \leq  \sum_{i=1}^N \|D_{x_i} V^{N,\eta}\|_\infty \leq C, \quad 
|D_{xm} \hat{V}^{N,\eta}(t,m; y)| \leq  \sum_{i=1}^N \|D^2_{x_ix_i} V^{N,\eta}\|_\infty \leq \frac{C}{\eta} + C'/N,
\end{align*}
where $C$ depends only on the data but $C'$ can depend also on $\|D^2_{mm} F\|_{\infty}$ and $\|D^2_{mm} G\|_{\infty}$. 
This yields
\[
\sup_{t,m} \norm{\frac{\delta \hat{V}^{N,\eta}}{\delta m} (t,m, \cdot) }_{2,\infty} \leq \frac{C}{\eta} + C'/N. 
\]
The definition of the linear derivative \eqref{linderivative} implies first that 
\begin{align*}
|\hat{V}^{N,\eta}(t,m) - \hat{V}^{N,\eta}(t,m')| &=  \int_0^1 \int_{\T^d} \frac{\delta}{\delta m} \hat{V}^{N,\eta}\big(t, rm + (1-r)m',x \big)  (m-m')(dx) dr \\
& \leq \|D_m \hat{V}^{N,\eta}\|_\infty \bd_1(m',m) \leq C \bd_1(m',m),
\end{align*} 
which provides the $\bd_1$-Lipschitz bound since by Lemma \ref{lem.pointwisevnhat} below, $\hat{V}^{N,\eta}$ converges pointwise to $U^{\eta}$. Similarly, we have
\begin{align*}
|\hat{V}^{N,\eta}(t,m) - \hat{V}^{N,\eta}(t,m')| 
& \leq \sup_{m\in\mathcal{P}(\mathbb{T}^d)} \norm{\frac{\delta \hat{V}^{N,\eta}}{\delta m} (t,m, \cdot) }_{2,\infty} \|m'-m\|_{-2,\infty} 
\\
&\leq \Big(\frac{C}{\eta} + C'/N\Big) \|m'-m\|_{-2,\infty} 
\end{align*}
which gives the $W^{-2,\infty}$-Lipschitz bound. To complete the proof, we can remove the assumption that $F$ and $G$ are $\cC^2$ using a mollification procedure as in Lemma 4.1 of \cite{ddj2023}.
\end{proof}
The next lemma states that $\hat{V}^{N,\eta}$ converges to $U^{\eta}$ pointwise, and is a consequence of the qualitative convergence of $V^{N,\eta}$ to $U$ together with a bound on $|V^{N,\eta}(t,\bx) - \hat{V}^{N,\eta}(t,m_{\bx}^N)|$ which will be established below.
\begin{lem}
    \label{lem.pointwisevnhat}
    We have $\hat{V}^{N,\eta}(t,m) \to U^{\eta}(t,m)$ as $N \to \infty$ for each fixed $(t,m)$.
\end{lem}

\begin{proof}
   This follows from Lemma \ref{lem.qualitativeconv} and Lemma \ref{lemma: approx v_vhat} below.
\end{proof}

The next property of $U^\eta$ we are interested in is the semiconcavity with respect to the $\tv$-norm.
\begin{prop}\label{prop: TV_semiconc}
Let Assumption \ref{assump.maindegen} hold. Then, $U^\eta$ is $\tv$-semiconcave with a constant depending only on the data, i.e. there is a constant $C$ depending on the data such that 
 \begin{align*}
     U^\eta(t,\lambda m + (1-\lambda)m') \geq \lambda U^\eta(t,m) + (1-\lambda) U^\eta(t,m') - C \lambda(1-\lambda) \|m -m'\|^2_{\tv} 
 \end{align*}
 for each $t \in [0,T]$, $m,m' \in \cP(\T^d)$, $\lambda \in (0,1)$. As a consequence, the same bound holds also for $U$.
\end{prop}
Before proving Proposition \ref{prop: TV_semiconc}, we need to introduce a useful change of variable that allows us to deal with the common noise $W^0$ (see, for instance, the proofs of Lemma 4.1 in \cite{cdjs2023} and Proposition 6.1 in \cite{djs2023}). In particular, we are going to work with the feedback formulation, and instead of working process $m$ defined by \eqref{feedbackdynamics}, we are going to work with 
\begin{equation}\label{eqn: rel_mubar_mu}
	\ov{m}_t = (\text{Id} - \sigma^0 W_t^0 + \sigma^0 W_{t_0}^0)_{\#}m_t,\quad t_0\leq t\leq T.
\end{equation}
We note that (see the proof of Lemma 4.1 in \cite{cdjs2023}) if $m$ satisfies \eqref{feedbackdynamics}, then $\ov{m}$ satisfies the following Fokker-Planck equation with random coefficients:
\begin{equation}\label{eqn: fp_rand_coeff}
    \begin{cases}
        \partial_t \bar{m}_t = \eta \Delta \bar{m}_t + \sum_{i,j = 1}^d D^2_{x_ix_j}\big( A_{ij} (x + W_t - W_t^0) \bar{m}_t \big)- \div \left (\alpha_t(x + W_t - W_t^0)\bar{m}_t \right),\quad t_0\leq t\leq T, \\
        \quad \bar{m}_{t_0} = m_0.
    \end{cases}
\end{equation}

\begin{lem}\label{lemma: stability_mubar}
	Let Assumption \ref{assump.maindegen} hold, and let us fix $t_0\in [0,T]$, $m_0$, $m_0' \in \cP(\T^d)$, and $\alpha \in \sA_{t_0}^f$. Let $m,m'$ be defined by the dynamics \eqref{feedbackdynamics}, with initial conditions $m_{t_0} = m_0$, $m'_{t_0} = m_0'$. Then we have
	\begin{equation*}
		\Big\| \sup_{t_0 \leq t \leq T} \| m_t - m_t' \|_{\tv} \Big\|_{\linf(\Omega^{0})} \leq \|m_0 - m_0'\|_{\tv}.
	\end{equation*}
\end{lem}
\begin{proof}
    Let $\ov{m}$ be defined by \eqref{eqn: rel_mubar_mu} and $\ov{m}'$ be defined analogously but with $m'$ replacing $m$. Clearly $\|\ov{m}_t - \ov{m}'_t\|_{\tv} = \|m_t - m'_t\|_{\tv}$, so it suffices to prove the estimate with $\ov{m}$ and $\ov{m}'$. But as discussed above, for almost every $\omega^0$, $t \mapsto \ov{m}_t(\omega^0)$ is the unique weak solution of the Fokker-Planck equation \eqref{eqn: fp_rand_coeff}. We claim that, for each fixed $\omega^0$, this equation generates a contraction in the $\tv$ norm. Indeed, suppose that we have $m^1$ and $m^2$ solving
    \begin{align*}
        \partial_t m^i = \sum_{j,k} D_{x_j x_k}( A_{jk}(t,x) m^i) - \text{div}(m^i B(t,x)), \quad t_0 \leq t \leq T, \quad m^i_{t_0} = m^i_0
    \end{align*}
    for $i = 1,2$,
    with $A(t,x) : [t_0,T] \times \T^d \to \text{Sym}(\R^{d \times d})$ Lipschitz in $x$ and uniformly non-degenerate and $B(t,x) : [t_0,T] \times \T^d \to \R^d$ Lipschitz in $x$. Then, by duality we have, for each $\phi : \T^d \to \R$ with $\|\phi\|_{\infty} \leq 1$ and any $t_1 \in (t_0,T]$, 
    \begin{align} \label{duality}
        \int \phi d(m_{t_1}^1 - m_{t_1}^2) = \int \Phi(t_0,\cdot) d(m^1_0- m^2_0), 
    \end{align}
    where $\Phi : [t_0,t_1] \times \T^d \to \R$ solves 
    \begin{align*}
        - \partial_t \Phi - \tr ( A D^2 \Phi) - B \cdot D \Phi = 0, \quad t_0 \leq t \leq t_1, \quad \Phi(t_1,x) = \phi(x).
    \end{align*}
    By the maximum principle $\|\Phi \|_{\infty} \leq \|\phi\|_{\infty}$, and together with \eqref{duality} this easily implies that 
    \begin{align} \label{dettvest}
        \|m_{t_1}^1 - m_{t_1}^2\|_{\tv} \leq \|m_{0}^1 - m_0^2\|_\tv.
    \end{align}  
    Coming back to $\ov{m}$ and $\ov{m}'$, we apply the estimate \eqref{dettvest} to conclude that for almost ever $\omega^0$, 
    \begin{align*}
        \sup_{t_0 \leq t \leq T} \|\ov{m}_{t}(\omega^0) - \ov{m}'_t(\omega^0)\|_{\tv} \leq \| m_{0} - m_0'\|_{\tv},
    \end{align*}
    which completes the proof.
    \end{proof}

\begin{proof}[Proof of Proposition \ref{prop: TV_semiconc}]
For any $\lambda\in[0,1]$ and $m_0,m'_0 \in\sP(\T^d)$, let us set $m_0^\lambda:=\lambda m_0 + (1-\lambda)m'_0$. Let us fix $\delta>0$. By Lemma \ref{lem.boundedcontrols}, we can choose $\alpha \in \sA_{f}^{t_0}$ to be a $\delta$-optimal feedback control for the problem starting from $(t_0, m^\lambda_0)$ which is bounded by some $R > 0$ depending only on the data. We denote by $m$ and $m'$ the solutions of \eqref{feedbackdynamics} with initial conditions $m_{t_0} = m_0$ and $m_{t_0}' = m_0'$. Notice that by linearity $m_{\lambda} = \lambda m + (1-\lambda)m'$. Then, by the $\tv$-semiconcavity of $F$ (hence $F^{\eta}$) and $G$ (hence $G^{\eta}$), we have
\begin{align*}
	U^\eta(t_0,m_0^{\lambda}) &\geq \E^{\bP^0}\left[\int_{t_0}^T \Big( \int_{\T^d} L\big(x,\alpha_t(x)\big) m_t^{\lambda}(dx) + F^{\eta} \big(m^{\lambda}_t \big) \Big)dt + G^{\eta}(m^{\lambda}_T)\right] - \delta\\
	& \geq  \lambda \E^{\bP^0}\left[\int_{t_0}^T \Big(\int_{\T^d} \Big( L(x,\alpha_t(x)) dm_t(x) + F^{\eta}(m_t) \Big) dt + G^{\eta}(m_T)\right] \\
	&\quad + (1 - \lambda) \E^{\bP^0}\left[\int_{t_0}^T \Big( \int_{\T^d} L(x,\alpha_t(x)) dm_t'(x) + F^{\eta}(m_t') \Big) dt + G^{\eta}(m_T')\right] \\
	& \quad - \E^{\bP^0}\left[\int_{t_0}^T C_L\frac{\lambda(1-\lambda)}{2}\norm{m_t - m'_t}_\tv dt  + C_G\frac{\lambda(1-\lambda)}{2}\norm{m_T - m'_T}_\tv\right] - \delta\\
	&\geq \lambda U^\eta (t_0,m) + (1-\lambda) U^\eta(t_0,m')\\
	&\quad - \E^{\bP^0}\left[\int_{t_0}^T C_L\frac{\lambda(1-\lambda)}{2}\norm{m_t - m'_t}_\tv dt 
	+ C_G\frac{\lambda(1-\lambda)}{2}\norm{m_T - m'_T}_\tv\right] - \delta.
\end{align*}
Applying Lemma \ref{lemma: stability_mubar} and then sending $\delta \to 0$ completes the proof.
\end{proof}
\begin{rmk} \label{rmk.semiconcavity}
   Note that, unlike \cite{ddj2023}, we work with $\tv$-semiconcavity rather than $\bd_1$-semiconcavity. This is because in order to obtain a $\bd_1$-semiconcavity estimate, we would need to use a stability estimate like in Lemma \ref{lemma: stability_mubar} for $\delta$-optimal feedback controls, but with respect to $\bd_1$ rather than $\tv$. But such an estimate would depend on the Lipchitz constant of $\alpha$, and one would expect that since $U$ is only $\bd_1$-Lipschitz and (formally) the optimal feedback control is given by $\alpha_t(x) = - D_p H(x, D_m U(t,m_t,x))$, the Lipschitz constant of the optimal feedback $\alpha$ should blow up as $\eta \to 0$. Thus while our $\tv$-semiconcavity estimate does not depend on $\eta$ (and this is important in Section \ref{sec:hardinequalityconditional}), we do not see a way to get a $\bd_1$-semiconcavity bound independent of $\eta$. Luckily, the $\tv$-semiconcavity is sufficient for our purposes.
\end{rmk}
\section{The easy inequality} \label{sec:easy}
The goal of this section is to obtain the lower bound for the inequality stated in Theorem \ref{thm.main}.
To this end, we will first prove the relation above for the \textit{regularized} value functions $V^{N,\eta}$ and $U^\eta$ introduced in Section \ref{sec: reg_values}. Then, by taking $\eta \to 0^+$ we will retrieve $V^{N}(t,\bx) \leq U(t,m_{\bx}^N) + Cr_{N,d}$. A key point in this procedure is the fact that all the constants involved in our estimate are independent of $\eta>0$.\\

Let us fix $\eta \in (0,1]$ and let us recall that $V^{N,\eta}$ is a classical solution of \eqref{hjbn_eps}.
To obtain the lower bound in Theorem \ref{thm.main} for the \textit{regularized} value functions, we follow the argument used to achieve a similar result in \cite[Subsection 3.2]{cdjs2023} and \cite[Section 6]{ddj2023}. This technique is based on the introduction of an \textit{integrated value function} $\hat{V}^{N,\eta}$ defined by \eqref{hatvndef} (see also the proof of Proposition 3.1 in \cite{cm2020}), which is used as an intermediate comparison term and for which the error can be quantified through the results in \cite{FournierGuillin}.

\begin{lem} \label{lem.vhatsubol}
    The function $\hat{V}^{N,\eta}$ is $\cC^{1,2}$ on $[0,T) \times \cP(\T^d)$, and there is a constant $C$ depending only on the data such that the following is satisfied in a classical sense: 
    \begin{align}
         \begin{cases}
        \ds - \partial_t \hat{V}^{N,\eta} - \eta \int_{\T^d} \tr\big(D_{xm} \hat{V}^{N,\eta}(t,m,x) \big) m(dx) - \int_{\T^d} \tr\big(A(x) D_{xm} \hat{V}^{N,\eta}(t,m,x) \big)m(dx) 
         \vspace{.1cm} \\ \ds 
        \quad - A_0 \com \hat{V}^{N,\eta}  + \int_{\T^d} H\big(x, D_m \hat{V}^{N,\eta}(t,m,x)\big) m(dx) \leq  F^{\eta}(m) + C r_{N,d}, \quad  (t,m) \in [0,T) \times \cP(\T^d), \vspace{.1cm} \\
        \ds \hat{V}^{N,\eta}(T,m) \leq  G^{\eta}(m) + C r_{N,d} \quad m \in \cP(\T^d).
    \end{cases}
    \end{align}
\end{lem}

\begin{proof}
    The smoothness of $\hat{V}^{N,\eta}$ follows from the smoothness of $V^{N,\eta}$. Explicit computation shows that $\hat{V}^{N,\eta}$ satisfies 
    \begin{equation}\label{eqn: v_hat_eqn_exact} 
    \begin{cases} 
    \ds - \partial_t \hat{V}^{N,\eta} - \eta \int_{\T^d} \tr\big(D_{xm} \hat{V}^{N,\eta}(t,m,x) \big) m(dx) - \int_{\T^d} \tr\big(A(x) D_{xm} \hat{V}^{N,\eta}(t,m,x) \big)m(dx) 
         \vspace{.1cm} \\  \ds  
        \quad - A_0 \com \hat{V}^{N,\eta}
   + \frac{1}{N} \int_{(\T^d)^N}\sum_{i = 1}^N H(x^i, ND_{x^i} V^{N,\eta})\prod_{i=1}^N m(dx^i) 
   \\ \ds \qquad = \hat{F}^{\eta}(m) \coloneqq \int_{(\T^d)^N} F^{\eta}(m_{\bx}^N) m^{\otimes N}(d\bx), \quad(t,m) \in [0,T) \times \sP(\T^d), \vspace{.1cm} \\
   \ds \hat{V}^{N,\eta}(T,m) = \hat{G}^{\eta}(m):= \int_{(\T^d)^N} G^{\eta}(m_{\bx}^N)\prod_{i=1}^N m(dx^i), \quad m\in\sP(\T^d).
    \end{cases}
\end{equation}
By convexity in $p$ of $H$, we have
\begin{align*}
	\int_{\T^d} H\big(x,D_m \hat{V}^{N,\eta}(t,m,x) \big) m (dx)\leq \frac{1}{N} \int_{(\T^d)^N}\sum_{i = 1}^N H(x^i, ND_{x^i} V^{N,\eta})\prod_{i=1}^N m(dx^i),
\end{align*}
and so combining this with \eqref{eqn: v_hat_eqn_exact} we get
\begin{align*}
	&- \partial_t \hat{V}^{N,\eta} - \eta \int_{\T^d} \tr\big(D_{xm} \hat{V}^{N,\eta}(t,m,x) \big) m(dx) - \int_{\T^d} \tr\big(A(x) D_{xm} \hat{V}^{N,\eta}(t,m,x) \big)m(dx) 
         \vspace{.1cm} \\  &\qquad \ds  
        \quad - A_0 \com \hat{V}^{N,\eta}
   + \int_{\T^d} H\big(x, D_m \hat{V}^{N,\eta}(t,m,x) \big)m(dx)  \leq  \hat{F}^{\eta}(m).
\end{align*}
To conclude, we apply the results of \cite{FournierGuillin} to get 
\begin{align*}
   | \hat{F}^{\eta}(m) - F^{\eta}(m) | = |\int_{(\T^d)^N} F^{\eta}(m_{\bx}^N) m^{\otimes N}(d\bx) - F^{\eta}(m) | \leq C_F \int_{(\T^d)^N} \bd_1(m_{\bx}^N,m) m^{\otimes N}(d\bx) \leq Cr_{N,d},
\end{align*}
and likewise for $|\hat{G}^{\eta}(m) - G^{\eta}(m)|$.
\end{proof}

We now combine the sub-solution property for $\hat{V}^{N,\eta}$ outlined above with the definition of $U^{\eta}$ to obtain:
\begin{lem}\label{lemma: easy_ineq_Vhat}
	Let Assumption \ref{assump.maindegen} hold. Then there exists a constant $C>0$, depending only on the data, such that $\hat{V}^{N,\eta}(t,m)\leq U^\eta(t, m) + C r_{N,d}$ for any $(t,m)$ in $[0,T]\times\sP(\T^d)$.
\end{lem}

\begin{proof}
Fix $(t_0,m)\in[0,T)\times\sP(\T^d)$, and let us consider an arbitrary control $\alpha \in \ov{\sA}_{t_0,m_0}$. By a standard verification argument exactly as in \cite[Proposition 3.7]{cdjs2023}, Lemma \ref{lem.vhatsubol} implies that
\begin{align*}
	\hat{V}^{N,\eta}(t_0,m)& \leq \E^{\ov{\bP}^{\infty, t_0, m_0}}\left[\int_{t_0}^T \Big( L\big(X_t, \alpha_t \big) + F^{\eta}(\sL^{0,m_0}(X_{t}^{t_0,\alpha})) + Cr_{N,d} \Big)dt + G^{\eta}(\sL^{0,m_0}(X_{T}^{t_0,\alpha}))\right] + Cr_{N,d}.
\end{align*}
Taking an infimum over $\alpha \in \ov{\sA}_{t_0,m_0}$ completes the proof.
\end{proof}

Since we want to use $\hat{V}^{N,\eta}(t, m^N_\bx)$ to approximate $V^\eta(t,\bx)$, we need to quantify the distance between these two functions. This is done in the following lemma. 
\begin{lem}\label{lemma: approx v_vhat}
	Let Assumption \ref{assump.maindegen} hold. There exists a constant $C>0$, depending only on the data, such that $\lvert \hat{V}^{N,\eta}(t,m^N_\bx) - V^{N,\eta}(t,\bx)\rvert \leq C r_{N,d}$  for any $\bx \in(\T^d)^N$ and $t\in[0,T]$.
\end{lem}
\begin{proof}
The claim is obtained as in the final part of the proof of Proposition 6.2 in \cite{ddj2023}, by exploiting the Lipschitz property \eqref{Lip:VNeps} stated in Proposition \ref{prop: bound_der_Veps}.
\end{proof}
Finally, we conclude with the main result of the section.

\begin{prop} \label{prop.easy}
	Let Assumption \ref{assump.maindegen} hold. Then there is a constant $C$ depending only on the data such that $V^N(t,\bx)\leq U(t, m^N_\bx) + C r_{N,d} $.
\end{prop}

\begin{proof}
By combining Lemmas \ref{lemma: easy_ineq_Vhat} and \ref{lemma: approx v_vhat}, we find that there is a constant $C$ depending only on the data such that $V^{N,\eta}(t,\bx)\leq U^\eta(t, m^N_\bx) + C r_{N,d} $. In light of Lemma \ref{lem.epstozero}, we can conclude.
\end{proof}

\section{The hard inequality}

\label{sec:hardinequalityconditional}

For most of this section, we will assume that the value function $U$ has some additional regularity properties. In particular, we assume that 
\begin{align} \label{assump.extrareg}
U \in \text{Lip}\big([0,T] \times \cP(\T^d) ; W^{-2,\infty} \big).
\end{align}
When \eqref{assump.extrareg} holds we write $\lip(U ; W^{-2,\infty})$ for the smallest constant $C$ such that 
\begin{align*}
    |U(t,m) - U(t,m') | \leq C \|m - m'\|_{-2,\infty}
\end{align*}
for all $t \in [0,T]$, $m,m' \in \cP(\T^d)$. We will simplify the exposition by making the following convention.
\begin{convention}
    Throughout Section \ref{sec:hardinequalityconditional}, the constant $C$ will be allowed to change from line to line but will depend only on the data (in the sense of Convention \ref{conv.data}), unless otherwise noted. Dependence of $C$ on additional parameters will be made explicit by writing $C(...)$. In other words, the statement ``$A \leq C(D) B$" will mean that there exists a constant $C$ depending only on the data (in the sense of Convention \ref{conv.data}) and on $D$ such that such that $A \leq CB$. Crucially, we will not allow $C$ to depend on $\lip(U; W^{-2,\infty})$, and instead will specify explicitly the dependence of all our estimates on $\lip(U; W^{-2,\infty})$.
\end{convention}

The goal of Subsection \ref{sec:hardinequalityconditional} is to prove the following proposition, which shows that if \eqref{assump.extrareg} holds in addition to Assumption \ref{assump.maindegen}, then we have the desired convergence result, and with a constant that depends explicitly on $\lip(U; W^{-2,\infty})$. 

\begin{prop} 
    \label{prop.mainextrareg}
    Let Assumption \ref{assump.maindegen} and \eqref{assump.extrareg} hold. Then there is a constant $C$ such that
    \begin{align*}
        U(t,m_{\bx}) \leq V^N(t,\bx) + C \big(1 + \lip(U; W^{-2,\infty}) \big)^{1-1/(2s^* + 1)} N^{-1/(2s^*  + 1)}, 
    \end{align*}
    for all $(t,\bx) \in [0,T] \times (\T^d)^N$, where 
    \begin{align*}
        s^* = \begin{cases}
            d/2 + 3 & \text{if $d$ is even} \\
            d/2 + 5/2 & \text{if $d$ is odd}.
            \end{cases}
    \end{align*}
    If $A$ is constant, then the dependence on $\lip(U; W^{-2,\infty})$ can be removed, i.e. there is a constant $C$ such that
    \begin{align*}
        U(t,m_{\bx}) \leq V^N(t,\bx) + C  N^{-1/(2s^*  + 1)}.
    \end{align*}
\end{prop}

Before moving on, let us see how Proposition \ref{prop.mainextrareg} implies Theorem \ref{thm.main}.

\begin{proof}[Proof of Theorem \ref{thm.main}]
    Let $U^{\eta}$, $V^{N,\eta}$ be as defined in Section \ref{sec: reg_values}. By combining Proposition \ref{prop: bound_der_Veps}, Lemma \ref{lem.epstozero}, and Proposition \ref{prop.mainextrareg}, we find that 
    \begin{align*}
        U - V^N \leq C\sqrt{\eta} + U^{\eta} - V^{N,\eta} &\leq C \sqrt{\eta} + C \Big(1 + \lip(U^{\eta}; W^{-2,\infty}) \Big)^{1 -1/(2s^* + 1)} N^{-1/(2s^* + 1)} 
        \\
        &\leq C \sqrt{\eta} + C \eta^{-(2s^*)/(2s^*+1)} N^{-1/(2s^* + 1)}.
    \end{align*}
    Choose
   $
        \eta = N^{-2/(6s^* + 1)}
$
    to get the bound $U \leq V^N + C N^{-1/(6s^* + 1)}$, and recall the definition of $s^*$ to complete the proof in the degenerate case. The lower bound follows from Proposition \ref{prop.easy}. Obviously, the improved bound when $A$ is constant or $U \in \lip(U ; W^{-2,\infty})$ follows immediately from Proposition \ref{prop.mainextrareg}, and this completes the proof.
\end{proof}

We now fix a smooth approximation to the identity $(\rho_{\delta})_{\delta > 0}$ on $\T^d$, and we define
\begin{align} \label{def.udelta}
    U^{\delta}(t,m) = U(t,m * \rho_{\delta}),
\end{align}
We then employ the change of variables from \cite{bayraktar2023} and \cite{djs2023}, and study the map $\hat{U}^{\delta} : [0,T] \times \R^d \times \cP(\T^d) \to \R$, given by 
\begin{align} \label{def.udeltahat}
    \hat{U}^{\delta}(t,z,m) = U^{\delta}(t,m^z), \quad m^z \coloneqq (\text{Id} + z)_{\#} m. 
\end{align}
Notice that $\hat{U}^{\delta}(t,z,m) = \hat{U}(t,z, \rho_{\delta}*m)$ since $\rho_{\delta} * (\text{Id} + z)_{\#} m  = (\text{Id} + z)_{\#} \rho_{\delta} *m$.
Next, we employ a ``sup-convolution", defining for each $\epsilon > 0$ the function $\hat{U}^{\delta, \epsilon}(t,z,q) : [0,T] \times \R^d \times H^{-s^*} \to \R$ by
\begin{align} \label{def.udeltaepshat}
    \hat{U}^{\delta, \epsilon}(t,z,q) = \sup_{z'\in \R^d,m' \in \cP(\T^d)} \Big\{\hat{U}^{\delta}(t,z',m') - \frac{1}{2\epsilon}\big(|z - z'|^2 + \|q - m'\|_{-s^*}^2 \big) \Big\}, 
\end{align}
where 
\begin{align} \label{sdef}
    s^* = \begin{cases}
        d/2 + 3 & \text{if $d$ is even}
        \\
        d/2 + 5/2 & \text{if $d$ is odd}.
    \end{cases}
\end{align}
In other words, we choose $s^*$ to be the smallest integer which is strictly larger than $d/2 + 2$. The reason for this choice are (i) we need $s^* > d/2 + 1$ to make use of \eqref{dmmbound}, (ii) we need $s^* > d/2 + 2$ to make use of \eqref{dxmcont}, and (iii) we want to choose an integer so that we have the explicit formula \eqref{dualidentity} and in particular the bound \eqref{c2sbound}.
Finally, as in \cite{ddj2023}, we find it necessary to employ a final transformation which ``forces" a lower bound on the density of the measure argument, replacing $\hat{U}^{\delta, \epsilon}$ by $\hat{U}^{\delta, \epsilon, \lambda}$ for $\lambda \in (0,1)$, which is given by
\begin{align}  \label{def.udeltaepshatlambda}
    \hat{U}^{\delta, \eps, \lambda}(t,z,q) = \hat{U}^{\delta, \epsilon}(t,z,\lambda \leb + (1-\lambda)q).
\end{align}

\subsection{Analysis of the various transformations of $U$}

We are now going to analyze the functions $U^{\delta}$, $\hat{U}^{\delta}$, $\hat{U}^{\delta, \eps}$, and $\hat{U}^{\delta, \epsilon, \lambda}$ defined above. For convenience, we also define $\hat{U}(t,z,m) = U(t,m^z)$. The following elementary lemma, whose proof we omit, will be useful in establishing some regularity properties of $\hat{U}^{\delta}$. 

\begin{lem} \label{lem.mollification}
For any $s \in \N$, we have
\begin{align*}
    \bd_1(m * \rho_{\delta},m' * \rho_{\delta}) \leq C(s)\delta^{-(s-1)}\|m - m' \|_{-s}, \quad \|m * \rho_{\delta}-m' * \rho_{\delta}\|_{\tv} \leq C(s) \delta^{-s}\|m - m'\|_{-s}
\end{align*}
for any $m,m' \in \cP(\T^d)$. In addition, we have $ \bd_1(m, m * \rho_{\delta}) \leq C \delta$
for any $\delta > 0$ and $m \in \cP(\T^d)$.
\end{lem}

We now establish some regularity properties of $U^{\delta}$. In what follows, we will use $\lip$ and $\semi$ to denote Lipschitz and semiconcavity bounds with respect to various metrics, for example $\lip(U^{\delta}; \bd_1)$ will denote the smallest constant $C$ such that 
\begin{align*}
    |U^{\delta}(t,m) - U^{\delta}(t,m')| \leq C \bd_1(m,m')
\end{align*}
holds for each $t,m,m'$, whereas $\semi(U^{\delta}; \tv)$ will denote the smallest constant $C$ such that 
\begin{align*}
    U(t,\lambda m + (1-\lambda) m') \geq \lambda U(t,m) + (1-\lambda) U(m') - \frac{C}{2} \lambda(1-\lambda) \|m - m'\|_{\tv}.
\end{align*}

\begin{lem} \label{lem.udelta}
    Let Assumption \ref{assump.maindegen} hold. For any $s \in \N$, $\delta > 0$, we have $U^{\delta} \in \text{Lip}\big([0,T] \times \cP(\T^d) ; H^{-s}\big)$, 
   and moreover $U^\delta$ is $H^{-s}$-semiconcave in $m$, with the estimates
    \begin{align}
        \lip(U^{\delta};H^{-s}) \leq C(s) \delta^{-(s-1)} , \quad \semi(U^{\delta}; H^{-s}) \leq C(s) \delta^{-2s}.
    \end{align}
\end{lem}

\begin{proof}
    For the Lipschitz bound, we use the $\bd_1$-Lipschitz bound for $U$ from Proposition \ref{prop: d1_lip} and the estimates in Lemma \ref{lem.mollification} to estimate 
    \begin{align*}
        |U^{\delta}(t,m) - U^{\delta}(t,m')| &= |U(t,m * \rho_{\delta}) - U(t,m' * \rho_{\delta})| \\ 
       &\leq \lip(U; \bd_1) \bd_1(m * \rho_{\delta}, m' * \rho_{\delta}) \leq C(s) \delta^{- (s-1)} \|m - m'\|_{-s}. 
    \end{align*}
    For the semiconcavity bound, we fix $t \in [0,T]$ and $m,m' \in \cP(\T^d)$. We set $m^{\lambda} = \lambda m + (1-\lambda)m'$, and use the $\tv$-semiconcavity bound from Proposition \ref{prop: TV_semiconc} together with Lemma \ref{lem.mollification} to get 
    \begin{align*}
       U^{\delta}(t,m^{\lambda}) &= U(t, m^{\lambda} * \rho_{\delta}) = U(t, \lambda m^1 * \rho_{\delta} + (1 - \lambda) \rho_{\delta} * m^0) \\
       &\geq \lambda U(t, m^1 * \rho_{\delta}) + (1- \lambda) U(t,m^0 * \rho_{\delta}) - \frac{\semi(U; \tv)}{2} \lambda(1-\lambda) \|m^1 * \rho_{\delta} - m^2 * \rho_{\delta}\|_{\tv}^2 \\
        &\geq \lambda U^{\delta}(t,m^0) + (1- \lambda) U^{\delta}(t,m^1) - C(s) \lambda(1- \lambda) \delta^{-2s} \|m^1 - m^2\|_{-s}^2. 
   \end{align*}
\end{proof}

The next technical lemma will be used to establish a joint semiconcavity estimate for $\hat{U}^{\delta}$.

\begin{lem} \label{lem.liftsemiconcavity}
    Let $s \in \N$ with $s > d/2 + 1$, and let $\Phi(m) : \cP(\T^d) \to \R$ be a function which is $H^{-s}$-semiconcave and $H^{-s-1}$-Lipschitz. Define $\hat{\Phi} : \R^d \times \cP(\T^d) \to \R$ by the formula $\hat{\Phi}(z,m) = \Phi(m^z)$. Then $\hat{\Phi}$ satisfies the semiconcavity estimate 
    \begin{align} \label{jointsemiconcavity}
        \hat{\Phi}(z_{\lambda},m_{\lambda}) &\geq (1-\lambda) \hat{\Phi}(z,m) + \lambda \hat{\Phi}(z',m') \nonumber \\ &\quad- C(s) \big(\semi(\Phi; H^{-s}) + \lip(\Phi; H^{-s-1}) \big) \lambda(1-\lambda) \big(|z - z'|^2 + \|m - m'\|_{-s}^2 \big),
    \end{align}
    for each $z,z' \in \R^d$, $m,m' \in \cP(\T^d)$, and $\lambda \in (0,1)$, where $(z_{\lambda}, m_{\lambda}) = \lambda(z',m') + (1-\lambda)(z,m)$. 
\end{lem}

\begin{proof}
    First, assume that $\Phi$ is $\cC^2$. Fix $m,m' \in \cP(\T^d)$ with smooth densities which are bounded below. Moreover fix $z,z' \in \R^d$. For $\lambda \in [0,1]$, define $(z_{\lambda},m_{\lambda}) = (1- \lambda) (z, m) + \lambda (z',m')$.
    Our goal will be to find an upper bound for $      \frac{d^2}{d \lambda^2} \hat{\Phi}(z_{\lambda}, m_{\lambda}) \Big|$. 
    We begin by computing
    \begin{align*}
        \frac{d}{d\lambda} \hat{\Phi}(z_{\lambda}, m_{\lambda}) &= \frac{d}{d\lambda} \Phi(m_{\lambda}^{z_{\lambda}}) \\
        &= \int_{\T^d} \frac{\delta \Phi}{\delta m}(m_{\lambda}^{z_{\lambda}}, x + z_{\lambda}) d(m' - m)(x) + \int_{\T^d} D_m\Phi(m_{\lambda}^{z_{\lambda}}, x + z_{\lambda}) \cdot (z' - z) dm_{\lambda}(x), 
    \end{align*}
    which allows us to find 
    \begin{align*}
        \frac{d^2}{d \lambda^2} \hat{\Phi}(z_{\lambda}, m_{\lambda}) \Big|_{\lambda = 0} &= \int_{\T^d} \int_{\T^d} \frac{\delta^2 \Phi}{\delta m^2} (m^z, x+ z, y+z) d(m'-m)^{\otimes 2}(x,y) \\
        &\quad + \int_{\T^d} \int_{\T^d} (z'-z)^T D_x D_y \frac{\delta^2 \Phi}{\delta m^2} (m^z, x+ z, y+ z) (z'-z) dm^{\otimes 2}(x,y) \\
        &\quad + \int_{\T^d} \int_{\T^d} D_y \frac{\delta^2 \Phi}{\delta m^2} (m^z, x + z, y + z) \cdot (z'-z) dm(y) d(m' - m)(x) \\
        &\quad+ \int_{\T^d} \int_{\T^d} D_{x} \frac{\delta^2 \Phi}{\delta m^2} (m^z, x + z, y + z) \cdot (z' - z) dm(x) d(m' - m)(y) \\
        &\quad+ \int_{\T^d} (z'-z)^T D^2_{xx} \frac{\delta \Phi}{\delta m}(m^z, x+ z) (z' - z) dm(x) \\ & \quad+ 2 \int_{\T^d} D_x \frac{ \delta \Phi}{\delta m} (m^z,x+z) \cdot (z'-z) d(m'-m)(x).
    \end{align*}
    Now, if we set $\phi = - D (m) \cdot (z' - z)$, we can use integration by parts to rewrite the above as
    \begin{align} \label{dlambda2}
        \frac{d^2}{d \lambda^2} \hat{\Phi}(z_{\lambda}, m_{\lambda}) \Big|_{\lambda = 0} 
        &= \int_{\T^d} \int_{\T^d} \frac{\delta^2 \Phi}{\delta m^2}(m^z, x + z , y + z) (m'  - m + \phi)(x) (m'  - m + \phi)(y)
         \nonumber \\
        & \quad + \int_{\T^d} (z_1 - z_0)^T D^2_{xx} \frac{\delta \Phi}{\delta m}
        (m_0^{z_0}, x+ z_0) (z_1 - z_0) dm_0(x) \nonumber \\ &\quad + 2 \int_{\T^d} D_x \frac{\delta \Phi}{\delta m}(m_0^{z_0},x+z_0) \cdot (z_1 - z_0) d(m_1 - m_0)(x).
    \end{align}
    Now, since $\Phi$ is $H^{-s}$ semiconcave with constant $\semi(\Phi ; H^{-s})$, and $s > d/2 + 1$, we have
    \begin{align*}
     \int_{\T^d} \int_{\T^d}& \frac{\delta^2 \Phi}{\delta m^2}(m^z, x + z , y + z) (m' + \phi - m)(x) (m' + \phi - m)(y) 
    \\ & = \int_{\T^d} \int_{\T^d} \frac{\delta^2 \Phi}{\delta m^2}(m^z, x , y) \big( (m' + \phi)^z - m^z \big)(x) \big( (m' + \phi)^z - m^z \big)(y) 
     \\
     &\leq \semi(\Phi; H^{-s}) \|(m' + \phi)^z - m^z \|_{-s}^2 =  \semi(\Phi; H^{-s}) \|m' + \phi - m \|_{-s}^2 \\
        &\leq C \semi(\Phi; H^{-s}) \big(\|\phi\|_{-s}^2 + \|m_1 - m_0\|_{-s}^2 \big) \\
        &\leq C \semi(\Phi; H^{-s}) \big(|z_1 - z_0|^2 + \|m_1 - m_0\|_{-s}^2 \big), 
    \end{align*}
    where we used the fact that $m$ is bounded from below and in the last line we used the fact that $s > d/2 + 1$ to conclude that 
    \begin{align*}
        \| \phi \|_{-s} &= \sup_{\|f\|_s \leq 1} \int_{\T^d} f (- D m \cdot (z' -z )) dx   \leq \sup_{\|f\|_{W^{1,\infty}} \leq 1} \int Df \cdot (z'-z) dm \leq C |z' - z|. 
    \end{align*}
    We note also that 
    \begin{align*}
        \int_{\T^d} (z_1 - z_0)^T D^2_{xx} 
        \frac{\delta \Phi}{\delta m}(m^z, x+ z) (z'-z) dm(x) \leq C \lip(\Phi; W^{-2,\infty}) |z'-z|^2,
        \end{align*}
         and 
    \begin{align*}
       \int_{\T^d} D_x \frac{\delta \Phi}{\delta m}(m^z,x+z) \cdot (z'-z) d(m'-m)(x) &\leq C \|m'-m\|_{-s} |z'-z| \sup_m \|\frac{\delta \Phi}{\delta m}(m,\cdot)\|_{H^{s+1}} \\
       &\leq C \|m' - m\|_{-s} |z' - z| \lip(\Phi; H^{-s-1}).
    \end{align*}
    Coming back to \eqref{dlambda2}, we find 
    \begin{align} \label{scdifferential}
        \frac{d^2}{d \lambda^2} \hat{\Phi}(z_{\lambda}, m_{\lambda}) \Big|_{\lambda = 0} & \leq C \Big(\semi(\Phi; H^{-s}) + \lip(\Phi; W^{-2,\infty}) + \lip(\Phi; H^{-s-1}) \Big)\big(\|m' - m\|_{-s}^2 + |z'-z|^2 \big) \nonumber \\
        &\leq C \Big(\semi(\Phi; H^{-s}) + \lip(\Phi; H^{-s-1}) \Big)\big(\|m' - m\|_{-s}^2 + |z'-z|^2 \big)
    \end{align}
    with the last inequality using the fact that $s > d+1$. Obviously, the same upper bound holds also for $\lambda \in [0,1]$. Integrating the inequality \eqref{scdifferential} shows that \eqref{jointsemiconcavity} holds whenever $\Phi$ is $\cC^2$ and $m,m'$ are smooth and bounded below. By a density argument, we get that \eqref{jointsemiconcavity} holds for all $z,z',m,m'$ whenever $\Phi$ is $\cC^2$. To extend to the case that $\Phi$ is not $\cC^2$, we claim that there exists a sequence of $\cC^2$ functions $\Phi^n(m) : \cP(\T^d) \to \R$ with the same $H^{-(s-1)}$-Lipschitz and $H^{-s}$-semiconcavity constants as $\Phi$ such that
    $\Phi^n \to \Phi$ uniformly as $n \to \infty$.
    This allows to extend to non-smooth $\Phi$. To construct such a mollification scheme, one can use the approximation argument of Lemma 4.1. in \cite{djs2023} (see also \cite{Cecchin:Delarue:CPDE}) once we have noticed that taking convolution with the Fejer kernel is not only Lipschitz with respect to $\bd_1$ but also with respect to $\norm{ \cdot}_{s}$ for any $s \in \R$, as can easily be seen using the equivalent norm $||| \phi |||^2_s = \sum_{k \in \Z^d} (1 + |k|^s)^2 |\hat{\phi}^k|^2$, defined in terms of the Fourier coefficients $\hat{\phi}^k$ of $\phi$. We omit the details.
\end{proof}

We now establish the key properties of $\hat{U}^{\delta}$.

\begin{lem} \label{lem.udeltahatreg}
    For each $s \in \N$ with $s > d/2 + 1$, $\delta > 0$, we have $      \hat{U}^{\delta} \in \text{Lip}\big([0,T] \times \R^d \times \cP(\T^d) ; H^{-s}\big)$,
    and moreover $\hat{U}^{\delta}$ is semiconcave in $(z,m)$, uniformly in $t$. More precisely, we have
    \begin{align*}
        |\hat{U}^{\delta}(t,z,m) - \hat{U}^{\delta}(t,z',m')| \leq C(s)\big(|z - z'| + \delta^{-(s-1)} \|m - m'\|_{-s} \big)
    \end{align*}
    and 
    \begin{align*}
        \hat{U}^{\delta}(t,z_{\lambda}, m_{\lambda}) \geq (1 - \lambda) \hat{U}^{\delta}(t,z,m) + \lambda \hat{U}^{\delta}(t,z',m') - C(s) \delta^{-2s} \lambda(1- \lambda) \big(|z-z'|^2 + \|m-m'\|_{-s}^2 \big)
    \end{align*}
    for all $t \in [0,T]$, $z,z' \in \R^d$, $m,m' \in \cP(\T^d)$, where $(z_{\lambda},m_{\lambda}) = (1- \lambda) (z,m) + \lambda (z',m')$. Finally, we have 
    \begin{align*}
       |\hat{U}^{\delta}(t,z,m) - \hat{U}(t,z,m)| \leq C \delta
    \end{align*}
    for each $(t,z,m) \in [0,T] \times \R^d \times \cP(\T^d)$.
   \end{lem}

\begin{proof}
    For the Lipschitz bound, we use Lemma \ref{lem.mollification} to estimate
    \begin{align*}
        |\hat{U}^{\delta}(t,z,m) - \hat{U}^{\delta}(t,z',m')| &= |U^{\delta}(t,m^z) - U^{\delta}(t,(m')^{z'})|  =  |U(t, \rho_{\delta} * m^z) - U(t, \rho_{\delta} * (m')^{z'})|  \\
        &\leq C \big(|t-t'| + \bd_1(\rho_{\delta} * m^z, \rho_{\delta} * (m')^{z'})  \big) \\
        &\leq C\big(|t-t'| + |z - z'| + \bd_1(\rho_{\delta} * m, \rho_{\delta} * m') \big) \\
        &\leq C(s)\big(|t-t'| + |z - z'| + \delta^{-(s-1)} \| m - m'\|_{-s} \big).
    \end{align*}
     The semiconcavity bound follows from combining Lemmas \ref{lem.udelta} and \ref{lem.liftsemiconcavity}. Finally, to prove the last bound in the statement of the lemma, we again use Lemma \ref{lem.mollification} to get 
    \begin{align*}
        |\hat{U}^{\delta}(t,z,m) - \hat{U}^{\delta}(t,z,m)| \leq C \bd_1( (m * \rho_{\delta})^z, m^z) = C \bd_1(m * \rho_{\delta}, m) \leq C \delta.
    \end{align*}
    \end{proof}
    Our next task is to show that $\hat{U}^{\delta}$ is almost a subsolution of \eqref{hjbz}. For this, we need to specify an appropriate set of test functions. We define $\cC_{\text{p}}^{1,2,2} = \cC_{\text{p}}^{1,2,2}((0,T) \times \R^d \times \cP(\T^d))$ as the set of functions $\Phi(t,z,m) : (0,T) \times \R^d \times \cP(\T^d) \to \R$ such that $\partial_t \Phi$, $D_z \Phi$, $D^2_{zz} \Phi$, $D_m \Phi$, $D_{xm} \Phi$ exist and are continuous. Given $\Psi : [0,T] \times \R^d \times \cP(\T^d) \to \R$, we say that $\Phi$ touches $\Psi$ from above at $(t_0,z_0,m_0)$ if we have 
    \begin{align*}
        \Phi(t_0,z_0,m_0) - \Psi(t_0,z_0,m_0) = \sup_{(t,z,m) \in (0,T) \times \R^d \times \cP(\T^d)} \Big(\Phi(t,z,m) - \Psi(t,z,m) \Big).
    \end{align*}

\begin{lem} \label{lem.udeltahatsubsol} There is a constant $C$ (depending only on the data) with the following property: if $\Phi \in \cC_{\text{p}}^{1,2,2}$ touches $\hat{U}^{\delta}$ from above at a point $(t_0,z_0,m_0) \in (0,T) \times \R^d \times \cP(\T^d)$ such that $m_0 \geq c \leb$ for some $c > 0$, then we have
    \begin{align} \label{subsoldelta}
        &- \partial_t \Phi(t_0,z_0,m_0) - A_0 \Delta_z \Phi(t_0,z_0,m_0) - \int_{\T^d}\tr\big( \hat{A}(x_0,z_0) D_{xm} \Phi(t_0,z_0,m_0,x) \big) m_0(dx) \nonumber  \\
        &\qquad  + \int_{\T^d} \hat{H}\big(x,z_0,D_m \Phi(t_0,z_0,m_0,x)\big) m_0(dx) \leq \hat{F}(z_0, m_0) +  C\big(1 + \lip(U; W^{-2,\infty}) \big) \delta.
    \end{align}
    If $A$ is constant, then the dependence on $\lip(U; W^{-2,\infty})$ can be removed, i.e. the right-hand side of \eqref{subsoldelta} can be replaced by $C\delta$.
\end{lem}

We recall the notation $\hat{U}(t,z,m) = U(t,m^z) = U(t,(\text{Id} + z)_{\#} m)$,
and start with a preliminary lemma, which gives a certain dynamic programming inequality for $\hat{U}$.

\begin{lem} \label{lem.dppinequality}
    Suppose that $(t_0,z_0,m_0) \in [0,T) \times \R^d \times \cP(\T^d)$ are given, and that $h \in [t_0 - T, T]$. Suppose further that we have an $\R^d$-valued random field $(\alpha_t(x))_{t_0 \leq t \leq T}$ which is uniformly bounded and Lipschitz in $x$, defined on the canonical space $\Omega^0$ introduced above, and adapted to the filtration $\bbF^{0,t_0}$. Then we have 
    \begin{align} \label{dppinequality}
        \hat{U}(t_0,z_0,m_0) &\leq \E^{\bP^{0}}\bigg[ \hat{U}(t_0+h, Z_{t_0+h}, m_{t_0+h}) + \int_{t_0}^T \Big(\int_{\T^d} L(x+Z_t, \alpha_t(x)) m_t(dx) 
        + \hat{F}(Z_t,m_t) \Big) dt \bigg],
    \end{align}
    where $\hat{F}$ and $\hat{G}$ are defined by \eqref{hatdefs}, $m_t$ is determined from $\alpha$ via the Fokker-Planck equation with random coefficients
\begin{equation} 
\partial_t m_t = - \div( \alpha_t(x) m_t) + \sum_{i,j} D^2_{x_i x_j}(\hat{A}^{ij}(x,Z_t) m_t) , \quad  \mbox{ in } (t_0,T) \times \T^d, \quad m_{t_0} = m_0,  
\label{eq:FPEafterchangeofvar}
\end{equation}
and $Z$ is given by $Z_t = z_0 + \sigma^0 (W_t^0 - W_{t_0}^0)$.
\end{lem}

\begin{proof}
Let $(t_0,z_0, m_0) \in [0,T) \times \cP(\T^d)$, and $\alpha$ be as in the statement of the lemma, and define $\alpha'_t(x) = \alpha_t(x - Z_t)$. Let $m_t'$ be the solution of
    \begin{align} \label{feedbackdynamicsnoeta}
    &dm'_t = \Big(  \sum_{i,j} D^2_{x_i x_j}(A_{ij}(x) m_t') + A_0 \Delta m_t' - \text{div}(m_t' \alpha_t') \Big) dt \nonumber  \\
    & \qquad \qquad \qquad  - \sigma^0 Dm'_t \cdot dW_t^0, \quad t_0 \leq t \leq T, \quad m'_{t_0} = m'_0 \coloneqq (\text{Id} - z_0)_{\#}m_0.
\end{align}
Notice that $m_t = (\text{Id} - Z_t)_{\#} m_t'$ satisfies \eqref{eq:FPEafterchangeofvar}.
Thus, from the the dynamic programming principle in \cite{djete2019mckean}, we have
\begin{align*}
     \hat{U}(t_0,z_0,m_0) &=  U(t_0,m_0') \leq \E^{0}\bigg[ U(t_0 +h, m'_{t_0 +h}) + \int_{t_0}^T \Big(\int_{\T^d} L(x, \alpha'_t(x)) m_t'(dx) 
        + F(m'_t) \Big) dt  \bigg] 
        \\
    &= \E^{0}\bigg[ \hat{U}(t_0 +h, Z_{t_0 + h}, m_{t_0 +h}) + \int_{t_0}^T \Big(\int_{\T^d} L(x + Z_t, \alpha_t(x)) m_t(dx) 
        + \hat{F}(Z_t,m_t) \Big) dt \bigg].
    \end{align*}
\end{proof}

\begin{proof}[Proof of Lemma \ref{lem.udeltahatsubsol}]
  For simplicity in this argument we write $\bP$ for $\bP^{0}$, $\E$ for $\E^0$, etc.  We fix $(t_0,z_0, m_0) \in [0,T] \times \R^d \times \mathcal{P}(\T^d)$ and $\Phi : [0,T] \times \R^d \times \mathcal{P}(\T^d)  \rightarrow \R$  a smooth test function touching $\hat{U}^{\delta}$ from above at $(t_0, z_0, m_0)$ (see the remark before Lemma \ref{lem.udeltahatsubsol}). Let us define 
$$\alpha^0(x) := -D_p \hat{H} \bigl(x,z_0,D_m \Phi(t_0,z_0,m_0,x) \bigr) = -D_p H \bigl(x+z_0,D_m \Phi(t_0,z_0,m_0,x) \bigr) $$
so that, by definition of $H$,
\begin{align}  \label{alphazeroprop}
-D_m\Phi(t_0,z_0,m_0,x) \cdot \alpha^0(x) - L \bigl( x+z_0, &\alpha^0(x) \bigr) = \hat{H} \bigl( x, z_0, D_m\Phi(t_0,z_0,m_0,x) \bigr).
\end{align}
We set $Z_t = z_0 + \sigma^0(W_t^0 - W_{t_0}^0)$, and let $m_t$ denote (stochastic) solution to
$$\partial_t m_t = - \div( \alpha^0 m_t) + \sum_{i,j} D^2_{x_i x_j} \bigl( A^{ij}(x+Z_t)m_t), \quad \mbox{ in } (t_0,T) \times \T^d,\quad m(t_0) = m_0. $$
Since $\norm{\rho_{\delta}*m_t - \rho_{\delta}*m_{t_0}}_{L^{\infty}} \leq C \bd_1(m_t,m_{t_0})$ for some $C>0$ depending on $\delta >0$ and $t \mapsto m_t \in (\mathcal{P}(\T^d),d_1)$ is clearly continuous, we get for $t -t_0$ small enough 
\begin{equation}
    \rho_{\delta}*m_t \geq 1/2 \inf_{x \in \T^d} \rho_{\delta}*m_{t_0} \geq 1/2 \inf_{x \in \T^d} m_{t_0} \geq c/2. 
\label{eq:lowerboundrhodeltam}
\end{equation}
Since we are interested in the behaviour near $t_0$, we can always assume that $t-t_0$ is small enough so that \eqref{eq:lowerboundrhodeltam} holds.

Then we can define
$$ \alpha_t^{\delta} := \frac{\rho_{\delta} * (\alpha^0 m_t)}{\rho_{\delta} * m_t} $$
and $(m_t^{\delta})_{t \geq t_0}$ the solution to
$$ \partial_t m_t^{\delta} = - \div( \alpha_t^{\delta} m_t^{\delta} ) + \sum_{i,j} D^2_{x_i x_j} \bigl( \hat{A}^{ij}(x, Z_t) m_t^{\delta} \bigr) =0, \quad m^{\delta}_{t_0} = m^{\delta}_0 \coloneqq \rho_{\delta} * m_0. $$
Notice that we would have $m_t^{\delta} = \rho_{\delta} * m_t$ if $A$ was constant. Since $\Phi$ touches $\hat{U}^{\delta}$ from above at $(t_0,m_0,z_0)$ we have, for all $t \geq t_0$, almost-surely,
$$\hat{U}(t,Z_t,\rho_{\delta}*m_t) - \Phi(t,Z_t,m_t) \leq \hat{U}(t_0,z_0,\rho_{\delta}*m_0) -\Phi(t_0,z_0,m_0)$$
and therefore
\begin{align*}
     &\Phi(t_0,z_0,m_0) - \E\big[\Phi(t,Z_t,m_t)\big] \\
     &\quad\leq \hat{U}(t_0,z_0,\rho_{\delta} * m_0) - \E\big[\hat{U}(t,Z_t, m_t^{\delta}) \big] + \lip(\hat{U}, W^{-2,\infty}) \norm{ m_t^{\delta}- \rho_{\delta}*m_t}_{-2,\infty}
    \\
    &\quad\leq \E\bigg[\int_{t_0}^t \Big(\int_{\T^d} L(x + Z_t, \alpha_t^{\delta}(x) ) m_t^{\delta}(dx) + \hat{F}(Z_t,m_t^{\delta}) \Big) dt \bigg] + \lip(\hat{U}, W^{-2,\infty}) \norm{ m_t^{\delta}- \rho_{\delta}*m_t}_{-2,\infty},
\end{align*}
with the second inequality coming from Lemma \ref{lem.dppinequality}.
Dividing by $t-t_0$ and letting $t \rightarrow t_0^+$, we get
\begin{align*}
0 &\leq  \int_{\T^d} L \bigl(x + z_0, \alpha_{t_0}^{\delta}(x)  \bigr) dm^{\delta}_0 (x) + \hat{F}(z_0, m^{\delta}_0) + \partial_t \Phi(t_0,z_0,m_0) + \int_{\T^d} \alpha^0(x)\cdot D_m\Phi(t_0,z_0,m_0,x) dm_0(x) \\
&\qquad + A_0 \Delta_z \Phi(t_0,z_0,m_0) +  \int_{\T^d} \tr (\hat{A} (x, z_0) D_x D_m \Phi(t_0,z_0,m_0,x) ) dm_0(x) 
\\
&\qquad + \lip(U,W^{-2,\infty}) \limsup_{t \rightarrow t_0^+} \E \bigg[ \frac{\norm{ m_t^{\delta} - \rho_{\delta} * m_t}_{-2,\infty}}{t-t_0} \bigg].
\end{align*}
Following the argument of \cite{djs2023} Lemma 5.7, with Jensen's inequality we have
\begin{align*} 
\int_{\T^d} L(x+z_0,\alpha_{t_0}^{\delta}(x)) dm_0^{\delta}(x) &\leq \int_{\T^d} L(x+z_0,\alpha^0(x))dm_0(x) + C \delta (\norm{ \alpha_0}_{L^{\infty}} +1) \\
&\leq \int_{\T^d} L(x+z_0,\alpha^0(x))dm_0(x)+  C \delta (\lip(U; \bd_1)+1), 
\end{align*}
where the second line follows from the definition of $\alpha^0$ and our assumptions on $H$. Recalling \eqref{alphazeroprop}, we deduce that
\begin{align} 
\notag &- \partial_t \Phi(t_0,z_0,m_0)  - \int_{\T^d} \tr \bigl(\hat{A} (x,z_0)D_x D_m \Phi(t_0,z_0,m_0,x) \bigr) dm_0(x) - A_0 \Delta_z \Phi(t_0,z_0,m_0)  \\
&\qquad + \int_{\T^d} \hat{H} \bigl(x,z_0, D_m \Phi(t_0,z_0,m_0,x) \bigr) dm_0(x) \nonumber \\
&\leq \hat{F}(z_0,m_0) + C \delta + \lip(U,W^{-2,\infty})  \limsup_{t \rightarrow t_0^+} \E \bigg[ \frac{\norm{ m_t^{\delta} - \rho_{\delta} * m_t}_{-2,\infty}}{t-t_0} \bigg]. 
\label{eq:subsolforUdeltaMay1st}
\end{align}
Since Fatou's Lemma gives
$$\limsup_{t \rightarrow t_0^+} \E \bigg[ \frac{\norm{ m_t^{\delta} - \rho_{\delta} * m_t}_{-2}}{t-t_0} \bigg] \leq \E \bigg[ \limsup_{t \rightarrow t_0^+} \frac{\norm{ m_t^{\delta} - \rho_{\delta} * m_t}_{-2,\infty}}{t-t_0} \bigg], $$
we can conclude using the lemma below. In the case that $A$ is constant, the improvement comes form the fact that in this case $m_t^{\delta} = \rho_{\delta} * m_t$, as noted above.
\end{proof}

\begin{lem} \label{lem.commutator}
    With the notation of the proof above, there is some $C>0$ independent of $\alpha^0$ and $\delta$ such that
    $$ \limsup_{t \rightarrow t_0^+} \frac{\norm{ m_t^{\delta} - \rho_{\delta} * m_t}_{-2,\infty}}{t-t_0} \leq C \delta, \quad \text{a.s.} $$
\end{lem}

\begin{proof}
For notational simplicity, we write $A_t(x) := A(x+Z_t)$.
We are going to argue by duality. For fixed $t >t_0$ let $\phi \in C^2$ with $\norm{\phi}_{2,\infty} \leq 1$ and let $\varphi :[t_0,t] \times \T^d \rightarrow \R  $ be the (classical) solution to
$$ -\partial_s \varphi_s - \alpha^0 \cdot D \varphi_s - \tr (A_s D^2 \varphi_s) = 0, \mbox{ in } (t_0,t) \times \T^d; \quad \varphi_t = \phi.$$
The equation is linear with regular coefficients and therefore, for all $s \in [t_0,t]$, $\varphi_s$ belongs to $C^2(\T^d)$ and
\begin{equation} 
\sup_{s \in [t_0,t]} \norm{ \varphi_s}_{C^2} \leq C
\label{eq:C2estimatevarphis}
\end{equation}
for some $C>0$ independent of $t \in [t_0,T]$ and $\phi$ (exact estimates are given in the proof of the next lemma).
Differentiating $s \mapsto \int_{\T^d} \varphi_s(x) dm_s^{\delta}(x)$, using the equations satisfied by $\varphi$ and $m^{\delta}$ and integrating back in $s$ we find 
$$\int_{\T^d} \phi(x) dm_t^{\delta}(x) = \int_{\T^d} \varphi_{t_0}(x) dm^{\delta}_{t_0} + \int_{t_0}^t \int_{\T^d} (\alpha_s^{\delta}(x) - \alpha^0(x)) \cdot D \varphi_s(x) dm_s^{\delta}(x)ds. $$
Convoluting by $\rho_{\delta}$ the equation for $m$ and using the definition of $\alpha^{\delta}$ we have that $\rho_{\delta} * m_t$ satisfies
$$\partial_s \rho_{\delta} * m_s + \div ( \alpha_{s}^{\delta} \rho_{\delta} * m_s) - \sum_{i,j}D^2_{x_i x_j}(\rho_{\delta}*(A_s^{ij}m_s)) = 0,   $$
and, proceeding as before we get
\begin{align*}
    \int_{\T^d} \phi(x) d\rho_{\delta}*m_t(x) &= \int_{\T^d} \varphi_{t_0}(x) d\rho_{\delta}*m_0(x) + \int_{t_0}^t \int_{\T^d} (\alpha^{\delta}_s(x) - \alpha^0(x)) \cdot D \varphi_s(x) d \rho_{\delta} * m_s(x)ds \\
    &-\int_{t_0}^t \int_{\T^d} \tr \bigl( D^2 \varphi_s (x) (A_s(x) m_s * \rho_{\delta}(x) - \rho_{\delta} * (A_s
     m_s)(x)) \bigr) dxds.
\end{align*}
As a consequence, recalling that $m^{\delta}_{t_0} = \rho_{\delta} * m_0$,
\begin{align*}
    \int_{\T^d} \phi(x) d (m_t^{\delta} - \rho_{\delta} * m_t)(x) &= \int_{t_0}^t \int_{\T^d} (\alpha_s^{\delta}(x) - \alpha^0(x)) \cdot D \varphi_s(x) d(m_s^{\delta} - m_s * \rho_{\delta})(x) ds \\
    &+ \int_{t_0}^t \int_{\T^d} \tr( D^2 \varphi_s (x) (A_s(x) m_s * \rho_{\delta}(x) - \rho_{\delta} * (A_s m_s)(x))) dxds.
\end{align*} 
For the first term we argue that $x \mapsto (\alpha_s^{\delta}(x) - \alpha^0(x)) \cdot D \varphi_s(x)$ is bounded in $C^1$ with bound uniform in $s \in [t_0,t]$ and independent of $\phi$ with $\norm{\phi}_{2,\infty}\leq 1$. This is a consequence of the definition of $\alpha^0$, the lower bound \eqref{eq:lowerboundrhodeltam} on $\rho^{\delta} * m_t$ and the regularity of $\varphi$ (recall \eqref{eq:C2estimatevarphis}). As a consequence,
 $$\Bigl| \int_{t_0}^t \int_{\T^d} (\alpha_s^{\delta}(x) - \alpha^0(x)) \cdot D \varphi_s(x) d(m_s^{\delta} - m_s * \rho_{\delta})(x) ds \Bigr| \leq C \int_{t_0}^t d_1(m_s^{\delta} , m_s*\rho_{\delta}) ds,  $$
 for some $C>0$ independent of $\phi$ with $\norm{\phi}_{2,\infty} \leq 1$. Since classical arguments (using the stochastic characteristics for instance) show that
 $$ \lim_{s \rightarrow t_0^+} d_1(m_s^{\delta}, m_s*\rho_{\delta}) \leq \lim_{s \rightarrow t_0^+} d_1(m_s^{\delta}, m_{t_0}^{\delta}) + d_1(m_{t_0}^{\delta}, m_s* \rho_{\delta}) =0,$$
we can use the mean-value theorem to conclude that
$$ \lim_{t \rightarrow t_0^+} \frac{1}{t-t_0} \sup_{ \norm{\phi}_{C^2} \leq 1} \Bigl| \int_{t_0}^t \int_{\T^d} (\alpha_s^{\delta}(x) - \alpha^0(x)) \cdot D \varphi_s(x) d(m_s^{\delta} - m_s * \rho_{\delta})(x) ds \Bigr| = 0.$$

For the second term, we have
\begin{align*}
    & \Bigl|\int_{t_0}^t \int_{\T^d} \tr \Bigl( D^2 \varphi_s (x) (A_s(x) m_s * \rho_{\delta}(x) - \rho_{\delta} * (A_s m_s)(x)) \Bigr) dx ds \Bigr|\\
    & \leq \int_{t_0}^t \norm{ D^2 \varphi_s}_{L^{\infty}} \int_{\T^d} | A_s(x) m_s * \rho_{\delta}(x) - \rho_{\delta} * (A_s m_s)(x) | dx ds\\
    &\leq \delta Lip(A) \int_{t_0}^t \norm{ D^2 \varphi_s}_{L^{\infty}} ds.
\end{align*}
The idea is that, for $t$ close to $t_0$, $\norm{ D^2 \varphi_s}_{L^{\infty}} \sim \norm{ D^2 \phi}_{L^{\infty}} \leq 1  $ and therefore we want to show that
\begin{equation} 
\limsup_{t \rightarrow  t_0^+} \frac{1}{t-t_0} \sup_{ \norm{\phi}_{C^2} \leq 1} \int_{t_0}^t \norm{ D^2 \varphi_s}_{L^{\infty}} ds \leq C  
\label{eq:integratedestimateMay1st}
\end{equation}
for some $C$ independent of $\alpha^0$. For fixed $t$ and $\phi$, Feynman-Kac formula gives
$\varphi_s(x) = \E \bigl[ \phi(X_t^{s,x}) \bigr] $
where $(X_u^{s,x})_{s \leq u \leq t}$ is solution to
$$ dX_u^{s,x} = \alpha^0(X_u^{s,x})du + \sqrt{2} \sigma_u(X_u^{s,x})dB_u, \quad u \geq s, \quad X_s^{s,x} = x,$$
where $(B_u)_{u \geq 0}$ is a standard Brownian motion (independent of $(W^0)$) and $\sigma_u ^T \sigma_u = A_u$.
As a consequence, 
$$ D \varphi_s(x) = \E \bigl[ D X_t^{s,x} D \phi(X_t^{s,x}) \bigr], \quad D^2 \varphi_s(x) = \E \bigl[ D^{2} X_t^{s,x} D \phi(X_t^{s,x}) + D X_t^{s,x} D X_t^{s,x} D^2 \phi(X_t^{s,x}) \bigr] $$
and, using that $\norm{\phi}_{C^2} \leq 1$,
\begin{equation} 
\| D^2 \varphi_s(\cdot) \|_{L^{\infty}} \leq \sup_{x} \E \bigl[ |D^2 X_t^{s,x} | \bigr] + \sup_{x} \E \bigl[ |DX_t^{s,x}|^2 \bigr]. 
\label{eq:upperboundhessianvarphis}
\end{equation}
Since
$$ d D X_u^{s,x} = D \alpha^0(X_u^{s,x}) D X_u^{s,x} du + \sqrt{2} D \sigma_u(X_u^{s,x})D X_u^{s,x} dB_u, \quad DX_s^{s,x} = I_d,$$
we have, for $\beta = 2,4,$
\begin{align*}
    d|D X_u^{s,x}|^{\beta} \leq  \beta |D X_u^{s,x}|^{\beta -2} D X_u^{s,x} d D X_u^{s,x} +  \beta (\beta-1)|D  \sigma_u(X_u^{s,x})|^2 |D X_u^{s,x}|^{\beta} du
\end{align*}
and
\begin{align*}
    d \E\bigl[ |D X_u^{s,x}|^{\beta} \bigr] \leq \beta \norm{ D \alpha^{0}}_{L^{\infty}}\E\bigl[ |D X_u^{s,x}|^{\beta} \bigr] du + C \E\bigl[ |D X_u^{s,x}|^{\beta} \bigr] du.
\end{align*}
Grönwall's lemma therefore leads to
\begin{equation} 
\sup_{x}\E\bigl[ |D X_u^{s,x}|^{\beta} \bigr] \leq C e^{( u-s) \norm{D \alpha^0}_{L^{\infty}}}. 
\label{eq:gradestimateMay1st}
\end{equation}
We go on with the second order derivative of $ x \mapsto X_t^{s,x}$. We have (differentiating, to alleviate notation, as if $d$ were equal to $1$)
\begin{align*}
    dD^2 X_u^{s,x} &= D^2 \alpha^0(X_u^{s,x}) D X_u^{s,x} DX_u^{s,x} du + D \alpha^0(X_u^{s,x}) D^2 X_u^{s,x} du \\
    &+ \sqrt{2} D^2 \sigma_u(X_u^{s,x}) D X_u^{s,x} D X_u^{s,x} dB_u + \sqrt{2} D \sigma_u(X_u^{s,x}) D^2 X_u^{s,x} dB_u, \quad D^2 X_s^{s,x} = 0. 
\end{align*}
And therefore,
\begin{align*} 
d \E \bigl[ |D^2 X_u^{s,x}|^2 \bigr]  &\leq 2 \norm{ D^2 \alpha^0}_{L^{\infty}}  \E \bigl[ |D^2 X_t^{s,x}| | D X_u^{s,x}|^2 \bigr]du \\
&+ 2\norm{ D \alpha^0}_{L^{\infty}}  \E \bigl[ |D^2 X_u^{s,x}|^2 \bigr] du \\
& + 2 \norm{ D^2 \sigma_u}^2_{L^{\infty}} \E \bigl[ |DX_u^{s,x}|^4\bigr]du + 2 \norm{ D \sigma_u}^2_{L^{\infty}} \E \bigl[ |D^2X_u^{s,x}|^2 ] du.
\end{align*}
Using Young's inequality and Grönwall's Lemma we get, thanks to \eqref{eq:gradestimateMay1st}, for some $C_0$ possibly depending on the derivatives of $\alpha^0$ and changing from line to line,
\begin{align*}
\E \bigl[ |D^2 X_t^{s,x}|^2 \bigr] \leq C_0e^{C_0(t-s)} \int_{s}^t\E \bigl[ |DX_u^{s,x}|^4\bigr]du \leq C_0 \int_{s}^t e^{C_0(t-s + u-s)} du
\end{align*}
But 
\begin{align*} 
 \lim_{t \rightarrow t_0^+} &\frac{1}{t-t_0} \int_{t_0}^t \int_{s}^t e^{C_0(t-s+u-s)} du ds  = \lim_{t \rightarrow t_0^+} \frac{1}{t-t_0} \int_{t_0}^t e^{C_0(t-s)} \frac{1}{C_0} (e^{C_0(t-s)}-1) ds  \\ 
&= \lim_{t \rightarrow t_0^+} \frac{1}{t-t_0} \bigl( \frac{1}{2 C_0^2} (e^{2C_0(t-t_0)} -1) - \frac{1}{C_0^2} (e^{C_0(t-t_0)} - 1) \bigr)=0.
\end{align*}
Similarly, using \eqref{eq:gradestimateMay1st},
\begin{align*} 
\frac{1}{t-t_0}  \int_{t_0}^t \sup_{x} \E \bigl[ | D X_t^{s,x}|^2 \bigr] &\leq \frac{C}{t-t_0} \int_{t_0}^t e^{(t-s) \norm{D \alpha^0}_{L^{\infty}}} ds \\
& \rightarrow^{t \rightarrow 0^+} C
\end{align*}
but, here, $C$ does not depend on $\alpha^0$. Recalling \eqref{eq:upperboundhessianvarphis} and observing that $\sup_{x} \E \bigl[ | D^2 X_t^{s,x}| \bigr] \leq 1 + \sup_{x} \E \bigl[ | D^2 X_t^{s,x}|^2 \bigr]$ we deduce \eqref{eq:integratedestimateMay1st}. To conclude there is $C>0$ independent of $\alpha^0$ such that
$$ \limsup_{t \rightarrow t_0^+} \frac{\norm{m_t^{\delta} - \rho_{\delta}*m_t }_{-2}}{t-t_0} \leq C \delta. $$
Getting back to \eqref{eq:subsolforUdeltaMay1st} this concludes the proof of the Lemma.
\end{proof}

We now turn our attention to $\hat{U}^{\delta, \eps}$. We have the following lemma concerning the regularity of $\hat{U}^{\delta,\eps}$.

\begin{lem} \label{lem.ovureg}
    There is a constant $C$ such that for all $\epsilon < \frac{\delta^{2s^*}}{C} $, we have 
    \begin{enumerate}
        \item $\sup_{t,z,m} |\hat{U}^{\delta,\epsilon}(t,z,m) - \hat{U}^{\delta}(t,z,m)| \leq C \epsilon \delta^{-2(s^*-1)}$, 
        \item $\R^d \times \cP(\T^d) \ni (z,m) \mapsto \hat{U}^{\delta, \epsilon}(t,z,m)$ is $C^{1,1}$ with respect to $H^{-s}$, uniformly in $t$, with the bound
        \begin{align*}
            \sup_t \big[ \hat{U}^{\delta, \epsilon}(t,\cdot,\cdot) \big]_{C^{1,1}(\R^d \times H^{-s^*})} \leq \frac{C}{\epsilon},
        \end{align*}
        \item for fixed $t$, the derivatives of $\hat{U}^{\delta, \epsilon}$ in $(z,m)$ are given by 
        \begin{align*}
            D_z \hat{U}^{\delta, \epsilon}(t,z,m) = \frac{1}{\epsilon} (z_{\epsilon} - z_{0}), \quad \frac{\delta \hat{U}^{\delta,\epsilon}}{\delta m}(t,z,m,\cdot) = \frac{1}{\epsilon}(m_{\epsilon} - m)^*,
        \end{align*}
        where $(z_{\eps}, m_{\eps}) \in \R^d \times \cP(\T^d)$ is the unique optimizer in the optimization problem defining $\hat{U}^{\delta, \eps}(t,z,m)$. Moreover, $D_z \hat{U}^{\delta, \epsilon}$ and $\frac{\delta \hat{U}^{\delta,\epsilon}}{\delta m}$ are both jointly continuous in $(t,z,m)$ and we have $|z_{\eps} - z| \leq C \eps$, with $C$ independent of $\delta, \eps$, and $(t,z,m)$. 
        \item $\hat{U}^{\delta,\eps} \in \text{Lip}\big([0,T] \times \T^d \times \cP(\T^d); H^{-s^*} \big)$, and
        \begin{align*}
            |\hat{U}^{\delta, \epsilon}(t,z,m) - \hat{U}^{\delta, \epsilon}(t,z',m')| \leq C \big(\delta^{-(s^*-1)} \|m - m'\|_{-s^*} + |z -z'| \big)
        \end{align*}
        for all $t \in [0,T]$, $z,z' \in \R^d$, $m,m' \in \cP(\T^d)$.
    \end{enumerate}
\end{lem}

\begin{proof}
    The proof is almost identical to that of \cite[Proposition 4.3]{ddj2023}. We omit the details.
\end{proof}

The following lemma gives a finer property of the optimization problem defining $\hat{U}^{\delta, \eps}(t,z,m)$.

\begin{lem} \label{lem.linfbound}
  There is a constant $C$ such that for any $0 < \epsilon < \frac{1}{C} \delta^{2s^*}$, and any $(t_0,z_0,m_0) \in [0,T] \times \R^d \times \cP(\T^d)$ satisfying $m_0 \geq C \epsilon \delta^{-(2s^* -1)} \leb$, there is a unique minimizer $(z_{\epsilon}, m_{\epsilon})$ for the problem defining $\hat{U}^{\delta,\epsilon}(t_0,z_0,m_0)$, which satisfies
    \begin{align} \label{mepsbound}
        \|m_{\epsilon} - m_0 \|_{\infty} \leq C\epsilon \delta^{-(2s^*  -1)}, 
    \end{align}
    and moreover we have
    \begin{align} \label{w2inflip}
         \norm{\frac{\delta \hat{U}^{\delta, \epsilon }}{\delta m}(t_0,z_0,m_0, \cdot)}_{1,\infty} \leq \lip(U; W^{-1,\infty}), \quad  \norm{\frac{\delta \hat{U}^{\delta, \epsilon }}{\delta m}(t_0,z_0,m_0, \cdot)}_{2,\infty} \leq \lip(U; W^{-2,\infty}). 
    \end{align}
\end{lem}

\begin{proof} 
    As in the proof of \cite[Proposition 4.4]{ddj2023}, we can can first mollify $U$ to produce a sequence $U^n$ which are smooth in $m$ for each fixed $t$, satisfy the same estimates as $U$, and converge uniformly to $U$. This means in particular that the functions $\hat{U}^{n,\delta}$, defined like $\hat{U}^{\delta}$ but with $U^n$ replacing $U$, are smooth and satisfy all the same estimates as $\hat{U}^{\delta}$. Then if we can obtain the analogue of \eqref{mepsbound} for the functions $\hat{U}^{n, \delta, \eps}$ defined like $\hat{U}^{\delta, \eps}$ but with $U^n$ replacing $U$, we can send $n \to \infty$ to obtain the corresponding estimates for $\hat{U}^{\delta, \eps}$. For simplicity, we avoid repeating this mollification argument in detail, and only give the proof under the additional assumption that $U$ is smooth in $m$ for each fixed $t$.
    
    We start with proving \eqref{mepsbound}. For notational simplicity we fix $t \in [0,T]$, and we define $\Phi(z,m) : \R^d \to \cP(\T^d)$ by the formula $\Phi(z,m) = \hat{U}^{\delta}(t,z,m)$. Our goal is to show that there is a constant $C$ such that if $\eps < \frac{1}{C} \delta^{2s^*}$, then the minimizer $(z_{\epsilon}, m_{\epsilon})$ of the optimization problem 
    \begin{align*}
        \sup_{z \in \R^d, m \in \cP(\T^d) } \Big\{ \Phi(z,m) - \frac{1}{2\eps} \Big(|z - z_0|^2 + \|m -m_0\|_{-s}^2 \Big) \Big\}
    \end{align*}
    satisfies the bound \eqref{mepsbound}.
    Because we are assuming that $U$ is smooth, we have an explicit formula
    \begin{align*}
        \frac{\delta \Phi }{\delta m}(z,m,\cdot) = \rho_{\delta} * \frac{\delta \hat{U}}{\delta m}(t,z,m * \rho_{\delta},\cdot), 
    \end{align*}
    from which we can deduce the bound
    \begin{align} \label{rbound}
        \norm{ \frac{\delta \Phi}{\delta m}(z,m,\cdot) }_{C^s} \leq C(s) \delta^{s-1}
    \end{align}
    for each $s > 1$. Now, we observe that by the semiconcavity bound on $\hat{U}^{\delta}$ from Lemma \ref{lem.udeltahatreg}, we can choose $C$ large enough that if $\epsilon < \frac{1}{C} \delta^{2s}$, then the map
    \begin{align*}
       \R^d \times \cP(\T^d) \ni  (z,m) \mapsto \Phi(z,m) - \frac{1}{2\epsilon} \big(|z - z_0|^2 + |m - m_0|^2 \big)
    \end{align*}
    is strictly concave (with respect to the Hilbertain structure on $\R^d \times H^{-s}$), which means that the optimization problem 
    \begin{align} \label{optimization}
        \sup_{z,m} \Big\{\Phi(z,m) - \frac{1}{2\eps}  \big(|z - z_0|^2 + \|m - m_0\|_{-s}^2 \big) \Big\}
    \end{align}
    admits exactly one solution. Moreover, again by strict concavity, if the system of first-order conditions
    \begin{align} \label{zmsystem}
        \begin{cases}
            z -z_0 = \epsilon D_z \Phi(z, m) , \\
            m - m_0 = \epsilon \Big( \frac{\delta \Phi}{\delta m}(z,m,\cdot) \Big)^*,
        \end{cases}
    \end{align}
    admits a solution $(z_{\epsilon}, m_{\epsilon}) \in \R^d \times \cP(\T^d)$, then $(z_{\epsilon}, m_{\epsilon})$ must be the unique optimizer for the optimization problem \eqref{optimization} (see Step 3 in the proof of \cite[Proposition 4.4]{ddj2023} for more details). We note that here $\big( \frac{\delta \Phi}{\delta m}(z,m,\cdot) \big)^* \in H^{-s^*}$ is the dual of  $\frac{\delta \Phi}{\delta m}(z,m,\cdot) \in H^{s^*}$.

    We now argue that if $m_0$ satisfies the lower bound from the statement of the lemma, then we can in fact solve \eqref{zmsystem}. For this, we define a map
    \begin{align*}
        \Psi : \R^d \times \cP(\T^d) \to \R^d \times H^{-s}, \quad \Psi(z,m) = \Big(z_0 + \epsilon D_z \Phi(z,m), \, m_0 + \epsilon \Big( \frac{\delta \Phi}{\delta m}(z,m,\cdot) \Big)^*\Big). 
    \end{align*}
    First, notice that by convention
    \begin{align*}
        \langle \Big( \frac{\delta \Phi}{\delta m}(z,m,\cdot) \Big)^*, 1 \rangle_{-s,s} = \int \frac{\delta \Phi}{\delta m}(z,m,x) dx = 0, 
    \end{align*}
    and moreover using \eqref{rbound} and \eqref{c2sbound}, we have 
    \begin{align} \label{linfbound}
        \norm{ \Big( \frac{\delta \Phi}{\delta m}(z,m,\cdot) \Big)^*}_{\linf} \leq C \norm{ \frac{\delta\Phi}{\delta m}(z,m,\cdot)}_{C^{2s^*}} \leq C \delta^{2s^* - 1}. 
    \end{align}
    Thus, taking $C$ larger if necessary, we find that if $m_0 \geq C \epsilon \delta^{-(2s^*  - 1)}$, then in fact $m_0 + \eps \big( \frac{\delta \Phi}{\delta m}(z,m,\cdot) \big)^*$ is a positive measure with total mass 1, i.e. a probability measure, so that
    \begin{align*}
        \Psi(z,m) \in \R^d \times \cP(\T^d), \quad \text{for each $(z,m) \in \R^d \times \cP(\T^d)$}. 
    \end{align*}
    Because we have assumed that $U$, hence $\Phi$, is smooth, we see that $\Psi$ is continuous, and so by Schauder's fixed point theorem it has a fixed point $(z_{\eps}, m_{\eps})$, which by \eqref{linfbound} must satisfy 
    \begin{align} \label{mndeltaepsbound}
        \|m_{\eps} - m_0\|_{\linf} \leq C \epsilon \delta^{-(2s^* - 1)}. 
    \end{align}
    Since by design $(z_{\eps}, m_{\eps})$ solves the system \eqref{zmsystem}, we conclude by our earlier discussion that $(z_{\eps}, m_{\eps})$ is the unique optimizer for the problem \eqref{optimization}. Moreover, we have the identity
    \begin{align*}
        \frac{\delta \Phi }{\delta m}(z_{\eps} ,m_{\eps} ,\cdot) = \frac{1}{\eps} (m_{\eps} - m_0)^*, 
    \end{align*}
    from which we infer 
    \begin{align} \label{w2infbound}
     \norm{ \frac{1}{\eps} (m_{\eps} - m_0)^*}_{1,\infty} \leq \lip(\Phi; W^{-1,\infty}) \leq \lip(U; W^{-1,\infty}), \nonumber  \\ 
        \norm{ \frac{1}{\eps}
        (m_{\eps} - m_0)^*}_{2,\infty} \leq \lip(\Phi; W^{2,\infty}) \leq \lip(U; W^{-2,\infty}), 
    \end{align}
    which yields \eqref{w2inflip}. This completes the proof.
\end{proof}

Next, we need to check the extent to which $\hat{U}^{\delta, \epsilon}$ is a subsolution to the relevant PDE. 

\begin{lem} \label{lem.uhatdeltaepssubsol}
There is a constant $C$ with the following property: if $\Phi(t,z,q) : (0,T) \times \R^d \times H^{-s^*} \to \R$ is $C^{1,2,1}$ (with regularity in $q$ understood in the Hilbertian sense) and 
\begin{align*}
    \hat{U}^{\delta, \eps}(t_0,z_0,m_0) - \Phi(t_0,z_0,m_0) = \sup_{(t,z,q) \in [0,T] \times \R^d \times H^{-s}} \Big( \hat{U}^{\delta, \eps}(t,z,q) - \Phi(t,z,q)  \Big) 
\end{align*}
for some $(t_0,z_0,m_0) \in (0,T) \times \R^d \times \cP(\T^d)$ satisfying $m_0 \geq C \eps \delta^{-(2s^*- 1)}$, then
     \begin{align} \label{subsoldeltaeps}
        &- \partial_t \Phi(t_0,z_0,m_0) - A_0 \Delta_z \Phi(t_0,z_0,m_0) - \int_{\T^d}\tr\big( \hat{A}(x_0,z) D_{xm} \Phi(t_0,z_0,m_0,x) \big) m_0(dx)  \nonumber \\
        &\qquad  + \int_{\T^d} \hat{H}\big(x,z_0,D_m \Phi(t_0,z_0,m_0,x)\big) m_0(dx) \leq \hat{F}(z_0,m_0) + C ( 1 + \lip(U ; W^{-2,\infty}) )\big(\delta + \epsilon \delta^{-(2s^* - 1)} \big).
     \end{align}
     If $A$ is constant, then the dependence on $\lip(U; W^{-2,\infty})$ can be removed, i.e. the right-hand side of \eqref{subsoldeltaeps} can be replaced by $C\big(\delta + \epsilon \delta^{-(2s^* - 1)} \big)$.
    
\end{lem}

\begin{proof}
    We let $C$ be large enough that the conclusion of Lemma \ref{lem.linfbound} holds, and we suppose that a $\Phi$ touches $\hat{U}^{\delta, \eps}$ from above at a point $(t_0,m_0,z_0) \in (0,T) \times \T^d \times \cP(\T^d)$ such that $m_0 \geq C \epsilon  \delta^{-(2s^*- 1)} $. Now let $(z_{\delta, \eps}, m_{\delta, \eps}) \in \T^d \times \cP(\T^d)$ be a point where the supremum 
    \begin{align*}
        \sup_{z \in \R^d,m \in \cP(\T^d)} \bigg\{\hat{U}^{\delta}(t_0,z,m) - \frac{1}{\eps} \big(|z - z_0|^2 + \|m - m_0\|_{-s^*}^2 \big) \bigg\}
    \end{align*}
    is achieved. Standard arguments from the theory of viscosity solution show that the function $\Psi(t,z,m) = \Phi(t, z_0 - z_{\delta, \eps} + z, m_0 - m_{\delta, \eps} + m)$ (here we use that $\Phi$ is defined on $H^{-s}$) touches $\hat{U}^{\delta}$ from above at $(t_0,z_{\delta, \eps}, m_{\delta, \eps})$. By Lemma \ref{lem.udeltahatsubsol} (which applies because $s^* > d/2 + 2$ implies that the restriction of a function in $C^{1,2,1}((0,T) \times \R^d \times H^{-s^*})$ is also in $\cC^{1,2,2}_p((0,T) \times \R^d \times \cP(\T^d))$), it follows that 
    \begin{align}
    &- \partial_t \Phi(t_0,z_0,m_0) - A_0 \Delta_z \Phi(t_0, z_0,m_0) - \int_{\T^d} \tr\big( \hat{A}(x,z_{\delta, \eps}) D_{xm} \Phi(t_0,z_0,m_0) )m_{\delta, \eps}(dx) \nonumber  \\ &\qquad \qquad + \int_{\T^d} \hat{H}\Big(x,z_{\delta, \eps}, D_m \Phi(t_0,z_0,m_0,x) \Big) m_{\delta, \eps}(dx) \nonumber \\
    &=- \partial_t \Psi(t_0,z_{\delta, \eps},m_{\delta, \eps}) - A_0 \Delta_z \Psi(t_0, z_{\delta, \eps},m_{\delta, \eps}) - \int_{\T^d} \tr\big( \hat{A}(x,z_{\delta, \eps}) D_{xm} \Psi(t_0,z_{\delta, \eps},m_{\delta, \eps}) )m_{\delta, \eps}(dx) \nonumber  \\
    &\qquad \qquad + \int_{\T^d} \hat{H}\Big(x,z_{\delta, \eps}, D_m \Psi(t_0,z_{\delta, \eps},m_{\delta, \eps},x) \Big) m_{\delta, \eps}(dx)
       \nonumber  \leq \hat{F}(z_{\delta, \eps}, m_{\delta, \eps}) + C(1 + \lip(U; W^{-2,\infty}))\delta. 
    \end{align}
   We deduce that 
    \begin{align} \label{errorterms}
     &- \partial_t \Phi(t_0,z_0,m_0) - A_0 \Delta_z \Phi(t_0, z_0,m_0) - \int_{\T^d} \tr\big( \hat{A}(x,z_0) D_{xm} \Phi(t_0,z_0,m_0) )m_{0}(dx) \nonumber  \\ &\qquad  + \int_{\T^d} \hat{H}\Big(x,z_0,m_0, D_m \Phi(t_0,z_0,m_0,x) \Big) m_0(dx) \leq \hat{F}(z_{0},m_0) + C(1 + \lip(U; W^{-2,\infty}))\delta
    \nonumber  \\
     &\qquad + I + II + III, 
     \end{align}
     where 
     \begin{align*}
         I &= \int_{\T^d} \tr\big(\hat{A}(x,z_{\delta, \eps}) D_{xm} \Phi(t_0,m_0,z_0,x) \big) m_{\delta, \eps}(dx) - \int_{\T^d} \tr\big(\hat{A}(x,z_0) D_{xm} \Phi(t_0,m_0,z_0,x) \big) m_0(dx), \\
         II &= \int_{\T^d} \hat{H}\Big(x,z_0, D_m \Phi(t_0,z_0,m_0,x) \Big) m_0(dx) - \int_{\T^d} \hat{H}\Big(x,z_{\delta, \eps}, D_m \Phi(t_0,z_0,m_0,x) \Big) m_{\delta, \eps}(dx), \\
         III &= \hat{F}(z_{\delta, \eps}, m_{\delta, \eps}) - \hat{F}(z_0,m_0).
     \end{align*}
     To bound $I$, Lemmas \ref{lem.ovureg} and \ref{lem.linfbound} to get
     \begin{align} \label{Icomp}
         I &\leq C \|D_{xm} \Phi(t_0,m_0,z_0,\cdot)\|_{\linf} \big( |z_{\delta, \eps} - z_0|  + \|m_{\delta, \eps} - m_0\|_{\linf} \big) \leq C  \epsilon \delta^{-(2s^* - 1)} (1 + \text{Lip}(U; W^{-2,\infty})). 
     \end{align}
     For the second error term, we use Assumption \ref{assump.maindegen} and the fact that $$ R \coloneqq \|D_m \Phi(t_0,z_0,m_0,\cdot)\|_{\infty} \leq \lip(U; W^{-1,\infty})$$ by Lemma \ref{lem.linfbound} to find a constant $C$ depending only on the data such that $\hat{H}$ is uniformly Lipschitz and bounded on $\T^d \times \R^d \times B_R$, and thus
     \begin{align} \label{IIcomp}
         II &\leq C \big(|z_{\delta, \eps} - z_0| + \int_{\T^d} \hat{H}\Big(x,z_0, D_m \Phi(t_0,z_0,m_0,x),m_0 \Big) (m_0 - m_{\delta, \eps} )(dx) \nonumber  \\
         &\leq C |z_{\delta, \eps} - z_0| + C \Big(1 + \norm{\hat{H}}_{\linf(\T^d \times \R^d \times B_{R}  )} \Big)\|m_{\delta, \eps} - m_0\|_{\infty} \nonumber \\ 
         &\leq C \eps \delta^{-(2s^* - 1)}.
     \end{align}
     Similarly, Lemmas \ref{lem.ovureg} and \ref{lem.linfbound} easily give 
     \begin{align} \label{IIIcomp}
         III \leq C \big( |z_{\delta, \eps} - z_{\delta}| + \bd_1(m_{\delta, \eps},m_0) \big) \leq C \eps \delta^{-(2s^* - 1)}.
     \end{align}
     By combining \eqref{errorterms}, \eqref{Icomp}, \eqref{IIcomp} and \eqref{IIIcomp}, we obtain \eqref{subsoldeltaeps}. To see the improvement when $A$ is constant, we use the representation of $\frac{\delta \hat{U}^{\delta, \eps}}{\delta m}$ from Lemma \ref{lem.ovureg} to see that 
     \begin{align*}
         I &= \int_{\T^d} \tr \big(A D^2_{xx} \frac{\delta \Phi}{\delta m}(t_0,m_0,z_0,x) \big) (m_{\delta,\eps} - m_0)(dx) = \frac{1}{\eps} \int_{\T^d}  \tr \big(A D^2_{xx} (m_{\delta, \eps} -m_0)^* \big) (m_{\delta,\eps} - m_0)(dx) 
         \\
         &= \frac{1}{\eps} \langle \tr( A D^2_{xx} (m_{\delta, \eps} - m_0)^*, m_{\delta, \eps} - m_0 \rangle_s \leq 0, 
     \end{align*}
     where the last inequality follows from the non-positivity of the operator 
     $\phi \mapsto \tr(A D^2_{xx} \phi)$ on $H^s$, which is an easy generalization of Lemma 5.5 in \cite{cdjs2023}.
\end{proof}

Finally, we are ready to study the function $\hat{U}^{\delta,\epsilon,\lambda} : [0,T] \times \R^d \times \cP(\T^d) \to \R$ defined for $\delta, \epsilon > 0$, $\lambda \in (0,1)$ by \eqref{def.udeltaepshatlambda}.

\begin{lem} \label{lem.udeltaepslambdareg} 
There is a constant $C$ such that for $\epsilon < \frac{\delta^{2s}}{C}$, we have
\begin{enumerate}
    \item $\sup_{t,z,m} |\hat{U}^{\delta, \eps, \lambda}(t,z,m)- \hat{U}(t,z,m)| \leq C \big(\delta + \lambda + \epsilon \delta^{-2(s^*-1)}\big)$, 
    \item $\R^d \times \cP(\T^d) \ni (z,m) \mapsto \hat{U}^{\delta,\eps, \lambda}(t,z,m)$ is $C^{1,1}$ with respect to $H^{-s^*}$, uniformly in $t$, with 
    \begin{align*}
        \sup_t \big[\hat{U}^{\delta, \eps,\lambda}(t,\cdot,\cdot) \big]_{C^{1,1}(\R^d \times H^{-s^*})} \leq \frac{C}{\eps},
    \end{align*}
    \item $\hat{U}^{\delta,\eps,\lambda} \in \text{Lip}\big([0,T] \times \R^d \times \cP(\T^d); H^{-s^*} \big)$, and
    \begin{align*}
        |\hat{U}^{\delta, \eps, \lambda}(t,z,m) - \hat{U}^{\delta, \eps, \lambda}(t,z',m
        )| \leq C \big(\delta^{-(s^*-1)} \|m - m'\|_{-s^*} + |z-z'| \big)
    \end{align*}
    for each $t\in [0,T]$, $z,z' \in \R^d$, $m,m' \in \cP(\T^d)$.
\end{enumerate}
\end{lem}

\begin{proof}
 For (1), we fix $(t,z,m)$ and set $m_{\lambda} = \lambda \leb+ (1-\lambda) m$, and then we simply note that by Lemmas \ref{lem.udeltahatreg} and \ref{lem.ovureg}, we have
 \begin{align*}
     |\hat{U}^{\delta, \eps, \lambda}(t,z,m)& - \hat{U}^{\delta}(t,z,m)| = |\hat{U}^{\delta, \eps}(t,z, m_{\lambda} ) -  \hat{U}^{\delta}(t,z,m)| \\
     &\leq |\hat{U}^{\delta, \eps}(t,z, m_{\lambda} ) -  \hat{U}^{\delta}(t,z,m_{\lambda})| + |\hat{U}^{\delta}(t,z, m_{\lambda} ) -  \hat{U}^{\delta}(t,z,m)| \\& \leq C\epsilon \delta^{-2(s^*-1)} + \lip(\hat{U}^{\delta} ; W^{-1,\infty})\bd_1(m_{\lambda}, m) \\
     &\leq C\big(\delta^{-(s^*-1)} \|m - m'\|_{-s^*} + |z-z'| \big).
 \end{align*}
 Then using the bound on $|\hat{U}^{\delta} - \hat{U}|$ from Lemma \ref{lem.udeltahatreg} completes the proof of (1).
 Points (2) and (3) are easy consequences of the regularity of $\hat{U}^{\delta, \eps}$ from Lemma \ref{lem.ovureg}.
\end{proof}

Finally, we estimate the amount by which $\hat{U}^{\delta, \eps, \lambda}$ fails to be a subsolution of \eqref{hjbz}.

\begin{lem} \label{lem.udeltaepslambdasubsol}
There is a constant $C > 0$ with the following property: for each each $\delta > 0$, $0 < \epsilon < \frac{1}{C} \delta^{2s^*}$, and $1/2  > \lambda > C \epsilon \delta^{-(2s^*  -1)}$, and any function $\Phi \in C^{1,2,1}((0,T) \times \R^d \times H^{-s^*})$ such that 
\begin{align*}
    \hat{U}^{\delta, \eps, \lambda}(t_0,z_0,m_0) - \Phi(t_0,z_0,m_0) = \sup_{t,z,q \in [0,T] \times \R^d \times H^{-s^*}} \Big( \hat{U}^{\delta, \eps, \lambda}(t,z,q) - \Phi(t,z,q)\Big)
\end{align*}
for some $(t_0,z_0,m_0) \in [0,T] \times \T^d \times \cP(\T^d)$, we have
\begin{align} \label{subsoludeltaepslambda}
         &- \partial_t \Phi(t_0,z_0,m_0) - A_0 \Delta_z \Phi(t_0,z_0,m_0) - \int_{\T^d}\tr\big( \hat{A}(x_0,z) D_{xm} \Phi(t_0,z_0,m_0,x) \big) m_0(dx) \nonumber  \\
        &\qquad  + \int_{\T^d} \hat{H}\big(x,z_0,D_m \Phi(t_0,z_0,m_0,x)\big) m_0(dx)\leq \hat{F}(z_0,m_0) + C \big(1 + \lip(U; W^{-2,\infty}) \big)\Big(\delta + \lambda + \epsilon \delta^{-(2s^* - 1)} \Big).
    \end{align}
    If $A$ is constant, then the dependence on $\lip(U; W^{-2,\infty})$ can be removed, i.e. the right-hand side of \eqref{subsoludeltaepslambda} can be replaced by $C\big(\delta + \lambda + \epsilon \delta^{-(2s^* - 1)}\big) $.
\end{lem}

\begin{proof}
   Choose $C$ large enough that the conclusion of Lemma \ref{lem.uhatdeltaepssubsol} holds, and let $\delta, \eps, \lambda$, $\Phi$, $(t_0,z_0,m_0)$ be as in the statement of the lemma. Define
    \begin{align*}
        \Phi^{\lambda}(t,z,m) = \Phi\Big(t,z, m^{\lambda}\Big), \quad m^{\lambda} = \frac{m}{1- \lambda} - \frac{\lambda}{1-\lambda} \leb.
    \end{align*}
    Notice that $m^{\lambda}$ may not be a probability measure, but this is not a problem since $\Phi$ is defined on $[0,T] \times \R^d \times H^{-s^*}$.
    Then $\Phi^{\lambda}$ touches $\hat{U}^{\delta, \eps}$ from above at $(t_0,z_0, (m_0)_{\lambda})$, where for $m \in \cP(\T^d)$, we have defined
    \begin{align*}
        m_{\lambda} = \lambda \leb +(1-\lambda) m.
    \end{align*}
    The key point here is that $m_0 = (m_0^{\lambda})_{\lambda} = \big((m_0)_{\lambda}\big)^{\lambda}$. 
    Notice that
    \begin{align*}
        \partial_t \Phi^{\lambda}(t,z,m) = \partial_t \Phi(t,z,m^{\lambda}), \,\,
        D^2_{zz} \Phi^{\lambda}(t,z,m) = D^2_{zz} \Phi(t,z,m^{\lambda}), \,\, D_m \Phi^{\lambda}(t,z,m,\cdot) = \frac{1}{1-\lambda} D_m \Phi(t,z,m^{\lambda},\cdot),
    \end{align*}
    From Lemma \ref{lem.uhatdeltaepssubsol}, we get 
    \begin{align*}
        &- \partial_t \Phi^{\lambda}(t_0,z_0,(m_0)_{\lambda}) - A_0 \Delta_z \Phi^{\lambda}(t_0,z_0,(m_0)_{\lambda})  - \int_{\T^d} \tr\big( \hat{A}(x,z_0) D_{xm} \Phi^{\lambda}(t_0,z_0,(m_0)_{\lambda},x) \big) (m_0)_{\lambda} (dx)  \\
        &\quad \qquad + \int_{\T^d} \hat{H}\big(x,z_0,D_m \Phi^{\lambda}(t_0,z_0,(m_0)_{\lambda},x)\big) (m_0)_{\lambda}(dx) \\
        & \leq \hat{F}(z_{0}, (m_0)_{\lambda}) + C \big(1 + \lip(U; W^{-2,\infty}) \big)\Big(\delta  + \epsilon \delta^{-(2s^* - 1)} \Big), 
    \end{align*}
    and so 
    \begin{align*}
        &- \partial_t \Phi(t_0,z_0,m_0) - A_0 \Delta_z \Phi(t_0,z_0,m_0)  - \int_{\T^d} \tr\big( \hat{A}(x,z_0) D_{xm} \Phi(t_0,z_0,m_0,x) \big) m_0(dx)  \\
        &\qquad \qquad + \int_{\T^d} \hat{H}\big(x,z_0,D_m \Phi(t_0,z_0,m_0,x) \big) m_0(dx) \\
        &= - \partial_t \Phi^{\lambda}(t_0,z_0,(m_0)_{\lambda}) - A_0 \Delta_z \Phi^{\lambda}(t_0,z_0,(m_0)_{\lambda})  - (1-\lambda) \int_{\T^d} \tr\big( \hat{A}(x,z_0) D_{xm} \Phi^{\lambda}(t_0,z_0,(m_0)_{\lambda},x) \big) m_0(dx)  \\
        &\qquad \qquad + \int_{\T^d} \hat{H}\big(x,z_0, (1-\lambda) D_m \Phi^{\lambda}(t_0,z_0,m_{\lambda},x) \big) m_0(dx)
        \\
        &\qquad \leq \hat{F}(z_0,m_0) +  C \big(1 + \lip(U; W^{-2,\infty}) \big)\Big(\delta  + \epsilon \delta^{-(2s^* - 1)} \Big) + I + II + III,
    \end{align*}
    where 
    \begin{align*}
        &I = -\lambda \int_{\T^d} \tr\big( \hat{A}(x,z_0) D_{xm} \Phi^{\lambda}(t_0,z_0,(m_0)_{\lambda},x) \big) \leb(dx), \\
        &II =\int_{\T^d} \hat{H}\big(x,z_0,D_m \Phi^{\lambda}(t_0,z_0,(m_0)_{\lambda},x), m_0 \big) m_{0}(dx)
        \\
        &\qquad-\int_{\T^d} \hat{H}\big(x,z_0, (1-\lambda) D_m \Phi^{\lambda}(t_0,z_0,(m_0)_{\lambda},x), (m_0)_{\lambda} \big) (m_0)_{\lambda}(dx), \\
        &III = \hat{F}(z_0,(m_0)_{\lambda}) - \hat{F}(z_0,m_0).
    \end{align*}
    To bound $I$, we note that in fact since $\Phi^{\lambda}$ touches $\hat{U}^{\delta, \eps}$ from above at $(t_0,z_0,(m_0)_{\lambda})$, and $\hat{U}^{\delta, \eps}$ is smooth in $m$, we have that 
    \begin{align*}
       D_{xm} \Phi^{\lambda}(t_0,z_0,(m_0)_{\lambda},\cdot) = D_{xm} \hat{U}^{\delta, \eps}(t_0,z_0,(m_0)_{\lambda},\cdot), 
    \end{align*}
    where we used that $m_{\lambda}$ has full support. Then, we have by Lemma \ref{lem.linfbound} that $\|D_{xm} \Phi(t_0,z_0,m_0, \cdot)\| \leq C \lip(U; W^{-2,\infty})$. Thus we easily get
    \begin{align*}
        I \leq C \lambda \lip(U; W^{-2,\infty}).
    \end{align*}
    A similar argument gives $II \leq C\lambda$ and $III \leq C \lambda$.
    To get the improved bound when $A$ is constant, we used the improved bound from Lemma \ref{lem.uhatdeltaepssubsol} and notice that when $A$ is constant $I = 0$ by integration by parts. This completes the proof.
\end{proof}

We now give a Lemma which explains that we can understand the subsolution property in Lemma \ref{lem.uhatdeltaepssubsol} in a pointwise sense.

\begin{lem} \label{lem.subsolpointwise}
    Let $\delta, \epsilon$, $\lambda$ and $C$ be as given in the statement of Lemma \ref{lem.udeltaepslambdasubsol}, and let $(t_0,z_0,m_0) \in [0,T] \times \R^d \times \cP(\T^d)$ be a point such that $m_0 \geq C \epsilon \delta^{2s^* - 1}$. Suppose that $a \in \R$, $p \in \R^d$, and $q \in \text{Sym}(\R^{d\times d})$ are such that
    \begin{align*}
        \hat{U}^{\delta, \eps, \lambda}(t,z,m_0) \leq \hat{U}^{\delta, \eps, \lambda}(t_0,z_0,m_0) + a (t-t_0) + p\cdot (z - z_0) + \frac{1}{2} (z-z_0)^T q (z-z_0) + o(|t- t_0| + |z-z_0|^2). 
    \end{align*}
    Then we have 
    \begin{align} \label{subsolpointwise}
        &- a - \tr(q) - \int_{\T^d}\tr\big( \hat{A}(x,z_0) D_{xm} \hat{U}^{\delta, \epsilon, \lambda}(t_0,z_0,m_0,x) \big) m_0(dx) \nonumber \\
        &\qquad  + \int_{\T^d} \hat{H}\big(x,z_0,D_m \hat{U}^{\delta, \epsilon, \lambda}(t_0,z_0,m_0,x)\big) \leq \hat{F}(z_0,m_0) + C \big(1 + \lip(U; W^{-2,\infty}) \big)\Big(\delta + \lambda + \epsilon \delta^{-(2s^*  - 1)} \Big).
    \end{align}
    If $A$ is constant, then the dependence on $\lip(U; W^{-2,\infty})$ can be removed, i.e. the right-hand side of \eqref{subsolpointwise} can be replaced by $C\big(\delta + \lambda + \epsilon \delta^{-(2s^* - 1)} \big) $.
\end{lem}

\begin{rmk}
    Recall that by Lemma \ref{lem.ovureg} and the relationship between smoothness on $H^{-s^*}$ and $\cP(\T^d)$ discussed in section \ref{sec: notation}, the derivatives $D_m \hat{U}^{\delta, \eps, \lambda}$ and $D_{xm} \hat{U}^{\delta, \eps, \lambda}$ appearing in \eqref{subsolpointwise} exist in a classical sense.
\end{rmk}

\begin{proof}[Proof of Lemma \ref{lem.subsolpointwise}]
We fix a small parameter $\gamma > 0$, and we use Lemma \ref{lem.ovureg} (in particular the $C^{1,1}$ bounds on $\hat{U}^{\delta, \eps}$ and the joint continuity of the derivatives $D_z \hat{U}^{\delta, \eps}$ and $\frac{\delta \hat{U}^{\delta, \eps}}{\delta m}$) to obtain
    \begin{align*}
        \hat{U}^{\delta, \eps, \lambda}&(t,z,m) = \hat{U}^{\delta, \eps, \lambda}(t_0,z_0,m_0) + \hat{U}^{\delta, \eps, \lambda}(t,z,m_0) - \hat{U}^{\delta, \eps, \lambda}(t_0,z_0,m_0) + \hat{U}^{\delta, \eps, \lambda}(t,z,m) - 
        \hat{U}^{\delta, \eps, \lambda}(t,z,m_0) \\
        &\leq \hat{U}^{\delta, \eps, \lambda}(t_0,z_0,m_0) + a (t-t_0) + p \cdot(z-z_0) 
      + \frac{1}{2} (z - z_0)^T q (z-z_0)  \\ &\qquad  + o\big(|t-t_0| + |z-z_0|^2\big)
         + \langle \nabla_{-s^*} \hat{U}^{\delta, \eps, \lambda} (t,z,m_0), m - m_0 \rangle_{-s^*} + o\big(\|m - m_0\|_{-s^*}\big) \\
         &\leq \hat{U}^{\delta, \eps, \lambda}(t_0,z_0,m_0) + a (t-t_0) + p \cdot(z-z_0) 
      + \frac{1}{2} (z - z_0)^T q (z-z_0)  \\ &\qquad  + o\big(|t-t_0| + |z-z_0|^2 + \|m -m_0\|_{-s^*} \big)
      + \langle \nabla_{-s^*} \hat{U}^{\delta, \eps, \lambda} (t_0,z_0,m_0), m - m_0 \rangle_{-s^*} \\
         &\qquad + C |z-z_0| \|m-m_0\|_{-s^*} + o\big(|t-t_0| + |z-z_0|^2 + \|m - m_0\|_{-s^*}\big)
         \\ 
         &\leq \hat{U}^{\delta, \eps, \lambda}(t_0,z_0,m_0) + a(t-t_0) + p\cdot(z-z_0)  + \frac{1}{2} (z - z_0)^T q (z-z_0)  \\
        &\qquad
         + \langle \nabla_{-s^*} \hat{U}^{\delta, \eps, \lambda} (t_0,z_0,m_0), m - m_0 \rangle_{-s^*}  + \frac{\gamma}{2} |z-z_0|^2 + o\big(|t-t_0| + |z-z_0|^2 + \|m - m_0\|_{-s^*}\big).
    \end{align*}
    By classical arguments from the theory of viscosity solutions (see e.g. page 18 of \cite{Barles1994} for the proof in finite dimensions, which works also in Hilbert spaces) it follows that there exists a smooth function $\Phi(t,z,q) : [0,T] \times \R^d \times H^{-s^*} \to \R$
     such that $\Phi$ touches $\hat{U}^{\delta, \eps}$ from above at $(t_0,z_0,m_0)$, and 
    \begin{align*}
        &\partial_t \Phi(t_0,z_0,m_0) = a , \,\, D_z \Phi(t_0,z_0,m_0) = p, \\ &D^2_{zz} \Phi(t_0,z_0,m_0) = q + \gamma I_{d \times d},  \,\, \nabla_{-s^*} \Phi(t_0,z_0,m_0) = \nabla_{-s^*} \hat{U}^{\delta, \eps}(t_0,z_0,m_0).
    \end{align*}
    Applying Lemma \ref{lem.uhatdeltaepssubsol} to $\Phi$, we find that 
    \begin{align*}
        &- a - \tr(q) - d\gamma + \int_{\T^d}\tr\big( \hat{A}(x,z_0) D_{xm} \hat{U}^{\delta, \epsilon}(t_0,z_0,m_0,x) \big) m_0(dx)  \\
        &\qquad  + \int_{\T^d} \hat{H}\big(x,z_0,D_m \hat{U}^{\delta, \epsilon}(t_0,z_0,m_0,x)\big) \leq  \hat{F}(z_0,m_0) + C \big(1 + \lip(U; W^{-2,\infty}) \big)\Big(\delta + \lambda + \epsilon \delta^{-(2s +  - 1)} \Big), 
    \end{align*}
    and we send $\gamma$ to get the desired bound. Lemma \ref{lem.uhatdeltaepssubsol} also gives the improved bound when $A$ is constant, which completes the proof.
\end{proof}

\subsection{The projection argument}

Define $\hat{U}^{\delta, \eps, \lambda,N} : [0,T] \times \R^d \times (\T^d)^N \to \R$ by the formula
\begin{align}
\label{projectiondef}
\hat{U}^{\delta, \eps, \lambda, N}(t,z,\bx) = \hat{U}^{\delta, \eps, \lambda}(t,z,m_{\bx}^N). 
\end{align}
We are now going to compare $\hat{U}^{\delta, \eps, \lambda, N}$ to the function $\hat{V}^N : [0,T] \times \R^d \times (\T^d)^N \to \R$ given by 
\begin{align}
    \label{vnhatdef}
    \hat{V}^N(t,z,\bx) = V^N(t, \bx^z), \quad \bx^z = (x^1 + z, ...,x^N + z). 
\end{align}

We will need to know that $\hat{V}^N$ solves \eqref{hjbhatn}: 

\begin{lem} \label{lem.hatvneqn}
The function $\hat{V}^N$ is the unique viscosity solution to \eqref{hjbhatn}.
\end{lem}

\begin{proof}
    It is a simple exercise to check that since $V^N$ solves \eqref{hjbn}, $\hat{V}^N$ solves \eqref{hjbhatn}.
\end{proof}

\begin{prop} \label{prop.comparison}
There is a constant $C$ such that for each $\delta > 0$, $\epsilon < \frac{1}{C} \delta^{2s^*}$, and $1/2 \geq \lambda > C \epsilon \delta^{-(2s^* -1)}$, we have 
\begin{align} \label{projectionest}
    \hat{U}^{\delta, \epsilon, \lambda,N}(t,z,\bx) \leq \hat{V}^N(t,z,\bx) + \frac{C}{N \epsilon} + C \big(1 + \lip(U; W^{-2,\infty}) \big)\Big(\delta + \lambda + \epsilon \delta^{-(2s^*  - 1)} \Big).
\end{align}
If $A$ is constant, then the dependence on $\lip(U; W^{-2,\infty})$ can be removed, i.e. the right-hand side of \eqref{subsoludeltaepslambda} can be replaced by $\hat{V}^N + C/(N\eps) + C\big(\delta + \lambda + \epsilon \delta^{-(2s^* - 1)} \big) $.
\end{prop}

\begin{proof}
     Let $C$ be as in the statement of Lemma \ref{lem.udeltaepslambdasubsol}, and let $\delta > 0$, $\eps < \frac{\delta^{2s}}{C(s)}$, $\frac{1}{2} > \lambda > C \eps \delta^{-(2s^* - 1)}$ be given. Our goal is to show that $\hat{U}^{\delta, \eps, \lambda,N}$ is a viscosity subsolution to \eqref{hjbhatn} up to a certain error term. To this end, we suppose that a smooth function $\phi(t,z,\bx) : (0,T) \times \R^d \times (\T^d)^N \to \R$ touches $\hat{U}^{\delta, \eps, \lambda,N}$ from above at a point $(t_0, z_0, \bx_0)$, with $\bx_0 = (x_0^1,...,x_0^N)$. Obviously, it follows that $(a,p,q)$ given by
    \begin{align*}
        a = \partial_t \phi(t_0,z_0,\bx_0), \,\, p = D_z \phi(t_0,z_0,\bx_0), \,\, q = D^2_{zz} \phi(t_0,z_0,\bx_0)
    \end{align*}
    satisfy the hypotheses of Lemma \ref{lem.subsolpointwise} (at the point $(t_0,z_0,m_{\bx_0}^N)$), and so we get 
     \begin{align} \label{phisubsol}
     &- \partial_t \phi(t_0,z_0,\bx_0) - A_0 \Delta_z \phi(t_0,z_0,\bx_0) - \frac{1}{N} \sum_{i = 1}^N \tr\big( \hat{A}(x_0^i,z_0) D_{xm} \hat{U}^{\delta, \epsilon, \lambda}(t_0,z_0,m_{\bx_0}^N,x_0^i) \big) \nonumber \\
        &\qquad \qquad + \frac{1}{N} \sum_{i = 1}^N \hat{H}\big(x^i_0,z_0,ND_{x^i} \phi(t_0,z_0,\bx_0)\big) \nonumber \\
        &= - \partial_t \phi(t_0,z_0,\bx_0) - A_0 \Delta_z \phi(t_0,z_0,\bx_0) - \frac{1}{N} \sum_{i = 1}^N \tr\big( \hat{A}(x_0^i,z_0) D_{xm} \hat{U}^{\delta, \epsilon, \lambda}(t_0,z_0,m_{\bx_0}^N,x_0^i) \big) \nonumber \\
        &\qquad \qquad + \frac{1}{N} \sum_{i = 1}^N \hat{H}\big(x^i_0,z_0,ND_{x^i} \hat{U}^{\delta, \eps, \lambda, N} (t_0,z_0,\bx_0)\big) \nonumber \\
        &= - \partial_t \phi(t_0,z_0,\bx_0) - A_0 \Delta_z \phi(t_0,z_0,\bx_0) - \int_{\T^d}\tr\big( \hat{A}(x,z_0) D_{xm} \hat{U}^{\delta, \epsilon, \lambda}(t_0,z_0,m_{\bx_0}^N,x) \big) m_{\bx_0}^N(dx) \nonumber \\
        &\qquad \qquad + \int_{\T^d} \hat{H}\big(x,z_0,D_m \hat{U}^{\delta, \epsilon, \lambda}(t_0,z_0,m_{\bx_0}^N,x)\big) m_{\bx_0}^N (dx)
        \\
        &\leq \hat{F}(z_0,m_{\bx_0}^N) + C \big(1 + \lip(U; W^{-2,\infty}) \big)\Big(\delta + \lambda + \epsilon \delta^{-(2s^*  - 1)} \Big).
    \end{align}
    To complete the proof, we must compare the matrix $D_{xm} \hat{U}^{\delta, \eps, \lambda}(t_0,z_0,m_{\bx_0}^N, x_0^i)$ to the matrix $D^2_{x^ix^i} \phi(t_0,z_0,\bx_0)$. To this end, we note that by (the proof of) \cite[Lemma 5.4]{ddj2023}, there is a set $A_{t_0,z_0} \subset (\T^d)^N$ of full measure on which the partial derivative 
    \begin{align*}
        \T^d \ni x \mapsto D_m \hat{U}^{\eps, \delta, \lambda}(t_0,z_0, \frac{1}{N} \sum_{j \neq i} \delta_{x^j} + \frac{1}{N} x, x^i)
    \end{align*}
    exists for each $\bx = (x^1,...,x^N) \in A_{t_0,z_0}$ and each $i = 1,...,N$. Moreover, $\bx \mapsto \hat{U}^{\eps, \delta, \lambda,N}(t_0,z_0,\bx)$ is twice differentiable at each $\bx \in A_{t_0,z_0}$, and satisfies 
    \begin{align*}
        D^2_{x^ix^i} \hat{U}^{\eps, \delta, \lambda,N}(t_0,z_0,\bx) = \frac{1}{N} D_{xm} \hat{U}^{\delta, \eps, \lambda}(t_0,z_0,m_{\bx}^N,x^i) + \frac{1}{N^2} R^{N,i}(\bx_0), 
    \end{align*}
    where $R_{t_0,z_0}^{N,i}(\bx) : A_{t_0,z_0} \to \R^{d \times d}$ satisfies 
    \begin{align*}
        \|R^{N,i}\|_{\infty} \leq C(s) \Big[\hat{U}^{\eps, \delta, \lambda}(t_0,z_0,\cdot) \Big]_{C^{1,1}(H^{-s^*})} \leq C(s)/\eps. 
    \end{align*}
    Notice that if the point $(t_0,z_0,\bx_0)$ were such that $\bx_0 \in A_{t_0,z_0}$, then we could immediately conclude that 
    \begin{align*}
        D^2_{x^ix^i} \phi(t_0,z_0,\bx_0) \geq D^2_{x^ix^i} \hat{U}^{\eps, \delta, \lambda,N}(t_0,z_0,\bx_0) = \frac{1}{N} D_{xm} \hat{U}^{\delta, \eps, \lambda}(t_0,z_0,m_{\bx_0}^N,x_0^i) + \frac{1}{N^2} R^{N,i}(\bx_0), 
    \end{align*}
    and combining this with \eqref{phisubsol} would give 
    \begin{align} \label{Aest}
    &- \partial_t \phi(t_0,z_0,\bx_0) - A_0 \Delta_z \phi(t_0,z_0,\bx_0) - \sum_{i = 1}^N \tr\big( \hat{A}(x_0^i,z_0) D^2_{x^ix^i} \phi(t_0,z_0,\bx_0) \big) \nonumber \\
        &\qquad \qquad + \frac{1}{N} \sum_{i = 1}^N \hat{H}\big(x^i_0,z_0,ND_{x^i} \phi(t_0,z_0,\bx_0) \big) \nonumber \\ \leq 
        &- \partial_t \phi(t_0,z_0,\bx_0) - A_0 \Delta_z \phi(t_0,z_0,\bx_0) - \frac{1}{N} \sum_{i = 1}^N \tr\big( \hat{A}(x_0^i,z_0) D_{xm} \hat{U}^{\delta, \epsilon, \lambda}(t_0,z_0,m_{\bx_0}^N,x_0^i) \big) \nonumber \\
        &\qquad \qquad - \frac{1}{N^2} \sum_{i = 1}^N \tr(\hat{A}(x_0^i,z_0) R^{N,i}_{t_0,z_0}(t,\bx))+ \frac{1}{N} \sum_{i = 1}^N \hat{H}\big(x^i_0,z_0,ND_{x^i} \phi (t_0,z_0,\bx_0)\big) \nonumber \\
        &\leq \hat{F}(z_0,m_{\bx_0}^N) + \frac{C}{N\eps} + C\big(1 + \lip(U; W^{-2,\infty}) \big)\Big(\delta + \lambda + \epsilon \delta^{-(2s^* - 1)} \Big).
        \end{align}
        The problem is that we cannot guarantee that $\bx_0 \in A_{t_0,z_0}$. To get around this technical issue, we apply Jensen's Lemma, see \cite{UsersGuide}, to find a sequence of points $\bx_n \in A_{t_0,z_0}$ with $\bx_n \to \bx_0$ and sequence of vectors $p_n \in (\R^d)^N$ with $|p_n| \to 0$ with the property that the function $\phi^n(\bx) : (\T^d)^n \rightarrow \R$ given by $          \phi^n(\bx) = \phi(t_0,z_0,\bx) + p_n \cdot \bx$
        satisfy 
        \begin{align*}
            \text{argmax} \bigg( \bx \mapsto \hat{U}^{\delta, \eps, \lambda}(t_0,z_0,\bx) - \phi^n(\bx) \bigg) = \bx_n. 
        \end{align*}
        Thus we deduce that 
        \begin{align} \label{compn}
        D^2_{x^ix^i} \phi(t_0,z_0,\bx_n) & = D^2_{x^ix^i} \phi^n(\bx_n) \geq  \frac{1}{N} D_{xm} \hat{U}^{\delta, \eps, \lambda}(t_0,z_0,m_{\bx_n}^N,x_n^i) + \frac{1}{N^2} R^{N,i}(\bx_n) \nonumber  \\
        &\geq \frac{1}{N} D_{xm} \hat{U}^{\delta, \eps, \lambda}(t_0,z_0,m_{\bx_n}^N,x_n^i) - C/N^2 I_{d \times d}.
    \end{align}
    Recall that $q \mapsto \hat{U}^{\delta, \eps, \lambda}(t_0,z_0,q)$ is $C^{1,1}$ on $H^{-s^*}$, which, as noted above in \eqref{dxmcont}, means that the map $(x,m) \mapsto D_{xm} \hat{U}^{\delta, \eps, \lambda}(t_0,z_0,m,x)$ is jointly continuous, so that we can pass to the limit in \eqref{compn} and get
    \begin{align*}
        D^2_{x^ix^i} \phi(t_0,z_0,\bx_0) \geq \frac{1}{N} D_{xm} \hat{U}^{\delta, \eps, \lambda}(t_0,z_0,m_{\bx_0}^N,x_0^i) - C/N^2 I_{d \times d}, 
    \end{align*}
    which, by following reasoning in \eqref{Aest}, can be combined with \eqref{phisubsol} to establish 
    \begin{align} \label{phisubsol2}
    &- \partial_t \phi(t_0,z_0,\bx_0) - A_0 \Delta_z \phi(t_0,z_0,\bx_0) - \sum_{i = 1}^N \tr\big( \hat{A}(x_0^i,z_0) D^2_{x^ix^i} \phi(t_0,z_0,\bx_0) \big) \nonumber \\
        &\qquad \qquad + \frac{1}{N} \sum_{i = 1}^N \hat{H}\big(x^i_0,z_0,ND_{x^i} \phi(t_0,z_0,\bx_0)\big) \nonumber \\ 
        &\leq \hat{F}(z_0,m_{\bx_0}^N) + \frac{C}{N\eps} + C \big(1 + \lip(U; W^{-2,\infty}) \big)\Big(\delta + \lambda + \epsilon \delta^{-(2s^* - 1)} \Big),
        \end{align}
        i.e. we have shown that $\hat{U}^{\delta, \eps, \lambda, N}$ is a subsolution of \eqref{hjbhatn} up to an error of order $\frac{C}{N\eps} + C \big(1 + \lip(U; W^{-2,\infty}) \big)\Big(\delta + \lambda + \epsilon \delta^{-(2s^* - 1)}\Big )$.
        Finally, we note that by Lemma \ref{lem.udeltaepslambdareg}, we have
        \begin{align*}
            \| \hat{V}^N(T, \cdot, \cdot) - \hat{U}^{\delta, \eps, \lambda, N}(T,\cdot, \cdot) \|_{\infty} \leq C \Big(\delta + \lambda + \eps \delta^{-2(s^*-1)} \Big),
        \end{align*}
        and so an application of the comparison principle for viscosity (sub/super)solutions of \eqref{hjbhatn} completes the proof. Of course, the improved bound when $A$ is constant comes from the improved bound in Lemma \ref{lem.subsolpointwise} in this case.
\end{proof}

Finally we can complete the proof of convergence under the additional regularity condition \eqref{assump.extrareg}. 

\begin{proof}[Proof of Proposition \ref{prop.mainextrareg}]
We apply Proposition \ref{prop.comparison}, and set $\lambda = \min\{1/2, \,\ C \eps \delta^{- (2s^*  - 1)} \}$ (taking $C$ slightly larger if necessary) to find that there is a constant $C> 0$ such that for each $\delta > 0$, we have 
\begin{align*}
    \hat{U}^{N} \leq  \hat{V}^N + \frac{C}{N\eps} + C \big(1 + \lip(U; W^{-2,\infty})\big) \big(\delta + \eps \delta^{-(2s^*- 1)}\big)
\end{align*}
for any $\eps < \frac{1}{C} \delta^{2s^* }$ (where again we enlarge $C$ if necessary). We now choose 
\begin{align*}
    \delta = \Big(N \big(1 + \lip(U; W^{-2,\infty}) \Big) \Big)^{-1/(2s^*  + 1)}, \quad \eps = \Big(C (1 + \lip(U; W^{-2,\infty})) \delta \Big)^{-1}
\end{align*}
to get 
\begin{align*}
    \hat{U}^{N} &\leq  \hat{V}^N + C \big(1 + \lip(U; W^{-2,\infty}) \big)^{1-1/(2s^* + 1)} N^{-1/(2s^*  + 1)}.
\end{align*}
It is also clear from Proposition \ref{prop.comparison} that the dependence on $\lip(U; W^{-2,\infty})$ can be removed when $A$ is constant.

\end{proof}

\section{The case of zero idiosyncratic noise} \label{sec: zeronoise}

When there is no idiosyncratic noise, we can actually sharpen the results of Theorem \ref{thm.main}.

\subsection{The hard inequality without idiosyncratic noise.}

When $\sigma = 0$, it turns out that the ``hard inequality" is trivial.


\begin{prop} \label{prop.zeronoise}
    Under Assumption \ref{assump.maindegen} and assuming as well that $\sigma=0$ (i.e. there is no idiosyncratic noise) we have, 
    $$ U(t, m^N_{\bx}) \leq V^{N}(t, \bx), \quad \forall \,\, (t,\bx) \in [0,T] \times (\T^d)^N.$$
\end{prop}

\begin{proof}
Let us fix $(t_0, \bx_0) \in [0,T] \times (\T^d)^N$, and assume for the moment that $x_0^1,...,x_0^N$ are distinct. We note that because $\sigma = 0$, it is easy to check that the infimum \eqref{vndef} can be restricted to the set of controls $\bm \alpha = (\alpha_t^1, \dots, \alpha_t^N)_{t_0 \leq t \leq T} \in \sA_{t_0}^N$ adapted to the filtration generated by $(W_t^0 - W_{t_0}^0)_{t_0\leq t \leq T}$. We fix such a control $\bm \alpha$, and let $(X_t^1, \dots, X_t^N)$ be the associated processes, solutions to
$$ dX_t^{i} = \alpha_t^{i} dt + \sigma^0 dW_t^0, \quad X_{t_0}^{i} = x_0^{i}.$$
The goal is to build a control $\alpha$ for the mean field problem such that 
\begin{align*}
    J^{\infty}(t_0,m_{\bx_0}^N,\alpha) \leq J^N(t_0,\bx_0,\bm \alpha).
\end{align*}
To this end, we note that because $\bm \alpha$ is progressively measurable with respect to the filtration generated by $(W^0_t - W_{t_0}^0)$, it admits a version which takes the form 
\begin{align*}
    \alpha_t^i = \hat{\alpha}^i(t,W^0_{[0,t]})
\end{align*}
for some non-anticipative functionals $\hat{\alpha}^i : [t_0,T] \times \cC^{t_0} \to \R^d$. 

Now we recall the canonical space $(\Omega^{\infty, t_0}, \bbF^{\infty, t_0,m_0}, \bP^{\infty,t_0,m_0})$ on which the mean field problem is defined, see Subsection \ref{susec: prob_statement}. We embed the controls $\alpha^i$ into the limiting problem by defining $\bm \alpha^{\infty} = (\alpha_t^{\infty,1},...,\alpha_t^{\infty,N})$ and $\bm{X}^{\infty} = (X_t^{\infty,1},...,X_t^{\infty,N}) $ via
\begin{align*}
    \alpha_t^{\infty,i}(\xi,\omega, \omega^0) = \hat{\alpha}_t^i(\omega^0_{[0,t]}), \quad
    X_t^{\infty,i} = x_0^i + \int_{t_0}^t \alpha_s^{\infty,i} ds + \sigma^0(W_t^0 - W_{t_0}^i).
\end{align*}
We define a control $\alpha \in \sA_{t_0,m_0}$ via 
\begin{align*}
    \alpha^{\infty}_t(\xi,\omega,\omega^0) = \sum_{i = 1}^N 1_{\xi = x_0^i} \alpha_t^{\infty,i}.
\end{align*}
Finally, we let $X$ be the state process determined by $\alpha^{\infty}$, which is determined by the dynamics \eqref{dynamicsinf}. It is easy to check that 
$\cL^{0,m_0}(X_t) = \frac{1}{N} \sum_{i = 1}^N \delta_{X_t^{\infty,i}}$, and thus 
\begin{align*}
    J^{\infty}(t_0,m_{\bx_0}^N,\alpha) = \E^{\bP^{\infty,t_0,m_0}} \bigg[ \int_{t_0}^T \frac{1}{N} \sum_{i=1}^N L(X_t^{\infty, i}, \alpha_t^{\infty, i})dt + \int_{t_0}^T F( m^N_{\bX_t^{\infty}}) dt + G ( m_{\bX_T^{\infty}}^N)  \bigg] \\
    = \E \bigg[ \int_{t_0}^T \frac{1}{N} \sum_{i=1}^N L(X_t^{i}, \alpha_t^{\infty, i})dt + \int_{t_0}^T F( m_{\bX_t}^N) dt + G ( m_{\bX_T}^N)  \bigg] = J^N(t_0,\bx_0,\bm \alpha).
\end{align*}
Taking infimum over all $\bm \alpha$, we see that $U(t_0,m_{\bx_0}^N) \leq V^N(t_0,\bx_0)$, at least provided that $x_0^1,...,x_0^N$ are distinct. We complete the proof with a density argument.

\end{proof}

\subsection{Purely Deterministic Case}

In this section we show how to sharpen the rates of Theorem \ref{thm.main} in the deterministic setting. To illustrate the specific features of this case we choose the simplest setting, with
\begin{align} \label{assumption.det}
    L(x,a,m) = \frac{1}{2} |a|^2, \quad F = 0, \quad \sigma = \sigma^0 = 0.
\end{align}
On the other hand, since we will not rely on the sup-convolution operation, there is no additional difficulty in taking the state space to be $\R^d$ instead of $\T^d$. As a consequence, the control problems of Section \ref{susec: prob_statement} are describe by value functions $V^N : [0,T] \times (\R^d)^N \to \R$ defined by
$$V^N(t_0, \bx_0) := \inf_{(\alpha_t^{i})_{1 \leq i \leq N}} \frac{1}{N} \sum_{i=1}^N \int_{t_0}^T \frac{1}{2} |\alpha_t^{i}|^2 dt + G ( m_{\bX_T}^N ),$$
where $m^N_{\bx} := \frac{1}{N} \sum_{i=1}^N \delta_{x^{i}}$ for any $ \bx = (x^1, \dots, x^N) \in (\R^d)^N$ and the dynamics of the particles $\bX : =(X^1, \dots, X^N)$ are given by
$$ d X_t^{i} = \alpha_t^{i} dt, \quad t_0 \leq t \leq T, \quad X_{t_0}^{i} = x_0^{i}, \quad 1 \leq i \leq N,$$
with $\alpha^{i} \in L^2([t_0,T]; \R^d)$ for all $1 \leq i \leq N.$
While the mean-field control problem can be written
\begin{equation}
    U(t_0, m_0) := \inf_{(m_t, \alpha_t)}  \int_{t_0}^T \int_{\R^d} \frac{1}{2} |\alpha_t(x)|^2 dm_t(x)dt + G(m_T),
\label{MFvaluefunction}
\end{equation}
where the infimum is taken over the couples $(m;\alpha)$ with $m \in AC ( [t_0,T], \mathcal{P}_2(\R^d))$ (absolutely continuous curves with values in $\mathcal{P}_2(\R^d)$), $\alpha \in L^2(dt \otimes m_t; \R^d)$ satisfying the continuity equation
\begin{equation}
\partial_t m_t + \div(\alpha_t m_t) = 0, \quad \mbox{ in } (t_0,T) \times \R^d, \quad m(t_0) = m_0.
\label{eq:continuityequation}
\end{equation}
We still call this set of controls $\sA_{t_0,m_0}$ as before. In fact, we have the following: 
\begin{lem}
    Under Assumption \ref{assump.maindegen} and \eqref{assumption.det}, the value in \eqref{MFvaluefunction}, over the set of controls defined above, is equal to the value in \eqref{udefcontrol} over the set of controls defined there in subsection \ref{susec: prob_statement}.
\end{lem}

\begin{proof}
    Given a couple $(m;\alpha)$ solution to the Fokker Planck equation, by the superposition principle 
    \cite[Thm. 1.3]{Lacker_superposition} there exist a weak solution $(\Omega, \mathbb{F}, \mathbb{P}, W,X)$ to $dX_t =\alpha(t,X_t)dt$ (viewed as an It\^o SDE)  associated to the vector field $\alpha$ and the initial condition $m_0$, where $W$ is an $\mathbb{F}$-Brownian motion. Thus it gives rise to a weak solution $(\Omega, \mathbb{F}, \mathbb{P}, W, X, \alpha)$ of the control problem, as defined e.g. in \cite[Def 8.1]{Lacker_superposition} 
    (with $\mathbb{G} = \{ \emptyset, \Omega\}$ therein) 
    \footnote{We prefer not to present the weak formulation of the control problem because it is invoked here only, and in the rest of the paper we employ the strong formulation.}
    ,
    via $\alpha_t = \alpha_t(X_t)$. 
    Conversely, given a weak solution of the control problem there exists a feedback control $\alpha(t,x)$, measurable in $x$, which mimics the marginal distribution and has a lower cost, thanks to \cite[Thm. 8.3]{Lacker_superposition}.
    Finally, notice that the infimum over weak controls equals the infimum over strong controls (as defined in \S \ref{susec: prob_statement}) by \cite[Thm. 3.1]{DjetePossamaiTan}.
\end{proof}

Notice that, importantly, controls in the strong open-loop formulation are in the form $\hat{\alpha}(t,\xi, \omega)$, i.e. functions of the initial condition but also of a Brownian motion which is not present in the dynamics. This is necessary to randomize and permits to split the mass. Considering controls as functions $\hat{\alpha}(t,\xi)$ of the initial condition only does not in general give the same infimum, as shown in Corollary \ref{cor:Utilde} below.


\subsubsection{Hopf-Lax formula and first convergence result.}

We can express the value functions in terms of Hopf-Lax transform of the terminal cost. 
\begin{lem}
    Assume that $G : \mathcal{P}_2(\R^d) \rightarrow \R$ is bounded from below. For every $(t, m) \in [0,T] \times \mathcal{P}_2(\R^d),$
\begin{equation}
U(t,m) = \inf_{\nu \in \mathcal{P}_2(\R^d)} \bigl \{ G(\nu) + \frac{1}{2(T-t)}\bd_2^2(m,\nu) \bigr \},
\label{eq:HLforU}
\end{equation}
where $\bd_2$ is the $2$-Wasserstein distance. Similarly, for every $(t, \bx) \in [0,T] \times (\R^d)^N$,
$$V^N(t, \bx) = \inf_{ \by \in (\R^d)^N} \bigl \{ G^N(\by) + \frac{1}{2(T-t)N} \sum_{i=1}^N |x^{i} - y^{i}|^2 \bigr \},$$
where $G^N(\bx) := G(m^N_{\bx}).$
\end{lem}

\begin{proof}
    These results are well-known (see for instance \cite{Lionsvideo} for the first statement) and we omit the proof, noting only that the first statement is a consequence of the Benamou-Brenier formulation of optimal transport with quadratic cost.
\end{proof}

For every $\by \in (\R^d)^N$ we can permute the coordinates without changing the value of $G^N( \by)$. Therefore we immediately see that 
\begin{equation}
V^N(t, \bx) = \inf_{\by \in (\R^d)^N} \bigl \{ G^N( \by) + \frac{1}{2(T-t)}\bd_2^2(m_{\bx}^N, m_{\by}^N) \bigr \}, 
\label{eq:21Aout15:26}
\end{equation}
and we deduce the following result.

\begin{prop}
Assume that $G : \mathcal{P}_2(\R^d) \rightarrow \R$ is bounded from below. For every $(t,\bx) \in [0,T] \times (\R^d)^N$ it holds
\begin{equation}
U(t, m_{\bx}^N) \leq V^N(t, \bx).
\label{eq:oneinequalitywithoutidnoise}
\end{equation}
Assume as well that $G$ satisfies, for some $C_G>0$,
\begin{equation} 
\bigl| G(m^2) - G(m^1) \bigr| \leq C_G \bigl( 1+ M_2(m^2) + M_2(m^1) \bigr) \bd_2(m^1,m^2), \quad \forall m^1,m^2 \in \mathcal{P}_2(\R^d),
\label{eq:GLip21Aout}
\end{equation}
where we use the notation $M_2 (\nu) := \Bigl( \int_{\R^d} |x|^2 d\nu(x) \Bigr)^{1/2}$. Then, for every $(t,\bx^N) \in [0,T] \times (\R^d)^N$ such that $m^N_{\bx^N} \rightarrow m$ in $\mathcal{P}_2(\R^d)$ it holds
$$ \lim_{ N \rightarrow +\infty} V^N(t, \bx^{N}) = U(t, m).$$
\label{prop:firstconvergence}
\end{prop}
\begin{proof}
    The first part of the statement follows directly from the representation formula \eqref{eq:HLforU} and \eqref{eq:21Aout15:26}. For the second part we fix $(t,m) \in [0,T] \times \mathcal{P}_2(\R^d)$. For some $\epsilon >0$ we let $\nu \in \mathcal{P}_2(\R^d)$ be such that
    $$ U(t,m) \geq G(\nu) + \frac{1}{2(T-t)} \bd_2^2( \nu,m) - \epsilon.$$
Now we take a sequence $(\by^N) \in (\R^d)^N$ such that
$$ \lim_{N \rightarrow +\infty} \bd_2( m^N_{\by^N}, \nu) =0.$$
Given the representation formula \eqref{eq:21Aout15:26} for $V^N$ we have
$$ V^N(t, \bx^N) \leq G( m _{\by^N}^N ) + \frac{1}{2(T-t)} \bd_2^2( m_{\by^N}^N, m_{\bx^N}^N). $$
Therefore, by continuity of $G$ we have
$$ \limsup_{N \rightarrow +\infty} V^N(t, \bx^N) \leq G(\nu) + \frac{1}{2(T-t)} \bd_2^2(\nu, m) \leq U(t,m) + \epsilon. $$
Since this is valid for any $\epsilon >0$ we have
\begin{equation} 
\limsup_{N \rightarrow +\infty} V^N(t, \bx^N) \leq U(t,m).
\label{eq:limsup21Aout}
\end{equation}
Notice that this inequality only requires the continuity of $G$ and not the quantitative regularity condition \eqref{eq:GLip21Aout}. On the other hand, condition \eqref{eq:GLip21Aout} ensures that $U(t,\cdot)$ is continuous over $\mathcal{P}_2(\R^d)$ for every $t \in [0,T]$, see \cite{Calder} Lemma 8.5 for the relevant property of the inf-convolution. Since, for all $N \geq 1$ we have
$$ V^N(t,\bx^N) \geq U(t, m_{\bx^N}^N)$$
we can deduce from the continuity of $U(t, \cdot)$ that
$$ \liminf_{N \rightarrow +\infty} V^N(t, \bx^N) \geq U(t,m)$$
which concludes the proof of the proposition.
\end{proof}

\subsubsection{On the choice of admissible controls}
\label{sec:onthechoiceofcontrols}
In this setting with degenerate idiosyncratic noise, the set of admissible controls plays an important role. To wit, define $U^f$ ($f$ standing for feedback), the feedback value function by
    $$U^f(t_0,m_0) = \inf_{ \alpha \in L^1_t(Lip_x)} \frac{1}{2} \int_{t_0}^T \int_{\R^d} |\alpha_t(x)|^2 dm_t(x)dt + G(m_T), $$
where the infimum is taken over the feedback controls in $L^1([0,T] \times \R^d, \R^d)$ which are Lipschitz in space and $(m_t)_{t \geq 0}$ solves
$$\partial_t m_t + \div(\alpha_t m_t) =0, \mbox{ in }(t_0,T) \times \R^d, \quad m_{t_0}= m_0.$$
Notice that this is the type of controls considered in \cite{Fornasier2014}. Define accordingly $V^{f,N}$ to be the infimum of the $N$-particle problem, when the controls take the form
$$ \alpha_t^{i} = \alpha^{i}(t,X_t^1, \dots, X_t^N)$$
for some Lipschitz maps $\alpha^i(t,\cdot).$
Now take $ \bx_0 = (x_0^1, \dots,x_0^N) \in (\R^d)^N$ and $\alpha_t(\cdot)$ an admissible control for $U^f(t_0, m_{\bx_0}^n)$ with corresponding trajectory $(m_t)_{t \geq t_0}$. Given the regularity of the control, we easily check that
$$m_t = \frac{1}{N} \sum_{i=1}^N \delta_{X_t^{i}},$$
where, for all $1 \leq i \leq N$, $X_t^{i}$ is the solution to the ODE
$$dX_t^{i} = \alpha_t(X_t^{i}) dt, \quad X_{t_0}^{i} = x_0^{i}$$
and the corresponding cost is
$$ \frac{1}{2} \int_{t_0}^T \frac{1}{N}\sum_{i=1}^N | \alpha_t(X_t^{i})|^2dt + G \bigl( \frac{1}{N} \sum_{i=1}^N \delta_{X_T^{i}} \bigr).$$
We can give the same control $\alpha$ to each particle in the $N$-particle control problem and, therefore, we have the inequality
\be 
\label{VfN<Vf}
V^{f,N}(t_0, \bx_0)  \leq  U^f(t_0, m^N_{\bx_0}), 
\ee
giving the opposite inequality of Proposition \ref{prop:firstconvergence}.

Notice that in general $U^f(t_0,m_0) \neq U (t_0,m_0)$, as shown in Corollary \ref{cor.feedback}.

\subsubsection{Convergence rate}

Although one inequality is clear, we would like to find a bound for the other side, i.e.~ we look for a bound of the form
$$V^N(t, \bx) \leq U(t, m_{\bx}^n) + \epsilon_n$$
for some $\epsilon_n \rightarrow 0.$ For any $N \in \mathbb{N}$ and $\nu \in \mathcal{P}_2(\R^d)$,  we define the following error
$$\epsilon_{N}(\nu) :=  \inf_{\by \in (\R^d)^N} \bd_2(m_{\by}^N,\nu).$$

\begin{lem} \label{prop.detbound}
    Assume that $G$ is $\bd_1$-Lipschitz continuous with constant $L_G$. Then, for all $(t, \bx) \in [0,T) \times (\R^d)^N$ there is $\bar{\nu} \in \mathcal{P}_2(\R^d)$ such that 
\begin{equation} 
G( \bar{\nu}) + \frac{1}{2(T-t)} \bd_2^2(m_{\bx}^N, \bar{\nu}) = U(t,m_{\bx}^N)= \min_{ \nu \in \mathcal{P}_2(\R^d)} G(\nu) + \frac{1}{2(T-t)} \bd_2^2 (m_{\bx}^N,\nu),  
\label{eq:defnbarnu13/01}
\end{equation}
and every such $\bar{\nu}$ satisfies
    $$V^N(t, \bx) \leq U(t, m_{\bx}^N) + 4 L_{G} \epsilon_N (\bar{\nu}). $$
    
\label{prop:Quantization}
\end{lem}


\begin{proof}[Proof of Lemma \ref{prop:Quantization}]
    Let us fix $(t, \bx) \in [t_0,T) \times (\R^d)^N$ for some $N \geq 1$. We first justify that there exists $\bar{\nu} \in \mathcal{P}_2(\R^d)$ such that \eqref{eq:defnbarnu13/01} holds. Since $\nu \mapsto G(\nu)$ is globally $\bd_1$-Lipschitz continuous it has at most linear growth in $\int_{\R^d} |x|d\nu(x)$ and therefore $\nu \mapsto G(\nu) + (2(T-t))^{-1} d_2^2( m^N_{\bx}, \nu)$ is bounded from below. Let $(\nu^n)_{n \in \mathbb{N}}$ be a minimizing sequence. For the same reason, we easily have that  $(\int_{\R^d} |x|^2 d\nu^n(x))_{n \in \mathbb{N}}$ is bounded.
 As a consequence, the sequence admits a converging sub-sequence (still denoted $(\nu^n)_{ n \in \mathbb{N}}$) converging in $\mathcal{P}_{1}(\R^d)$ to some $\bar{\nu} \in \mathcal{P}_2(\R^d)$. Since $G$ is $d_1$-Lipschitz continuous and $(\mu^1, \mu^2) \mapsto \bd_2(\mu^1,\mu^2)$ is weakly lower-semi-continuous (when $\mathcal{P}(\R^d)$ is endowed with the usual topology of weak convergence of probability measures, i.e., convergence holds against bounded continuous functions) we deduce that $\bar{\nu}$ is indeed a minimum of $ \mathcal{P}_2(\R^d) \ni \nu \mapsto G(\nu) + \frac{1}{2(T-t)} \bd_2^2(m_{\bx}^N,\nu)$.

Let take such an element $\bar{\nu}$ of $\mathcal{P}_2(\R^d)$ satisfying \eqref{eq:defnbarnu13/01}.  Using the Hopf-Lax formula for $V^N$ and $U$ we have, for every $\by \in (\R^d)^N,$
\begin{align*}
V^N(t, \bx) - U(t, m_{\bx}^N) & \leq G( m_{\by}^N) - G(\bar{\nu}) + \frac{1}{2(T-t)} \bd_2^2( m^N_{\bx}, m^N_{\by}) - \frac{1}{2(T-t)}\bd_2^2( m_{\bx}^n, \bar{\nu}) \\
&\leq L_G \bd_2( m_{\by}^N, \bar{\nu}) + \frac{1}{2(T-t)} \left[   \bd_2(m_{\bx}^n, m_{\by}^N) + \bd_2( m_{\bx}^N, \bar{\nu}) \right] \bd_2( m_{\by}^N, \bar{\nu}) \\
&\leq L_G \bd_2( m_{\by}^N, \bar{\nu}) + \frac{1}{2(T-t)} \left[   2\bd_2(m_{\bx}^N, \bar{\nu}) + \bd_2( m_{\by}^N, \bar{\nu}) \right] \bd_2( m_{\by}^N, \bar{\nu}),
\end{align*}
where we used the $\bd_2$-Lipschitz regularity of $G$ (which is clearly inherited by its $\bd_1$-Lipschitz regularity) as well as the triangular inequality.
By definition of $\bar{\nu}$ we have
$$G( m_{\bx}^n) \geq G(\bar{\nu}) + \frac{1}{2(T-t)}\bd_2^2( m_{\bx}^N, \bar{\nu}),$$
which leads, in view of the Lipschitz regularity of $G$, to
\begin{equation}
\frac{1}{2(T-t)} \bd_2^2( m_{\bx}^N, \bar{\nu}) \leq G(m_{\bx}^N) - G(\bar{\nu}) \leq L_G \bd_2(m_{\bx}^N, \bar{\nu})
\label{eq:21Aout9:46}
\end{equation}
and therefore, it holds
\begin{equation}
\bd_2( m_{x}^N, \bar{\nu}) \leq 2L_G(T-t).
\label{eq:21Aout9:47}
\end{equation}
In particular, this implies that $M_2(\bar{\nu}) \leq M_2( m_{\bx}^N) + 2L_G(T-t).$ 
As a consequence, we have
$$V^N(t , \bx) - U(t, m_{\bx}^N) \leq \Bigl( 3 L_G + \frac{\bd_2(m_{\by}^N, \bar{\nu})}{2(T-t)} \Bigr) \bd_2(m_{\by}^N,\bar{\nu}). $$
Taking the infimum over $\by \in (\R^d)^N$ we get, by definition of $\epsilon_N$, 
\begin{equation}
V^N(t , \bx) - U(t, m_{\bx}^N) \leq \Bigl( 3 L_G + \frac{\epsilon_N ( \bar{\nu} )}{2(T-t)} \Bigr)\epsilon_N ( \bar{\nu}). 
\label{eq:firstestimate21Aout}
\end{equation}
On the other hand, using first the representation formula for $V^N$ and then the definition of $\bar{\nu}$ we have
\begin{align*}
V^N(t, m_{\bx}^N) - U(t, m_{\bx}^N) &\leq G( m_{\bx}^N) - U(t,m_{\bx}^N)  \\
&\leq G( m_{\bx}^N)  - G(\bar{\nu}). 
\end{align*}
Recalling \eqref{eq:21Aout9:46} and \eqref{eq:21Aout9:47} this leads to
\begin{equation} 
V^N(t,m_{\bx}^N) \leq U(t, m_{\bx}^N) +  2L^2_G(T-t).
\label{eq:secondestimate21Aout}
\end{equation}
Now we distinguish between the values of $t \in [0,T]$. If $2L_G(T-t) \geq \epsilon_N (\bar{\nu}) $ we use estimate \eqref{eq:firstestimate21Aout} and otherwise we use estimate \eqref{eq:secondestimate21Aout}. In both cases we find
$$ V^N(t,m_{\bx}^N) \leq U(t, m_{\bx}^N) + 4 L_G \epsilon_N (\bar{\nu}). $$
This concludes the proof of the lemma.
\end{proof}

\begin{lem}
    Assume that $G$ is $\bd_1$-Lipschitz with Lipschitz constant $L_G$. Take $(t,\bx) \in [0,T] \times (\R^d)^N$ and $\bar{\nu} \in \mathcal{P}_2(\R^d)$ satisfying \eqref{eq:defnbarnu13/01}. Then for any $ p \geq 5$, 
    $$ M_p(\bar{\nu}) \leq L_{G}(T-t) + M_p(m_{\bx}^N)$$
where $M_p(\nu) :=  \bigl(\int_{\R^d} |x|^p d\nu(x) \bigr)^{1/p}$ . 
\label{lem:momentestimate13/01}
\end{lem}

\begin{proof}
    \textit{Step 1.} 
We fix $(t,\bx) \in [0,T) \times (\R^d)^N$ and $\bar{\nu}$ and element of  $\mathcal{P}_2(\R^d)$ satisfying \eqref{eq:defnbarnu13/01}. Let $(\Omega,\mathcal{F}, \mathbb{P})$ be an atomless probability space and $(\bar{Y},X_{\bx}^N)$ a couple of random variables in $L^2(\Omega; \R^d)$ such that $\mathcal{L} \bigl( (\bar{Y}, X_{\bx}^N) \bigr)$ is an optimal transport plan between $\bar{\nu}$ and $m_{\bx}^N$. Then, for any $Y \in L^2(\Omega; \R^d)$ we have, denoting $\bar{G}$ the lift of $G$ to $L^2(\Omega; \R^d)$,
\begin{align*}
    \bar{G}(Y) + \frac{1}{2(T-t)} \E \bigl[ |X_x^N - Y|^2 \bigr] &\geq G( \mathcal{L}(Y)) + \frac{1}{2(T-t)} \bd_2^2(m_{\bx}^N,\mathcal{L}(Y)) \\
    &\geq G( \bar{\nu}) + \frac{1}{2(T-t)} \bd_2^2(m_{\bx}^N,\bar{\nu}) \\
    &= \bar{G}(\bar{Y}) + \frac{1}{2(T-t)} \E \bigl[ |X_{\bx}^N - \bar{Y}|^2 \bigr]
\end{align*}
where we first used the definition of $\bd_2$ as infimum over couplings, next the definition of $\bar{\nu}$ as a minimum in \eqref{eq:defnbarnu13/01} and finally the fact that $\mathcal{L}((X_{\bx}^N,Y))$ is an optimal coupling between $m_{\bx}^N$ and $\bar{\nu}$. This shows that $\bar{Y}$ is a minimum, over $L^2(\Omega;\R^d)$ of 
$$ Y \mapsto \bar{G}(Y) + \frac{1}{2(T-t)} \E \bigl[ |X_{\bx}^N - Y|^2 \bigr].$$

\textit{Step 2.}  The map $\bar{G}$ is obviously Lipschitz continuous with respect to the $L^1$-metric over $(\Omega,\mathcal{F}, \mathbb{P})$ with constant $L_G$. We claim that $\bar{Y} - X_{\bx}^N$ belongs to $L^{\infty}$ and 
\begin{equation} 
\norm{ \bar{Y} - X_{\bx}^N}_{L^{\infty}} \leq L_G(T-t).
\label{eq:Linfty13/01}
\end{equation}
Indeed, for any $Y \in L^2$ we have, by minimaliy of $\bar{Y}$
\begin{align*}
    \frac{1}{2} \E \bigl[ |X_{\bx}^N - \bar{Y}|^2 \bigr] - \frac{1}{2} \E \bigl[ | X_{\bx}^N - Y |^2 \bigr] & \leq (T-t) \bigl( \bar{G}(Y) - \bar{G}(\bar{Y})) \\
    &\leq L_G(T-t) \E \bigl[ |Y - \bar{Y}|].
\end{align*}
And therefore, for all $Y \in L^2$,
$$ \E \bigl[ (\bar{Y} - X_{\bx}^N) \cdot(\bar{Y} - Y) \bigr] \leq L_G(T-t) \E \bigl[ | Y - \bar{Y}| \bigr] + \frac{1}{2} \E \bigl[ |Y - \bar{Y}|^2 \bigr]. $$
For fixed $Y$ and any $\epsilon >0$, considering the above estimate for $\bar{Y} + \epsilon (\bar{Y} - Y)$ and dividing by $\epsilon $ we get
$$ \E \bigl[ (\bar{Y} - X_{\bx}^N) \cdot(\bar{Y} - Y) \bigr] \leq L_G(T-t) \E \bigl[ | Y - \bar{Y}| \bigr] + \frac{\epsilon}{2} \E \bigl[ |Y - \bar{Y}|^2 \bigr]. $$
Letting $\epsilon \rightarrow 0^+$ we deduce that
$$ \E \bigl[ (\bar{Y} - X_{\bx}^N) \cdot(\bar{Y} - Y) \bigr] \leq L_G(T-t) \E \bigl[ | Y - \bar{Y}| \bigr], \quad \mbox{ for all } Y \in L^2. $$
This shows that $(\bar{Y} - X_{\bx}^N)$ belongs to $L^{\infty}$ and \eqref{eq:Linfty13/01} holds. As a consequence, we get
\begin{align*} 
M_p (\bar{\nu}) &= \E \bigl[ | \bar{Y}|^p \bigr]^{1/p} \leq \E \bigl[ |\bar{Y} - X_{\bx}^N|^p \bigr]^{1/p} + \E \bigl[ | X_{\bx}^N|^p \bigr]^{1/p} \\
&\leq L_G(T-t) + M_p(m_{\bx}^N), 
\end{align*}
which concludes the proof of the Lemma.
\end{proof}

As a consequence, by the results of \cite{FournierGuillin}, we have

\begin{prop}
Assume that $G$ is $L_G$-Lipschitz with respect to $\bd_1$. Then, for all $p \geq 5$, there is $C_{p,d}>0$ depending only on $p$ and $d$ such that, for all $(t,\bx) \in [0,T] \times (\R^d)^N$,
\begin{equation} 
V^N(t,m_{\bx}^N) \leq U(t,m_{\bx}^N) +C_{p,d} L_{G} \Bigl( L_{G}(T-t) + M_p(m_{\bx}^N) \Bigr)r_{d,N}  
\label{eq:mainestimate13/01}
\end{equation}
with 
\begin{equation} 
r_{d,N} := \left \{
\begin{array}{ll}
N^{-1/4} & \mbox{ if } d<4 \\
N^{-1/4} \log(1+N)^{1/2} &\mbox{ if } d =4 \\
N^{-1/d} & \mbox{ if } d >4. 
\end{array}
\right.
\label{eq:FG_Rate_p=2}
\end{equation}
\end{prop}

\begin{proof}

A consequence of  Theorem 1 in \cite{FournierGuillin}, is that for every $p \geq 5$ there is a constant $C_{p,d}$ depending only on $p$ and $d$ such that, for every $N \geq 1$ we can find $\by \in (\R^d)^N$ with
$$ \bd_2^2 (m_{\by}^N, \bar{\nu}) \leq C^2_{p,d} M_p^{2} (\bar{\nu}) r_{d,N}^2,  $$
where $r_{d,N}$ is defined in \eqref{eq:FG_Rate_p=2}. And so, we have
$$ \epsilon_N(\bar{\nu}) \leq C_{p,d} M_p (\bar{\nu}) r_{d,N}. $$
Using the conclusions of Lemma \ref{lem:momentestimate13/01}, we get 
$$ \epsilon_N(\bar{\nu}) \leq C_{p,d} \Bigl( L_{G}(T-t) + M_p(m_{\bx}^N) \Bigr) r_{d,N},  $$
which concludes the proof of the Proposition. 
\end{proof}

Now we exploit results on the optimal quantization of the Lebesgue measure over the unit cube $[0,1]^d$ to prove that the bound is optimal.

\begin{prop}[Optimality of the bound]
\label{prop:opt:bound:N}
We take  $G(m) := \bd_1(m, \mathrm{Leb} |_{[0,1]^d})$. To simplify we take $t_0= T- 1/2$. There exists $c_d >0$ and $N_0 \in \mathbb{N}$ such that, for all  $N \geq N_0$ there exists $\bx_N \in ([0,1]^d)^N$ such that
$$V^N(t_0, \bx^N) - U(t_0, m^N_{\bx^N}) \geq c_d N^{-1/d}.$$
\end{prop}

\begin{proof}
The value functions read 
\begin{equation} 
U(t_0, m) = \inf_{ \nu \in \mathcal{P}_2(\R^d)} \bigl \{ \bd_1(\nu, \leb|_{[0,1]^d}) + \bd_2^2(m, \nu) \bigr \},
\label{eq:defnU21/01}
\end{equation}
\begin{equation}
V^N(t_0, \bx) = \inf_{ \by \in (\R^d)^N} \bigl \{ \bd_1(m_{\by}^N, \leb|_{[0,1]^d}) + \bd_2^2(m_{\by}^N, m_{\bx}^N) \bigr \}.
\label{eq:defnVN21/01}
\end{equation}
On the one hand, by \eqref{eq:defnU21/01} we have, for all $m \in \mathcal{P}_2(\R^d),$
$$U(t_0,m) \leq \bd_2^2(m, \leb|_{[0,1]^d}).$$
By partitioning the unit cube into $N$ cubes of length roughly equal to $N^{-1/d}$ and putting a point at the center of each of this cube we can easily show that, for some $c_1 >0$ depending only on the dimension $d$, 
$$\forall N \in \mathbb{N}, \exists \bx^N \in ([0,1]^d)^N, \quad \bd_2(m^N_{\bx^N}, \leb|_{[0,1]^d})^2 \leq c_1 N^{-2d}.$$
And therefore, for every $N \in \mathbb{N}^*$, for the same constant $c_1$ and the same point $\bx^N \in ([0,1]^d)^N$ we have 
\begin{equation}
    U(t_0, m^N_{\bx^N}) \leq c_1 N^{-2/d}. 
\label{eq:upperboundU21/01}
\end{equation}
On the other hand, by \eqref{eq:defnVN21/01}, we have
\begin{equation} 
V^N(t_0, \bx) \geq \inf_{ \by \in (\R^d)^N} \bd_1(m_{\by}^N, \leb|_{[0,1]^d}), \quad \forall \bx \in (\R^d)^N. 
\label{eq:comparisonVN21/01}
\end{equation} 
Now, let us denote $v_N := \inf_{ \mu \in \mathcal{P}_{1,N}(\R^d)} d_1(\leb|_{[0,1]^d}, \mu) $ where the infimum is taken over the probability measures $\mu$ over $\R^d$ with a support containing at most $N$ points. It is obvious that 
\begin{equation} 
 \inf_{ \by \in (\R^d)^N} \bd_1 \bigl(  m^N_{\by},\leb|_{[0,1]^d}) \geq v_N. 
\label{eq:21/01}
\end{equation}
Now, combining Lemma 3.5 in \cite{Graf2000} (that gives an equivalent definition of $v_N$) and Theorem 6.2 in the same reference we get  
$$ \inf_{ N \geq 1} N^{1/d} v_N >0.$$
Together with \eqref{eq:21/01} and \eqref{eq:comparisonVN21/01} this implies that there is $c_2 >0$ such that, for every $N \geq 1$ and every $\bx \in (\R^d)^N$ 
\begin{equation}
V^N(t_0, \bx) \geq c_2 N^{-1/d}. 
\label{eq:lowerboundVN21/01}
\end{equation}
To conlcude the proof of the Proposition it is enough to combined  the upper bound \eqref{eq:upperboundU21/01} and the lower bound \eqref{eq:lowerboundVN21/01}.
\end{proof}

We note that the proof that $U(t,m_{\bx}^N) \leq V^N(t,\bx)$ holds equally well in the present setting, and so as a consequence of Proposition \ref{prop:opt:bound:N} we get the following.

\begin{cor} \label{cor.feedback}
    In the setting of the above Proposition, there exists $(t_0,m_0) \in [0,T] \times \cP_2(\R^d)$ such that $U(t_0,m_0) < U^f(t_0,m_0)$.
\end{cor}

\begin{proof}
    Applying \eqref{eq:oneinequalitywithoutidnoise} and \eqref{VfN<Vf} and exploiting the inclusion of the sets of controls, we have for any $t_0$, $N$ and $\bm{x}$
    \[
    U(t_0, m^N_{\bm{x}}) \leq V^N(t_0, m^N_{\bm{x}}) \leq V^{N,f}(t_0, m^N_{\bm{x}}) \leq U^f(t_0, m^N_{\bm{x}}). 
    \]
    Thus if, by contradiction, $U^f(t_0, m^N_{\bm{x}}) \leq U(t_0, m^N_{\bm{x}})$, then we would get $U(t_0, m^N_{\bm{x}}) = V^N(t_0, m^N_{\bm{x}})$ which is in contrast  with Proposition \ref{prop:opt:bound:N}.
\end{proof}

Similarly, denoting by $\tilde{U}(t_0,m_0)$ the infimum over controls $\hat{\alpha}(t,\xi)$ that are measurable functions of time and the initial condition only, we have the following:
\begin{cor} 
\label{cor:Utilde}
    In the setting of the above Proposition, there exists $(t_0,m_0) \in [0,T] \times \cP_2(\R^d)$ such that $U(t_0,m_0) < \tilde{U}(t_0,m_0)$.
\end{cor}
\begin{proof}
    We show that 
    $V^N(t_0, m^N_{\bm{x}}) \leq \tilde{U}(t_0, m^N_{\bm{x}})$ for any $t_0\in [0,T]$ and $\bm{x}\in (\R^d)^N$, so that the conclusion follows as in the previous corollary. 
    Consider a control for the mean field problem $\alpha_t=\hat{\alpha}(t, \xi)$, with corresponding process $X$, for an initial condition 
$\xi= \sum_{i=1}^N x_i 1_{\xi=x_i}$ with law $m^N_{\bm{x}}$. For the $N$ players, starting at $\bm{X}_{t_0}=\bm{x}$, define the controls $\alpha^i_t = \hat{\alpha}(t, x_i)$. Thus we have $\alpha_t = \sum_i \alpha^i_t 1_{\xi=x_i}$ and then $X_t = \sum_i X^i_t 1_{\xi=x^i}$, so that $\mathcal{L}(X_t) = m^N_{\bm{X}_t}$ and we conclude $V^N(t_0,\bm{x})\leq \tilde{U}(t_0, m^N_{\bm{x}})$.
\end{proof}

We remark that in the case of finite dimensional control, without mean-field interaction in the cost, all the formulations are equivalent, basically because the cost becomes a linear functional of the law of the process: $U=U^f=\tilde{U}$.

\subsubsection{Convex Case.}

Under some conditions on $G$ of convexity type we can actually prove equality of the value functions.

\begin{prop}
    Assume that $G : \mathcal{P}_2(\R^d) \rightarrow \R$ is bounded from below, continuous and satisfies, for any probability space $(\Omega, \mathcal{F}, \mathbb{P})$,
\begin{equation} 
G \bigl( \mathcal{L}\bigl( \E \bigl[X |Y \bigr] \bigr) \bigr) \leq G  \bigl( \mathcal{L}(X)\bigr) \quad \forall X,Y \in L^2(\Omega ; \R^d). 
\label{eq:weakconvexity13/01}
\end{equation}
Then, for all $t \in [0,T]$ and all $\bx \in (\R^d)^N$, we have $V^N( t, \bx^N ) = U(t, m_{\bx}^N )$.
\label{prop:equalityifconvexity}
\end{prop}
Notice that the condition on $G$ is satisfied, for instance, if the lift of $G$ to any $L^2$ space of random variables over an atomless probability space is convex, as a consequence of Jensen's inequality, see Lemma \ref{lem:Jensen} below.
\begin{proof}
Thanks to Proposition \ref{prop:firstconvergence}, it suffices to prove the inequality $V^N(t_0, \bx^N) \leq U(t_0, m^N_{\bx})$. We are going to show that, for all $n \in \mathbb{N}^*$ and all $\bx^N \in (\R^d)^N$ we have
    $$ V^N (t, \bx^N) \leq V^{nN} (t, \bx^N,\dots, \bx^N), $$
where $\bx^N$ is repeated $n$ times in the argument of $V^{nN}$. Let $(y_i^{j})_{1 \leq i \leq N, 1 \leq j \leq n} \in (\R^d)^{nN}$. 
On some probability space, we build two independent random variables $I$ and $J$ uniformly distributed respectively on $\{ 1, \dots, N \}$ and $\{ 1, \dots, n \}$ and we let $ Y := y^{I}_{J}$. We check that
$$ \mathcal{L} \Bigl( \E \Bigl[ Y | I \Bigr] \Bigr) = \frac{1}{N} \sum_{i=1}^N \delta_{ \frac{1}{n} \sum_{j=1}^n y_i^{j} }  $$
and so, by assumption on $G$,
$$ G \Bigl( \frac{1}{N} \sum_{i=1}^N \delta_{ \frac{1}{n} \sum_{j=1}^n y_i^{j} }  \Bigr) \leq G \Bigl( \frac{1}{nN} \sum_{i=1}^N \sum_{j=1}^n \delta_{y_i^j} \Bigr).$$
On the other hand, by convexity of $x \mapsto | x |^2$, we have
$$ \frac{1}{N} \sum_{i=1}^N \bigl| \frac{1}{n}\sum_{j=1}^n y_i^j - x_i \bigr|^2 \leq \frac{1}{nN} \sum_{i=1}^N \sum_{j=1}^n \bigl| y_i^j - x_i \bigr|^2.$$
Combined together we get, for all $(y_i^{j})_{1 \leq i \leq N, 1 \leq j \leq n} \in (\R^d)^{nN}$,
\begin{align*} 
G \Bigl( \frac{1}{N} \sum_{i=1}^N \delta_{ \frac{1}{n} \sum_{j=1}^n y_i^{j} }  \Bigr) &+ \frac{1}{2(T-t)} \frac{1}{N} \sum_{i=1}^N \bigl| \frac{1}{n}\sum_{j=1}^n y_i^j - x_i \bigr|^2  \\
&\leq G \Bigl( \frac{1}{nN} \sum_{i=1}^N \sum_{j=1}^n \delta_{y_i^j} \Bigr) + \frac{1}{2(T-t)}  \frac{1}{nN} \sum_{i=1}^N \sum_{j=1}^n \bigl| y_i^j - x_i \bigr|^2. 
\end{align*}
Since the left-hand side is always greater than $V^N(t, \bx^N)$ we get, taking the infimum over $\by \in (\R^d)^{nN} $, 
$$ V^N(t, \bx^N) \leq V^{nN} (t, \bx^N, \dots, \bx^N).$$
Since $G$ is continuous we have, see \eqref{eq:limsup21Aout},
$$\limsup_{n \rightarrow +\infty} V^{nN}(t,\bx^N, \dots, \bx^N) \leq U(t, m^N_{\bx^N})$$
and we can deduce that, for all $N \geq 1$,
$$ V^N(t, \bx^N) \leq U(t, m^N_{\bx}), \quad \forall t \in [0,T], \forall \bx^N \in (\R^d)^N,$$
which concludes the proof of the proposition.
\end{proof}

We say that the map $G : \mathcal{P}_2(\R^d) \rightarrow \R $ is $L$-convex if, for any atomless probability space $(\Omega, \mathcal{F}, \mathbb{P})$ its lift $\bar{G}$, defined, for all $X \in L^2(\Omega; \R^d)$ by $\bar{G} ( X) = G \bigl( \mathcal{L}(X) \bigr)$, is convex in the usual sense.
\begin{lem}
    Assume that $G : \mathcal{P}_2(\R^d) \rightarrow \R$ is $L$-convex function, then it satisfies the condition \eqref{eq:weakconvexity13/01} for any probability space $(\Omega, \mathcal{F}; \mathbb{P})$.
\label{lem:Jensen}
\end{lem}

\begin{proof}
    For some probability space $(\Omega, \mathcal{F}, \mathbb{P})$, let $X,Y \in L^2(\Omega; \R^d)$ be two random variables. By Corollary 3.22 in \cite{GangboTudorascu} we can find an element of the sub-differential of $\bar{G}$ at $\E \bigl[ X |Y \bigr]$ (with $\mathbb{E}$ denoting the expectation under $\mathbb{P}$) of the form $\xi \bigl( \E \bigl[ X |Y \bigr] \bigr) $ for some measurable map $\xi : \R^d \rightarrow \R^d$ that is square-integrable with respect to $\mathcal{L}( \E \bigl[X |Y \bigr])$. By definition of the sub-differential of a convex function we have
\begin{equation} 
\bar{G} \bigl( X \bigr) - \bar{G} \bigl ( \E \bigl[ X | Y \bigr] \bigr) \geq \E \Bigl[\xi \bigl( \E \bigl[ X |Y \bigr] \bigr) \cdot \bigl( X - \E \bigl[ X | Y \bigr] \bigr) \Bigr]. 
\label{eq:towardJensen}
\end{equation}
However, by properties of the conditional expectation we have
    \begin{align*}
        \E \Bigl[\xi \bigl( \E \bigl[ X |Y \bigr] \bigr) \cdot \bigl( X - \E \bigl[ X | Y \bigr] \bigr) \Bigr] &= \E \Bigl[ \E \Bigl[    \xi \bigl( \E \bigl[ X |Y \bigr] \bigr) \cdot \bigl( X - \E \bigl[ X | Y \bigr] \bigr) \Big| Y \Bigr] \Bigr] \\
        &= \E \Bigl[   \xi \bigl( \E \bigl[ X |Y \bigr] \bigr) \cdot \bigl( \E \bigl[X |Y \bigr] - \E \bigl[ X | Y \bigr] \bigr)  \Bigr] \Bigr] =0.
    \end{align*}
Getting back to \eqref{eq:towardJensen} and recalling the definition of $\bar{G}$ concludes the proof of the statement.
\end{proof}

\bibliographystyle{alpha}

\bibliography{cdjmreferences}

\end{document}